\documentclass[10pt]{amsart}
1
\usepackage[T1]{fontenc}
\usepackage{lmodern}
\usepackage{amsfonts}
\usepackage{amssymb}
\usepackage{eucal}
\usepackage[margin=1.0in]{geometry}
\usepackage{xcolor}
\definecolor{cite}{rgb}{0.50,0.00,1.00}
\definecolor{url}{rgb}{0.00,0.50,0.75}
\definecolor{link}{rgb}{0.00,0.00,0.50}
\usepackage[colorlinks,linkcolor=link,urlcolor=url,citecolor=cite,pagebackref,breaklinks]{hyperref}
\usepackage{mathtools}
\usepackage{colonequals}
\usepackage[all]{xy}
\usepackage[shortlabels]{enumitem}
\usepackage[lite,abbrev,msc-links]{amsrefs}

\numberwithin{equation}{section}

\theoremstyle{plain}
\newtheorem{proposition}{Proposition}[section]
\newtheorem{corollary}[proposition]{Corollary}
\newtheorem{lemma}[proposition]{Lemma}
\newtheorem{theorem}[proposition]{Theorem}

\theoremstyle{definition}
\newtheorem{definition}[proposition]{Definition}
\newtheorem{notation}[proposition]{Notation}
\newtheorem{assumption}[proposition]{Assumption}
\newtheorem{construction}[proposition]{Construction}

\theoremstyle{remark}
\newtheorem{example}[proposition]{Example}
\newtheorem{remark}[proposition]{Remark}

\newcommand{\cC}{\mathcal{C}}
\newcommand{\cD}{\mathcal{D}}
\newcommand{\cE}{\mathcal{E}}
\newcommand{\cF}{\mathcal{F}}
\newcommand{\cG}{\mathcal{G}}
\newcommand{\cH}{\mathcal{H}}
\newcommand{\cI}{\mathcal{I}}
\newcommand{\cK}{\mathcal{K}}

\newcommand{\cM}{\mathcal{M}}
\newcommand{\cN}{\mathcal{N}}
\newcommand{\cQ}{\mathcal{Q}}
\newcommand{\cR}{\mathcal{R}}
\newcommand{\cT}{\mathcal{T}}
\newcommand{\cU}{\mathcal{U}}
\newcommand{\id}{\mathrm{id}}
\newcommand{\Komp}{\c{K}\r{pt}}
\newcommand{\Kpt}{\mathrm{Kpt}}
\newcommand{\Hom}{\r{Hom}}
\newcommand{\Fun}{\mathrm{Fun}}
\newcommand{\Map}{\mathrm{Map}}
\newcommand{\Ob}{\mathrm{Ob}}
\newcommand{\Ar}{\mathrm{Ar}}
\newcommand{\res}{\mathbin{|}}
\newcommand{\rres}{\mathrm{res}}
\newcommand{\Set}{\c{S}\r{et}}
\newcommand{\cart}{\mathrm{cart}}
\newcommand{\diag}{\mathrm{diag}}
\newcommand{\op}{\mathrm{op}}
\newcommand{\sh}{\H}

\newcommand{\xto}[2][]{\xrightarrow[#1]{#2}}

\renewcommand{\b}[1]{\mathbf{#1}}
\renewcommand{\c}[1]{\mathcal{#1}}
\renewcommand{\r}[1]{\mathrm{#1}}

\renewcommand{\H}{\mathrm{h}}
\newcommand{\N}{\r{N}}
\newcommand{\Crt}{\r{Cart}}
\newcommand{\Cart}{\cC\r{art}}
\newcommand{\corr}{\r{corr}}
\newcommand{\Kart}{\c{K}\r{art}}

\newcommand{\del}{\mathbf{\Delta}}

\newcommand{\Sset}[1][]{\Set_{#1\Delta}}
\newcommand{\Mset}{\Set_{\Delta}^+}
\newcommand{\Sch}{\c{S}\r{ch}}
\newcommand{\Cpt}{\b{Cpt}}
\newcommand{\CCpt}{\cC\r{pt}}
\newcommand{\RCpt}{\r{Cpt}}
\newcommand{\cat}{\cC\r{at}_1}
\newcommand{\Cat}{\cC\r{at}_{\infty}}
\newcommand{\rD}{\r{D}}
\newcommand{\RES}{\r{res}}
\newcommand{\RKE}{\r{RKE}}

\newcommand{\sfU}{\mathsf{U}}
\newcommand{\sfV}{\mathsf{V}}
\newcommand{\sfW}{\mathsf{W}}

\begin{document}

\title{Gluing restricted nerves of $\infty$-categories}

\author{Yifeng Liu}
\address{Department of Mathematics, Massachusetts Institute of Technology, Cambridge, MA 02139, United States}
\email{liuyf@math.mit.edu}

\author{Weizhe Zheng}
\address{Morningside Center of Mathematics, Academy of Mathematics and Systems Science, Chinese Academy of Sciences, Beijing 100190, China}
\email{wzheng@math.ac.cn}

\date{June 4, 2015}

\begin{abstract}
In this article, we develop a general technique for gluing subcategories of
$\infty$-categories. We obtain categorical equivalences between simplicial
sets associated to certain multisimplicial sets. Such equivalences can be
used to construct functors in different contexts. One of our results
generalizes Deligne's gluing theory developed in the construction of the
extraordinary pushforward operation in \'etale cohomology of schemes. Our
results are applied in subsequent articles \cites{LZ1,LZ2} to construct
Grothendieck's six operations in \'etale cohomology of Artin stacks.
\end{abstract}

\maketitle

\tableofcontents

\section*{Introduction}

The extraordinary pushforward, one of Grothendieck's six operations, in
\'etale cohomology of schemes was constructed in SGA 4 XVII \cite{SGA4XVII}.
Let $\Sch'$ be the category of quasi-compact and quasi-separated schemes,
with morphisms being separated of finite type, and let $\Lambda$ be a fixed
torsion ring. For a morphism $f\colon Y\to X$ in $\Sch'$, the extraordinary
pushforward by $f$ is a functor
\begin{align*}
f_!\colon \rD(Y,\Lambda)\to\rD(X,\Lambda),
\end{align*}
between unbounded derived categories of $\Lambda$-modules in the \'{e}tale
topoi. The functoriality of this operation is encoded by a pseudofunctor
\begin{align*}
F\colon \Sch'\to\cat
\end{align*}
sending a scheme $X$ in $\Sch'$ to $\rD(X,\Lambda)$ and a morphism $f\colon
Y\to X$ in $\Sch'$ to the functor $f_!$. Here $\cat$ denotes the
$(2,1)$-category of categories\footnote{A $(2,1)$-category is a $2$-category
in which all $2$-cells are invertible.}. There are obvious candidates for the
restrictions $F_P$ and $F_J$ of $F$ to the subcategories $\Sch'_P$ and
$\Sch'_J$ of $\Sch'$ spanned respectively by proper morphisms and open
immersions. The construction of $F$ thus amounts to gluing the two
pseudofunctors. For this, Deligne developed a general theory for gluing two
pseudofunctors of target $\cat$ \cite{SGA4XVII}*{Section 3}. Deligne's gluing
theory, together with its variants (\cite{Ayoub}*{Section 1.3}, \cite{Zh1}),
have found several other applications (\cite{Ayoub}, \cite{CD}, \cite{Zh}).

In the present article, we study the problem of gluing in higher categories.
The technique developed here can be used to construct Grothendieck's six
operations in different contexts (\cites{LZ1,LZ2}, \cite{Robalo}). In
\cites{LZ1,LZ2}, we use the gluing technique to construct higher categorical
six operations in \'etale cohomology of higher Artin stacks and prove the
base change theorem. Even for $1$-Artin stacks and ordinary six operations,
this theorem was previously only established on the level of sheaves (and
subject to other restrictions) \cites{LO1,LO2}. Our construction of the six
operations makes essential use of higher categorical descent, so that even
if one is only interested in the six operations and base change in ordinary
derived categories, the enhanced version is still an indispensable step of
the construction. As a starting point for the descent procedure, we need an
enhancement of the pseudofunctor $F$ above. In the language of
$\infty$-categories \cite{Lu1}, such an enhancement is a functor
\begin{align*}
F^{\infty}\colon \N(\Sch')\to\Cat
\end{align*}
between $\infty$-categories, where $\N(\Sch')$ is the nerve of $\Sch'$ and
$\Cat$ denotes the $\infty$-category of $\infty$-categories. For every
scheme $X$ in $\Sch'$, $F^\infty(X)$ is an $\infty$-category
$\cD(X,\Lambda)$, whose homotopy category is equivalent to $\rD(X,\Lambda)$.
For every morphism $f\colon Y\to X$ in $\Sch'$, the image $F^{\infty}(f)$ is
a functor
\[f_!^{\infty}\colon \cD(Y,\Lambda)\to\cD(X,\Lambda)\]
such that the induced functor $\H f_!^{\infty}$ between homotopy categories
is equivalent to the classical $f_!$.

One major difficulty of the construction of $F^\infty$ is the need to keep
track of coherence of all levels. By Nagata compactification \cite{Conrad},
every morphism $f$ in $\Sch'$ can be factorized as $p\circ j$, where $j$ is
an open immersion and $p$ is proper. One can then define $F(f)$ as
$F_P(p)\circ F_J(j)$. The issue is that such a factorization is not
canonical, so that one needs to include coherence with composition as part
of the data. Since the target of $F$ is a $(2,1)$-category, in Deligne's
theory coherence up to the level of $2$-cells suffices. The target of
$F^{\infty}$ being an $\infty$-category, we need to consider coherence of
\emph{all} levels.

Another complication is the need to deal with more than two subcategories.
This need is already apparent in \cite{Zh}. We will give another illustration
in the proof of Corollary \ref{0co:scheme2} below.

To handle these complications, we propose the following general framework.
Let $\cC$ be an (ordinary) category and let $k\geq 2$ be an integer. Let
$\cE_1,\dots,\cE_k\subseteq\Ar(\cC)$ be $k$ sets of arrows of $\cC$, each
containing every identity morphism in $\cC$. In addition to the nerve
$\N(\cC)$ of $\cC$, we define another simplicial set, which we denote by
$\delta^*_k\N(\cC)^\cart_{\cE_1,\dots,\cE_k}$. Its $n$-simplices are
functors $[n]^k\to\cC$ such that the image of a morphism in the $i$-th
direction is in $\cE_i$ for $1\leq i\leq k$, and the image of every square
in direction $(i,j)$ is a Cartesian square (also called pullback square) for
$1\le i<j\le k$. For example, when $k=2$, the $n$-simplices of
$\delta^*_2\N(\cC)^\cart_{\cE_1,\cE_2}$ correspond to diagrams
\begin{align}\label{0eq:simplex}
\xymatrix{
c_{00} \ar[r]\ar[d] & c_{01} \ar[r]\ar[d] & \dotsb \ar[r] & c_{0n}\ar[d] \\
c_{10} \ar[r]\ar[d] & c_{11} \ar[r]\ar[d] & \dotsb \ar[r] & c_{1n}\ar[d] \\
\vdots \ar[d]       & \vdots \ar[d]       &               & \vdots\ar[d] \\
c_{n0} \ar[r]       & c_{n1} \ar[r]       & \dotsb \ar[r] & c_{nn}
}
\end{align}
where vertical (resp.\ horizontal) arrows are in $\cE_1$ (resp.\ $\cE_2$) and
all squares are Cartesian. The face and degeneracy maps are defined in the
obvious way. Note that $\delta^*_k\N(\cC)^\cart_{\cE_1,\dots,\cE_k}$ is
\emph{seldom} an $\infty$-category. It is the simplicial set associated to a
$k$-simplicial set $\N(\cC)^\cart_{\cE_1,\dots,\cE_k}$. The latter is a
special case of what we call the \emph{restricted multisimplicial nerve} of
an ($\infty$-)category with extra data (Definition
\ref{3de:restricted_nerve}).

Let $\cE_0\subseteq \Ar(\cC)$ be a set of arrows stable under
composition and containing $\cE_1$ and $\cE_2$. Then there is a
natural map
\begin{equation}\label{0eq:equivalence}
g\colon \delta^*_k\N(\cC)^\cart_{\cE_1,\cE_2,\cE_3,\dots,\cE_k}\to
\delta^*_{k-1}\N(\cC)^\cart_{\cE_0,\cE_3,\dots,\cE_k}
\end{equation}
of simplicial sets, sending an $n$-simplex of the source corresponding to a
functor $[n]^k\to\cC$, to its partial diagonal
\[[n]^{k-1}=[n]\times[n]^{k-2}\xrightarrow{\r{diag}\times\r{id}_{[n]^{k-2}}}[n]^k=[n]^2\times[n]^{k-2}\to\cC,\] which
is an $n$-simplex of the target.

We say that a subset $\cE\subseteq\Ar(\cC)$ is \emph{admissible} (Definition
\ref{3de:admissible_edge}) if $\cE$ contains every identity morphism, $\cE$
is stable under pullback, and for every pair of composable morphisms
$p\in\cE$ and $q$ in $\cC$, $p\circ q$ is in $\cE$ if and only if $q\in\cE$.
One main result of the article is the following.

\begin{theorem}[Special case of Theorem \ref{5th:cartesian_descent}]\label{0th:main}
Let $\cC$ be a category admitting pullbacks and let
$\cE_0,\cE_1,\dots,\cE_k\subseteq \Ar(\cC)$, $k\ge 2$, be sets of morphisms
containing every identity morphism and satisfying the following conditions:
\begin{enumerate}
  \item $\cE_1,\cE_2\subseteq \cE_0$; $\cE_0$ is stable under composition
      and $\cE_1,\cE_2$ are admissible.
  \item For every morphism $f$ in $\cE_0$, there exist $p\in\cE_1$ and
      $q\in\cE_2$ such that $f=p\circ q$.
  \item For every $3\le i \le k$, $\cE_i$ is stable under pullback by
      $\cE_1$.
\end{enumerate}
Then the natural map \eqref{0eq:equivalence}
\[g\colon
\delta^*_k\N(\cC)^\cart_{\cE_1,\cE_2,\cE_3,\dots,\cE_k}\to
\delta^*_{k-1}\N(\cC)^\cart_{\cE_0,\cE_3,\dots,\cE_k}\] is a
\emph{categorical equivalence} (Definition
\ref{1de:categorical_equivalence}).
\end{theorem}

Taking $k=2$ and $\cE_0=\Ar(\cC)$ we obtain the following.

\begin{corollary}\label{0co:main}
Let $\cC$ be a category admitting pullbacks. Let $\cE_1,\cE_2\subseteq
\Ar(\cC)$ be admissible subsets. Assume that for every morphism $f$ of
$\cC$, there exist $p\in\cE_1$ and $q\in\cE_2$ such that $f=p\circ q$. Then
the natural map
\[g\colon \delta^*_2\N(\cC)^\cart_{\cE_1,\cE_2}\to \N(\cC)\]
is a categorical equivalence.
\end{corollary}

In the situation of Corollary \ref{0co:main}, for every $\infty$-category
$\cD$, the functor
\[\Fun(\N(\cC),\cD)\to \Fun(\delta^*_2\N(\cC)^\cart_{\cE_1,\cE_2},\cD)\]
is an equivalence of $\infty$-categories. We remark that such equivalences
can be used to construct functors in many different contexts. For instance,
we can take $\cD$ to be $\N(\cat)$ \footnote{Here $\N(\cat)$ denotes the
simplicial nerve \cite{Lu1}*{Definition 1.1.5.5} of $\cat$, the latter
regarded as a simplicial category.}, $\Cat$, or the $\infty$-category of
differential graded categories.

In the above discussion, we may replace $\N(\cC)$ by an $\infty$-category
$\cC$ (not necessarily the nerve of an ordinary category), and define the
simplicial set $\delta_k^*\cC_{\cE_1,\dots,\cE_k}^\cart$. Moreover, in
\cite{LZ1}, we need to encode information such as the Base Change
isomorphism, which involves both pullback and (extraordinary) pushforward.
To this end, we will define in Section \ref{3ss}, for every subset
$L\subseteq \{1,\dots,k\}$, a variant
$\delta_{k,L}^*\cC^\cart_{\cE_1,\dots,\cE_k}$ of
$\delta_k^*\cC_{\cE_1,\dots,\cE_k}^\cart$ by ``taking the opposite'' in the
directions in $L$. For $L\subseteq \{3,\dots, k\}$, the theorem remains
valid modulo slight modifications. We refer the reader to Theorem
\ref{5th:cartesian_descent} for a precise statement. Let us mention in
passing that there exists  a canonical categorical equivalence from the
simplicial set $\delta^*_{2,\{2\}}\cC^\cart_{\cE_1,\cE_2}$ to the
$\infty$-category of correspondences introduced in \cite{Gait} (Example
\ref{4ex:corr}).

Next we turn to applications to categories of schemes.

\begin{corollary}\label{0co:scheme1}
Let $P\subseteq \Ar(\Sch')$ be the subset of proper morphisms and let
$J\subseteq \Ar(\Sch')$ be the subset of open immersions. Then the natural
map
\[\delta_2^*\N(\Sch')^\cart_{P,J}\to\N(\Sch')\]
is a categorical equivalence.
\end{corollary}

\begin{proof}
This follows immediately from Corollary \ref{0co:main} applied to
$\cC=\Sch'$, $\cE_1=P$, $\cE_2=J$.
\end{proof}

As many important moduli stacks are not quasi-compact, in \cite{LZ1} we work
with Artin stacks that are not necessarily quasi-compact. Accordingly, we
need the following variant of Corollary \ref{0co:scheme1}.

\begin{corollary}\label{0co:scheme2}
Let $\Sch''$ be the category of disjoint unions of quasi-compact and
quasi-separated schemes, with morphisms being separated and \emph{locally}
of finite type. Let $F=\Ar(\Sch'')$ be the set of morphisms of $\Sch''$. Let
$P\subseteq F$ be the subset of proper morphisms, and let $I\subseteq F$ be
the subset of local isomorphisms \cite{EGAIn}*{D\'efinition 4.4.2}. Then the
natural map
\[\delta_2^*\N(\Sch'')^\cart_{P,I}\to\N(\Sch'')\]
is a categorical equivalence.
\end{corollary}

Corollary \ref{0co:scheme2} still holds if one replaces $I$ by the subset
$E\subseteq F$ of \'{e}tale morphisms.

One might be tempted to apply Corollary \ref{0co:main} by taking $\cE_1=P$,
$\cE_2=I$. However, the assumption of Corollary \ref{0co:main} does not
hold. For example, we may take $f$ to be the structural morphism of the
disjoint union of varieties of unbounded dimensions over a field.

\begin{proof}[Proof of Corollary \ref{0co:scheme2}]
Put $\cC=\Sch''$. We introduce the following auxiliary sets of morphisms. Let
$F_{\r{ft}}\subseteq F$ be the set of separated morphisms of finite type, and
let $I_{\r{ft}}=I\cap F_{\r{ft}}$. Consider the following commutative diagram
\[\xymatrix{
\delta_3^*\N(\cC)^\cart_{P,I_{\r{ft}},I} \ar[r] \ar[d] &  \delta_2^*\N(\cC)^\cart_{F_{\r{ft}},I} \ar[d] \\
\delta_2^*\N(\cC)^\cart_{P,I} \ar[r]  & \N(\cC), }
\]
where the upper arrow is induced by ``composing morphisms in $P$ and
$I_{\r{ft}}$'', while the left arrow is induced by ``composing morphisms in
$I_{\r{ft}}$ and $I$''. We will apply Theorem \ref{0th:main} to all arrows
in the diagram, except the lower one, to show that they are categorical
equivalences. It then follows that the lower arrow is also a categorical
equivalence.

For the upper arrow, we apply Theorem \ref{0th:main} to $k=3$,
$\cE_0=F_{\r{ft}}$, $\cE_1=P$, $\cE_2=I_{\r{ft}}$, $\cE_3=I$. Conditions (1)
and (3) are obviously satisfied. For Condition (2), note that every morphism
$f$ in $F_{\r{ft}}$ can be written as a disjoint union $\coprod f_i$ of
morphisms $f_i$ of $\Sch'$. It then suffices to apply Nagata
compactification to each $f_i$. For the left arrow, we apply Theorem
\ref{0th:main} to $k=3$, $\cE_0=\cE_1=I$, $\cE_2=I_{\r{ft}}$, $\cE_3=P$. All
the conditions are obviously satisfied. Finally, we apply Corollary
\ref{0co:main} to ($k=2$, $\cE_0=F$,) $\cE_1=I$, $\cE_2=F_{\r{ft}}$. To
verify the assumption of Corollary \ref{0co:main}, let $f$ be a morphism of
$\Sch''$. Then $f$ has the form $\coprod_{i,j} X_{ij}\to \coprod_i Y_i$ and
is induced by morphisms $X_{ij}\to Y_i$, where $X_{ij}$ and $Y_i$ are
quasi-compact and quasi-separated schemes. Then $f$ is the composition
$\coprod_{i,j} X_{ij}\xto{q}\coprod_{i,j}Y_i\xto{p} \coprod_i Y_i$ with
$p\in I$ and $q\in F_{\r{ft}}$.
\end{proof}

The proof of Theorem \ref{0th:main} consists of two steps. Let us illustrate
them in the case of Corollary \ref{0co:main}. The map $g$ can be decomposed
as
\[\delta_2^*\N(\cC)_{\cE_1,\cE_2}^\cart\xrightarrow{g'}\delta_2^*\N(\cC)_{\cE_1,\cE_2}
\xrightarrow{g''}\N(\cC),
\]
where $\delta_2^*\N(\cC)_{\cE_1,\cE_2}$ is the simplicial set whose
$n$-simplices are diagrams \eqref{0eq:simplex} without the requirement that
every square is Cartesian, $g'$ is the natural inclusion and $g''$ is the map
remembering the diagonal. We prove that both $g'$ and $g''$ are categorical
equivalences. The fact that $g''$ is a categorical equivalence is an
$\infty$-categorical generalization of Deligne's result
\cite{SGA4XVII}*{Proposition 3.3.2}.

The article is organized as follows. In Section \ref{1ss}, we collect some
basic definitions and facts in the theory of $\infty$-categories \cite{Lu1}
for the reader's convenience. In Section \ref{2ss}, we develop a general
technique for constructing functors to $\infty$-categories. The method will
be used several times in this article and its sequels \cites{LZ1,LZ2}. In
Section \ref{3ss}, we introduce several notions related to multisimplicial
sets used in the statements of our main results. In particular, we define the
restricted multisimplicial nerve of an $\infty$-category with extra data. In
Section \ref{4ss}, we prove a multisimplicial descent theorem, which implies
that the map $g''$ is a categorical equivalence.  In Section \ref{5ss}, we
prove a Cartesian gluing theorem, which implies that the inclusion $g'$ is a
categorical equivalence. A Cartesian gluing formalism for pseudofunctors
between $2$-categories was developed in \cite{Zh1}. Our treatment here is
quite different and more adapted to the higher categorical context. In
Section \ref{6ss}, we prove some facts about inclusions of simplicial sets
used in the previous sections.

\subsection*{Conventions}

Unless otherwise specified, a category is to be understood as an ordinary
category. We will not distinguish between sets and categories in which the
only morphisms are identity morphisms. For categories $\cC$ and $\cD$, we let
$\Fun(\cC,\cD)$ denote the \emph{category of functors} from $\cC$ to $\cD$,
whose objects are functors and whose morphisms are natural transformations.
We let $\id$ denote various identity morphisms.

Throughout this article, an effort has been made to keep our notation
consistent with those in \cite{Lu1}.

\subsection*{Acknowledgments}

We thank Ofer Gabber, Luc Illusie, Aise Johan de~Jong, Shenghao Sun, and
Xinwen Zhu for useful discussions during the preparation of this article and
its sequels. We thank the referee for a careful reading of the manuscript
and for many useful comments, especially his suggestions on Lemma
\ref{1le:component}. Part of this work was done during a visit of the first
author to the Morningside Center of Mathematics, Chinese Academy of
Sciences, in Beijing. He thanks the Center for its hospitality. The first
author was partially supported by NSF grant DMS--1302000. The second author
was partially supported by China's Recruitment Program of Global Experts;
National Natural Science Foundation of China Grant 11321101; National Center
for Mathematics and Interdisciplinary Sciences and Hua Loo-Keng Key
Laboratory of Mathematics, Chinese Academy of Sciences.

\section{Simplicial sets and $\infty$-categories}
\label{1ss}

In this section, we collect some basic definitions and facts in the theory
of $\infty$-categories developed by Joyal \cites{Joyal0,Joyal} (who calls
them ``quasi-categories'') and Lurie \cite{Lu1}. For a more systematic
introduction to Lurie's theory, we recommend \cite{Groth}.

For $n\ge 0$, we let $[n]$ denote the totally ordered set $\{0,\dots,n\}$
and we put $[-1]=\emptyset$. We let $\del$ denote the \emph{category of
combinatorial simplices}, whose objects are the totally ordered sets $[n]$
for $n\geq0$ and whose morphisms are given by (non-strictly)
order-preserving maps. For $n\geq 0$ and $0\leq k\leq n$, the face map
$d^n_k\colon [n-1]\to[n]$ is the unique injective order-preserving map such
that $k$ is not in the image; and the degeneracy map $s^n_k\colon [n+1]\to
[n]$ is the unique surjective order-preserving map such that
$(s^n_k)^{-1}(k)$ has two elements.

\begin{definition}[Simplicial set and $\infty$-category]
We let $\Set$ denote the category of sets\footnote{More rigorously, $\Set$
is the category of sets in a universe that we fix once and for all.}.
\begin{itemize}
  \item We define the category of \emph{simplicial sets}, denoted by
      $\Sset$, to be the functor category $\Fun(\del^{op},\Set)$. For a
      simplicial set $S$, we denote by $S_n=S([n])$ its set of
      $n$-simplices.

  \item For $n\geq0$, we denote by $\Delta^n=\Fun(-,[n])$ the simplicial
      set represented by $[n]$. We let $\partial \Delta^n\subseteq
      \Delta^n$ denote the simplicial subset obtained by removing the
      interior, namely the $n$-simplex defined by $\id_{[n]}\colon [n]\to
      [n]$. In particular, $\partial\Delta^0=\emptyset$. For each $0\leq
      k\leq n$, we define the \emph{$k$-th horn}
      $\Lambda^n_k\subseteq\partial\Delta^n$ to be the simplicial subset
      obtained by removing the face opposite to the $k$-th vertex, namely
      the $(n-1)$-simplex defined by $d^n_k\colon [n-1]\to [n]$.

  \item An \emph{$\infty$-category} (resp.\ Kan complex) is a simplicial
      set $\cC$ such that $\cC\to \Delta^0$ has the right lifting property
      with respect to all inclusions $\Lambda^n_k\subseteq\Delta^n$ with
      $0<k<n$ (resp.\ $0\le k\le n$). In other words, a simplicial set
      $\cC$ is an $\infty$-category (resp.\ Kan complex) if and only if
      every map $\Lambda^n_k\to\cC$ with $0<k<n$ (resp.\ $0\le k\le n$)
      can be extended to a map $\Delta^n\to\cC$.
\end{itemize}
\end{definition}

Note that a Kan complex is an $\infty$-category. The lifting property in the
definition of $\infty$-category was first introduced (under the name of
``restricted Kan condition'') by Boardman and Vogt \cite{BV}*{Definition
IV.4.8}.

The lifting property defining $\infty$-category (resp.\ Kan complex) can be
adapted to the relative case. More precisely, a map $f\colon T\to S$ of
simplicial sets is called an \emph{inner fibration} (resp.\ \emph{Kan
fibration}) if it has the right lifting property with respect to all
inclusions $\Lambda^n_k\subseteq\Delta^n$ with $0<k<n$ (resp.\ $0\leq k\leq
n$). A map $i\colon A\to B$ of simplicial sets is said to be \emph{inner
anodyne} (resp.\ \emph{anodyne}) if it has the left lifting property with
respect to all inner fibrations (resp.\ Kan fibrations).

\begin{example}[Nerve of an ordinary category]
Let $\cC$ be an ordinary category. The \emph{nerve} $\N(\cC)$ of $\cC$ is
the simplicial set given by $\N(\cC)_n=\Fun([n],\cC)$. It is easy to see
that $\N(\cC)$ is an $\infty$-category and we can identify $\N(\cC)_0$ and
$\N(\cC)_1$ with the set of objects $\Ob(\cC)$ and the set of arrows
$\Ar(\cC)$, respectively.
\end{example}

Conversely, given a simplicial set $S$, one constructs an ordinary category
$\H S$, the \emph{homotopy category of} $S$ (\cite{Lu1}*{Definition
1.1.5.14}, ignoring the enrichment) such that $\Ob(\H S)=S_0$. For an
$\infty$-category $\cC$, $\Hom_{\H\cC}(x,y)$ consists of homotopy classes of
edges $x\to y$ in $\cC_1$ \cite{Lu1}*{Proposition 1.2.3.9}. By
\cite{Lu1}*{Proposition 1.2.3.1}, $\H$ is a left adjoint to the nerve
functor $\N$.

\begin{definition}[Object, morphism, equivalence]
Let $\cC$ be an $\infty$-category. Vertices of $\cC$ are called
\emph{objects} of $\cC$ and edges of $\cC$ are called \emph{morphisms} of
$\cC$. A morphism of $\cC$ is called an \emph{equivalence} if it defines an
isomorphism in the homotopy category $\H \cC$.
\end{definition}

The category $\Sset$ is Cartesian-closed. For objects $S$ and $T$ of
$\Sset$, we let $\Map(S,T)$ denote the internal mapping object defined by
$\Hom_{\Sset}(K,\Map(S,T))\simeq \Hom_{\Sset}(K\times S,T)$. If $\cC$ is an
$\infty$-category, we write $\Fun(S,\cC)$ instead of $\Map(S,\cC)$. One can
show that $\Fun(S,\cC)$ is an $\infty$-category \cite{Lu1}*{Proposition
1.2.7.3 (1)} (see also \cite{Lu1}*{Corollary 2.3.2.5}).

\begin{definition}[Functor, natural transformation, natural equivalence]
Objects of $\Fun(S,\cC)$ are called \emph{functors} $S\to \cC$, morphisms of
$\Fun(S,\cC)$ are called \emph{natural transformations}, and equivalences in
$\Fun(S,\cC)$ are called \emph{natural equivalences}.
\end{definition}

\begin{remark}
Let $f,g\colon S\to\cC$ be functors and $\phi\colon f\to g$ a natural
transformation. Then $\phi$ is a natural equivalence if and only if for every
vertex $s$ of $S$, the morphism $\phi(s)\colon f(s)\to g(s)$ is an
equivalence in $\cC$. We refer the reader to \cite{Lu1}*{Proposition 3.1.2.1}
for a generalization (see \cite{Lu1}*{Remark 2.4.1.4}).
\end{remark}

\begin{remark}\label{1re:homotopy}
Let $\cC$ be an $\infty$-category and let $f,g\colon x\to y$ be morphisms of
$\cC$. Then $f$ and $g$ are homotopic (namely, having the same image in $\H
\cC$) if and only if they are equivalent when viewed as objects of the
$\infty$-category defined by the fiber of the map $\Fun(\Delta^1,\cC)\to
\Fun(\partial \Delta^1,\cC)$. Indeed, the latter condition means that there
exist a morphism $h\colon x\to y$ and two $2$-simplices of $\cC$ as shown in
the diagram
\[\xymatrix{x\ar[d]_{\id_x}\ar[r]^f\ar[rd]^h & y\ar[d]^{\id_y}\\x\ar[r]^g&y.}\]
By definition (resp.\ \cite{Lu1}*{Remark 1.2.3.6}), the existence of the
$2$-simplex in the upper right (resp.\ lower left) corner means that $f$
(resp.\ $g$) and $h$ are homotopic. This proves the ``if'' part. For the
``only if'' part, it suffices to take $h=g$ and to take the $2$-simplex in
the lower left corner to be degenerate.
\end{remark}

We now recall the notion of categorical equivalence of simplicial sets,
which is essential to our article. There are several equivalent definitions
of categorical equivalence. The one given below (equivalent to
\cite{Lu1}*{Definition 1.1.5.14} in view of \cite{Lu1}*{Proposition
2.2.5.8}), due to Joyal \cite{Joyal}, will be used in the proofs of our
theorems.

\begin{definition}[Categorical equivalence]\label{1de:categorical_equivalence}
A map $f\colon T\to S$ of simplicial sets is a \emph{categorical
equivalence} if for every $\infty$-category $\cC$, the induced
functor
\[\H\Fun(S,\cC)\to\H\Fun(T,\cC)\]
is an equivalence of ordinary categories.
\end{definition}

If $f\colon T\to S$ is a categorical equivalence, then the induced functor
$\H T\to \H S$ is an equivalence of ordinary categories. An inner anodyne
map is a categorical equivalence \cite{Lu1}*{Lemma 2.2.5.2}. The category
$\Sset$ admits the Joyal model structure \cite{Lu1}*{Theorem 2.2.5.1}, for
which weak equivalences are precisely categorical equivalences.

\begin{remark}
Let $\cC$ and $\cD$ be $\infty$-categories. A functor $f\colon \cC\to \cD$
is a categorical equivalence if and only if there exist a functor $g\colon
\cD\to \cC$ and natural equivalences between $f\circ g$ and $\id_\cD$ and
between $g\circ f$ and $\id_{\cC}$. Indeed, the ``only if'' direction
follows from \cite{Lu1}*{Proposition 1.2.7.3} and the other direction is
clear.
\end{remark}

The following criterion of categorical equivalence will be used in the
proofs of our theorems. Given maps of simplicial sets $v,v'\colon Y\to X$
and an inner fibration $p\colon X\to S$ such that $p\circ v=p\circ v'$, we
say that $v$ and $v'$ are \emph{homotopic over $S$} if they are equivalent
when viewed as objects of the $\infty$-category defined by the fiber of the
inner fibration $\Map(Z,X)\to \Map(Z,S)$ induced by $p$.

\begin{lemma}\label{1le:categorical_equivalence}
A map of simplicial sets $f\colon Y\to Z$ is a categorical equivalence if
and only if the following conditions are satisfied for every
$\infty$-category $\cD$:
\begin{enumerate}
\item For every $l=0,1$ and every commutative diagram
  \[\xymatrix{Y\ar[r]^-v\ar[d]_f & \Fun(\Delta^l,\cD)\ar[d]^p\\
  Z\ar[r]^-w & \Fun(\partial \Delta^l,\cD)}\]
where $p$ is induced by the inclusion $\partial \Delta^l\subseteq
\Delta^l$, there exists a map $u\colon Z\to \Fun(\Delta^l,\cD)$ satisfying
$p\circ u=w$ such that $u\circ f$ and $v$ are homotopic over
$\Fun(\partial \Delta^l,\cD)$.

\item For $l=2$ and every commutative diagram as above, there exists a map
    $u\colon Z\to \Fun(\Delta^l,\cD)$ satisfying $p\circ u=w$.
\end{enumerate}
\end{lemma}

\begin{proof}
By definition, that $f$ is a categorical equivalence means that for every
$\infty$-category $\cD$, the functor
\[
F\colon \sh \Fun(Z,\cD)\to \sh\Fun(Y,\cD)
\]
induced by $f$ is an equivalence of categories. We show that the conditions
for $l=0,1,2$ mean that $F$ is essentially surjective, full, and faithful,
respectively. For $l=0$, this is clear. For $l=1$, this follows from Remark
\ref{1re:homotopy}. For $l=2$, the condition means that for functors
$g_0,g_1,g_2\colon Z\to \cD$, and natural transformations $\phi\colon g_0\to
g_1$, $\psi\colon g_1\to g_2$, $\chi\colon g_0\to g_2$ such that
$F([\psi]\circ [\phi])=F([\chi])$, we have $[\psi]\circ [\phi]=[\chi]$. Here
$[\phi]$, $[\psi]$, $[\chi]$ denote the homotopy classes of $\phi$, $\psi$,
$\chi$, respectively. The condition is clearly satisfied if $F$ is faithful.
Conversely, if $F$ is faithful, it suffices to take $g_1=g_2$ and
$\psi=\id$.
\end{proof}

In Section \ref{3ss}, we will introduce the notion of multi-marked
simplicial sets, which generalizes the notion of marked simplicial sets in
\cite{Lu1}*{Definition 3.1.0.1}. Since marked simplicial sets play an
important role in many arguments for $\infty$-categories, we recall its
definition.

\begin{definition}[Marked simplicial set]
A \emph{marked simplicial set} is a pair $(X,\cE)$ where $X$ is a simplicial
set and $\cE\subseteq X_1$ is a subset containing all degenerate edges. A
morphism $f\colon (X,\cE)\to (X',\cE')$ of marked simplicial sets is a map
$f\colon X\to X'$ of simplicial sets such that $f(\cE)\subseteq\cE'$. The
category of marked simplicial set will be denoted by $\Mset$.
\end{definition}

The forgetful functor $F \colon \Mset\to \Sset$ carrying $(X,\cE)$ to $X$
admits a right adjoint carrying a simplicial set $S$ to $S^{\sharp}=(S,S_1)$
and a left adjoint carrying $S$ to $S^{\flat}=(S,\cE)$, where $\cE$ is the
set of all degenerate edges. For an $\infty$-category $\cC$, we let
$\cC^\natural$ denote the marked simplicial set $(\cC,\cE)$, where $\cE$ is
the set of all edges of $\cC$ that are equivalences. The category $\Mset$ is
equipped with the \emph{Cartesian model structure} \cite{Lu1}*{Proposition
3.1.3.7}. The adjoint pair $((-)^\flat, F)$ is a Quillen equivalence between
the Joyal model structure on $\Sset$ and the Cartesian model structure on
$\Mset$ \cite{Lu1}*{Theorem 3.1.5.1}.

The category $\Mset$ is Cartesian-closed. For objects $X$ and $Y$ of
$\Mset$, we let $\Map^\flat(X,Y)$ denote the underlying simplicial set of
the internal mapping object $Y^X$. We let $\Map^\sharp(X,Y)\subseteq
\Map^\flat(X,Y)$ denote the largest simplicial subset such that
$\Map^\sharp(X,Y)^\sharp\subseteq Y^X$. If $\cC$ is an $\infty$-category,
then $\Map^\flat(X,\cC^\natural)$ is an $\infty$-category and
$\Map^\sharp(X,\cC^\natural)$ is the largest Kan complex
\cite{Lu1}*{Proposition 1.2.5.3} contained in $\Map^\flat(X,\cC^\natural)$
\cite{Lu1}*{Remark 3.1.3.1} (see also \cite{Lu1}*{Lemma 3.1.3.6}), so that
$(\cC^\natural)^X=\Map^\flat(X,\cC^\natural)^\natural$.

\section{Constructing functors via the category of simplices}
\label{2ss}

In this section, we develop a general technique for constructing functors to
$\infty$-categories, which is the key to several constructions in this
article and its sequels. For a functor $F\colon K\to \cC$ from a simplicial
set $K$ to an $\infty$-category $\cC$, the image $F(\sigma)$ of a simplex
$\sigma$ of $K$ is a simplex of $\cC$, functorial in $\sigma$. Here we
address the problem of constructing $F$ when, instead of having a canonical
choice for $F(\sigma)$, one has a weakly contractible simplicial set
$\cN(\sigma)$ of candidates for $F(\sigma)$.

We start with some generalities on diagrams of simplicial sets. Let $\cI$ be
a (small) ordinary category. We consider the injective model structure on the
functor category $(\Sset)^\cI\colonequals\Fun(\cI,\Sset)$. We say that a
morphism $i\colon \cN\to \cM$ in $(\Sset)^\cI$ is \emph{anodyne} if
$i(\sigma)\colon \cN(\sigma)\to \cM(\sigma)$ is anodyne for every object
$\sigma$ of $\cI$. We say that a morphism $\cR\to \cR'$ in $(\Sset)^\cI$ is
an \emph{injective fibration} if it has the right lifting property with
respect to every anodyne morphism $\cN\to \cM$ in $(\Sset)^\cI$. We say that
an object $\cR$ of $(\Sset)^\cI$ is \emph{injectively fibrant} if the
morphism from $\cR$ to the final object $\Delta^0_\cI$ is an injective
fibration. The right adjoint of the diagonal functor $\Sset\to (\Sset)^\cI$
is the global section functor
\[\Gamma\colon (\Sset)^\cI\to \Sset, \quad
\Gamma(\cN)_q=\r{Hom}_{(\Sset)^\cI}(\Delta^q_\cI,\cN),
\]
where $\Delta^q_\cI\colon \cI \to \Sset$ is the constant functor of value
$\Delta^q$.

\begin{notation}
Let $\Phi\colon \cN\to \cR$ be a morphism of $(\Sset)^\cI$. We let
$\Gamma_\Phi(\cR)\subseteq \Gamma(\cR)$ denote the simplicial subset, union
of the images of $\Gamma(\Psi)\colon \Gamma(\cM)\to \Gamma(\cR)$ for all
factorizations
\[\cN\xrightarrow{i}\cM\xrightarrow{\Psi} \cR\]
of $\Phi$ such that $i$ is anodyne.
\end{notation}

\begin{remark}\label{2re:components}
As the referee pointed out, $\Gamma_\Phi(\cR)$ can be computed using one
single factorization
\[\cN\xto{i'}\cM'\xto{\Psi'}\cR\]
of $\Phi$, where $i'$ is anodyne and $\Psi'$ is an injective fibration. For
every factorization $\cN\xrightarrow{i}\cM\xrightarrow{\Psi} \cR$ of $\Phi$
such that $i$ is anodyne, there exists a dotted arrow rendering the diagram
\[\xymatrix{\cN\ar[r]^{i'}\ar[d]_i & \cM'\ar[d]^{\Psi'}\\
\cM\ar[r]^{\Psi}\ar@{..>}[ru] & \cR}
\]
commutative. Thus $\Gamma_\Phi(\cR)$ is simply the image of $\Gamma(\Psi')$.
Since $\Gamma(\Psi')$ is a Kan fibration, $\Gamma_\Phi(\cR)$ is a union of
connected components of $\Gamma(\cR)$. Indeed, the inclusion
$\Gamma_\Phi(\cR)\subseteq \Gamma(\cR)$ satisfies the right lifting property
with respect to the inclusion $\Delta^{\{j\}}\subseteq\Delta^n$ for all
$0\le j\le n$.
\end{remark}

By definition, the map $\Gamma(\Phi)\colon \Gamma(\cN)\to \Gamma(\cR)$
factorizes through $\Gamma_\Phi(\cR)$. The construction of
$\Gamma_\Phi(\cR)$ enjoys the following functoriality.

\begin{remark}
For a commutative diagram
\[\xymatrix{\cN\ar[r]^\Phi\ar[d]_G & \cR\ar[d]^F \\ \cN' \ar[r]^{\Phi'} & \cR'}\]
in $(\Sset)^\cI$, the map $\Gamma(F)\colon \Gamma(\cR)\to \Gamma(\cR')$
carries $\Gamma_\Phi(\cR)$ into $\Gamma_{\Phi'}(\cR')$. Indeed, for every
factorization $\cN\xrightarrow{i}\cM\xrightarrow{\Psi} \cR$ of $\Phi$ such
that $i$ is anodyne, we have a commutative diagram
\[\xymatrix{\cN\ar[r]^i\ar[d]_G & \cM\ar[r]^\Psi\ar[d] & \cR\ar[d]^F\\
\cN'\ar[r]^{i'}&\cM'\ar[r]^{\Psi'} & \cR',}
\]
where $i'$ is the pushout of $i$ by $G$, hence is anodyne.

For a functor $g\colon \cI'\to \cI$, composition with $g$ induces a functor
$g^*\colon (\Sset)^\cI\to (\Sset)^{\cI'}$. By a slight abuse of notation, we
still denote by $g^*\colon \Gamma(\cR)\to \Gamma(g^*\cR)$ the pullback map
induced by the functor $g^*$. Then the pullback map $g^*$ carries
$\Gamma_\Phi(\cR)$ into $\Gamma_{g^*\Phi}(g^*\cR)$. Indeed, for every
factorization $\cN\xrightarrow{i}\cM\xrightarrow{\Psi} \cR$ of $\Phi$ such
that $i$ is anodyne, $g^*\cN\xrightarrow{g^*i}g^*\cM\xrightarrow{g^*\Psi}
g^*\cR$ is a factorization of $g^*\Phi$ such that $g^* i$ is anodyne, and we
have the following commutative diagram
\[\xymatrix{\Gamma(\cM)\ar[d]_{g^*}\ar[rr]^{\Gamma(\Psi)}&&\Gamma(\cR)\ar[d]^{g^*}\\
\Gamma(g^*\cM)\ar[rr]^{\Gamma(g^*\Psi)}&& \Gamma(g^*\cR).}
\]
\end{remark}

Our construction technique relies on the following property of
$\Gamma_\Phi(\cR)$. In a previous draft of this article, the statement of
part (1) was incorrect. We thank the referee for suggesting the following
correction.

\begin{lemma}\label{1le:component}
Let $\cI$ be a category. Let $\cN$, $\cR$ be objects of $(\Sset)^\cI$ such
that $\cN(\sigma)$ is weakly contractible for all objects $\sigma$ of $\cI$
and $\cR$ is injectively fibrant.
\begin{enumerate}
\item For every morphism $\Phi\colon \cN\to \cR$, the simplicial set
    $\Gamma_\Phi(\cR)$ is nonempty and connected, hence a connected
    component of $\Gamma(\cR)$.

\item For homotopic morphisms $\Phi,\Phi'\colon \cN\to \cR$, we have
    $\Gamma_\Phi(\cR)=\Gamma_{\Phi'}(\cR)$.
\end{enumerate}
\end{lemma}

The condition in (2) means that there exists a morphism $H\colon
\Delta^1_\cI\times \cN\to \cR$ such that $H\res \Delta^{\{0\}}_\cI\times
\cN=\Phi$ and $H\res \Delta^{\{1\}}_\cI\times \cN=\Phi'$. Note that
$\Gamma(\cR)$ is a Kan complex.

\begin{proof}
(1) We apply Remark \ref{2re:components}. Since the morphism $\cM'\to
\Delta^0_\cI$ is a trivial fibration, $\Gamma(\cM')$ is a contractible Kan
complex. Therefore its image $\Gamma_\Phi(\cR)$ is nonempty and connected,
hence a connected component of $\Gamma(\cR)$.

(2) We define an object $\cN^\triangleright$ by
$\cN^\triangleright(\sigma)=\cN(\sigma)^\triangleright$ \cite{Lu1}*{Notation
1.2.8.4}. Since the inclusion $\Delta^1_\cI\times \cN\hookrightarrow
\Delta^1_\cI\times \cN^\triangleright$ is anodyne and $\cR$ is injectively
fibrant, we can find a morphism $H'$ as shown in the diagram
\[
\xymatrix{\Delta_\cI^1\times \cN\ar[r]^-{H}\ar@{^(->}[d]& \cR\\
\Delta_\cI^1\times\cN^{\triangleright}\ar@{..>}[ur]_{H'} }
\]
rendering the diagram commutative. We denote by $h\colon \Delta_\cI^1\to
\cR$ the restriction of $H'$ to the cone point of $\cN^{\triangleright}$,
corresponding to an edge of $\Gamma(\cR)$. Then $h(0)$ belongs to
$\Gamma_\Phi(\cR)$ and $h(1)$ belongs to $\Gamma_{\Phi'}(\cR)$. Since
$\Gamma_\Phi(\cR)$ and $\Gamma_{\Phi'}(\cR)$ are connected components of
$\Gamma(\cR)$ by (1), we have $\Gamma_\Phi(\cR)=\Gamma_{\Phi'}(\cR)$.
\end{proof}

Let $K$ be a simplicial set. The \emph{category of simplices of} $K$, which
we denote by $\del_{/K}$ following \cite{Lu1}*{Notation 6.1.2.5}, plays a key
role in our construction technique. Recall that $\del_{/K}$ is the strict
fiber product $\del\times_{\Sset}(\Sset)_{/K}$. An object of $\del_{/K}$ is a
pair $(n,\sigma)$, where $n\geq0$ is some integer and $\sigma\in
\Hom_{\Sset}(\Delta^n,K)$. A morphism $(n,\sigma)\to (n',\sigma')$ is a map
$d\colon \Delta^n\to \Delta^{n'}$ such that $\sigma=\sigma'\circ d$. Note
that $d$ is a monomorphism (resp.\ epimorphism) if and only if the underlying
map $[n]\to [n']$ is injective (resp.\ surjective). Every epimorphism of
$\del_{/K}$ is split. Moreover, $\del_{/K}$ admits pushouts of epimorphisms
by epimorphisms. In what follows, we sometimes simply write $\sigma$ for an
object of $\del_{/K}$ if $n$ is insensitive.

The usefulness of $\del_{/K}$ is demonstrated by the following lemma.

\begin{lemma}[\cite{Hovey}*{Lemma 3.1.3}]\label{2le:Hovey}
The maps $\sigma\colon \Delta^n\to K$ exhibit $K$ as the colimit of the
functor $\del_{/K}\to\Sset$ carrying $(n,\sigma)$ to $\Delta^n$.
\end{lemma}

\begin{proof}
We include a proof for completeness. Let $X$ be the colimit. Given $m\geq
0$, the set $X_m$ is the colimit of the functor $F_m\colon \del_{/K}\to
\Set$ carrying $(n,\sigma)$ to $(\Delta^n)_m\simeq
\Hom_{\Sset}(\Delta^m,\Delta^n)$. We denote by $\del_{/K}^{[m]}$ the
category of elements of $F_m$. Objects of $F_m$ are triples
\[(n,\sigma,\tau):\xymatrix{
\Delta^m \ar[r]^-{\tau} & \Delta^n \ar[r]^-{\sigma} & K, }
\]
and morphisms $(n,\sigma,\tau)\to(n',\sigma',\tau')$ are commutative
diagrams
\[\xymatrix{
\Delta^m \ar[r]^-{\tau}\ar@{=}[d] & \Delta^n\ar[r]^{\sigma}\ar[d]^-{d} & K \ar@{=}[d] \\
\Delta^m \ar[r]^-{\tau'} & \Delta^{n'}\ar[r]^{\sigma'}  & K. }
\]
Note that $\del_{/K}^{[m]}$ is a disjoint union of categories indexed by
$\rho=\sigma\tau\in K_m$, each admitting an initial object
\[(m,\rho,\id_{\Delta^m})\colon \Delta^m\xto{\id} \Delta^m\xto{\rho} K.\]
The lemma then follows from the fact that the colimit of any functor $F\colon
\cC\to \Set$ from a category $\cC$ to $\Set$ can be identified with the set
of connected component of the category of elements of $F$.
\end{proof}

\begin{notation}
We define a functor $\Map[K,-]\colon \Mset\to (\Sset)^{(\del_{/K})^{op}}$ as
follows. For a marked simplicial set $M$, we define $\Map[K,M]$  by
\[\Map[K,M](n,\sigma)=
\Map^\sharp((\Delta^n)^\flat,M),
\]
for every object $(n,\sigma)$ of $\del_{/K}$. A morphism $d\colon
(n,\sigma)\to (n',\sigma')$ in $\del_{/K}$ goes to the natural restriction
map $\r{Res}^d\colon \Map^\sharp((\Delta^{n'})^\flat,M)\to
\Map^\sharp((\Delta^n)^\flat,M)$. For an $\infty$-category $\cC$, we set
$\Map[K,\cC]=\Map[K,\cC^\natural]$.
\end{notation}

The following remark shows how $\Map[K,-]$ is related with the problem of
constructing functors.

\begin{remark}\label{2re:functors}
The map
\[\Map^{\sharp}(K^\flat,M)\to \Gamma(\Map[K,M])\]
induced by the restriction maps $\Map^{\sharp}(K^\flat,M)\to
\Map^{\sharp}((\Delta^n)^\flat,M)$ is an isomorphism of simplicial sets.
Indeed, the set of $m$-simplices of $\Map^{\sharp}(K^\flat,M)$ can be
identified with $\Hom_{\Mset}((\Delta^m)^\sharp\times K^\flat,M)$, while the
set of $m$-simplices of $\Gamma(\Map[K,M])$ is a limit of the functor
$\del_{/K}\to \Set$ carrying $(n,\sigma)$ to
$\Hom_{\Mset}((\Delta^m)^\sharp\times (\Delta^n)^\flat,M)$. We are thus
reduced to showing that the maps $\sigma\colon \Delta^n\to K$ exhibit
$(\Delta^m)^\sharp \times K^\flat$ as the colimit of the functor
$\del_{/K}\to \Sset$ carrying $(n,\sigma)$ to $(\Delta^m)^\sharp \times
(\Delta^n)^\flat$. Note that the functor
$(\Delta^m)^\sharp\times(-)^\flat\colon \Sset\to \Mset$ admits a right
adjoint $\Map^\flat((\Delta^m)^\sharp,-)$, hence preserves colimits. The
assertion then follows from Lemma \ref{2le:Hovey}.

Note that $\Map^\sharp(K^\flat,\cC^\natural)$ is the largest Kan complex
contained in $\Fun(K,\cC)$.
\end{remark}

If $g\colon K'\to K$ is a map, then composition with the functor
$\del_{/K'}\to \del_{/K}$ induced by $g$ defines a functor $g^*\colon
(\Sset)^{(\del_{/K})^{op}}\to (\Sset)^{(\del_{/K'})^{op}}$. We have
$g^*\Map[K,M]= \Map[K',M]$.

\begin{proposition}\label{1le:injective_fibration}
Let $f\colon Z\to T$ be a fibration in $\Mset$ with respect to the Cartesian
model structure, and let $K$ be a simplicial set. Then the morphism
$\Map[K,f]\colon \Map[K,Z]\to \Map[K,T]$ is an injective fibration in
$(\Sset)^{(\del_{/K})^{op}}$. In other words, for every commutative square in
$(\Sset)^{(\del_{/K})^{op}}$ of the form
 \[\xymatrix{\cN\ar[r]^-\Phi\ar@{^{(}->}[d] & \Map[K,Z]\ar[d]^-{\Map[K,f]}\\
   \cM\ar[r]_-\Psi\ar@{..>}[ur]^{\Omega} & \Map[K,T]}\]
such that $\cN\hookrightarrow \cM$ is anodyne, there exists a dotted arrow
as indicated, rendering the diagram commutative.
\end{proposition}

In particular, if $\cC$ is an $\infty$-category, then $\Map[K,\cC]$ is an
injectively fibrant object of $(\Sset)^{(\del_{/K})^{op}}$ since
$\cC^\natural$ is a fibrant object of $\Mset$ \cite{Lu1}*{Proposition
3.1.4.1}.

\begin{remark}\label{2re:Quillen}
The functor $\Map[K,-]$ admits a left adjoint $F\colon
(\Sset)^{(\del_{/K})^{op}}\to \Mset$ carrying $\cR$ to the coend of the
diagram
\[(\del_{/K})^{op}\times \del_{/K}\to \Mset,\quad
((n,\sigma),(m,\tau))\mapsto\cR(n,\sigma)^\sharp\times (\Delta^m)^\flat.
\]
The functor $F$ can be described more explicitly as follows. Note that for
functors $G\colon \cC^{op}\to \Set$, $H\colon \cC\to \Set$, where $\cC$ is a
category, the coend of the diagram
\[\cC^{op}\times \cC\to \Set,\quad (A,B)\mapsto G(A)\times H(B)
\]
can be identified with the colimit of the functor $\cD^{op}\to \Set$
carrying $(A,h)$ to $G(A)$, where $\cD$ is the category of elements of $H$.
Thus if we write $F\cR=(X,\cE)$, then $X_m$ is the colimit of the functor
$(\del_{/K}^{[m]})^{op}\to \Set$ carrying $(n,\sigma,\tau)$ to
$\cR(n,\sigma)_m$, where $\del_{/K}^{[m]}$ is the category defined in the
proof of Lemma \ref{2le:Hovey}. Therefore, $X_m$ is the disjoint union of
$\cR(m,\sigma)_m$ for all $m$-simplices $\sigma\colon \Delta^m\to K$.
Moreover, $\cE\subseteq X_1$ is the union of $\cR(1,\sigma)_1$ for all
degenerate edges $\sigma\colon \Delta^1\to K$.

It follows from the above description that $F$ preserves monomorphisms. Thus
Proposition \ref{1le:injective_fibration} shows that the pair
$(F,\Map[K,-])$ is a Quillen adjunction between $\Mset$ endowed with the
Cartesian model structure and $(\Sset)^{(\del_{/K})^{op}}$ endowed with the
injective model structure.
\end{remark}

\begin{remark}
If we replace $\del_{/K}$ by the full subcategory $\del_{/K}^{\r{nd}}$
spanned by nondegenerate simplices, then Proposition
\ref{1le:injective_fibration} still holds and the proof becomes simpler.
However, $\del_{/K}^{\r{nd}}$ is only functorial with respect to
monomorphisms of simplicial sets, which is insufficient for our
applications.
\end{remark}

\begin{proof}[Proof of Proposition \ref{1le:injective_fibration}]
For $n\ge 0$, we let $\cI_n$ denote the full subcategory of $\del_{/K}$
spanned by $(m,\sigma)$ for $m\le n$. We construct $\Omega\res \cI_n^{op}$ by
induction on $n$. It suffices to construct, for every $\sigma\colon
\Delta^n\to K$, a map $\Omega(n,\sigma)$ as the dotted arrow rendering the
following diagram commutative
\[
 \xymatrix{
  \cN(n,\sigma) \ar@{^(->}[d] \ar[rr]^-{\Phi(n,\sigma)}
  &&    \r{Map}^{\sharp}((\Delta^n)^{\flat},
  Z) \ar[d]^-{\Map^\sharp((\Delta^n)^{\flat},f)}   \\
  \cM(n,\sigma) \ar@{..>}[urr]^-{\Omega(n,\sigma)}\ar[rr]_-{\Psi(n,\sigma)}
  && \Map^\sharp((\Delta^n)^{\flat},T),}
\]
such that for every monomorphism $d\colon (n-1,\rho)\to(n,\sigma)$ and every
epimorphism $s\colon (n,\sigma)\to (n-1,\tau)$, the following two diagrams
commute
\[\xymatrix{
   \cM(n,\sigma) \ar[d]_-{\cM(d)}
   \ar[rr]^-{\Omega(n,\sigma)}
   && \r{Map}^{\sharp}((\Delta^{n})^{\flat},
   Z) \ar[d]^-{\r{Res}^d} \\
   \cM(n-1,\rho) \ar[rr]^-{\Omega(n-1,\rho)}
   && \r{Map}^{\sharp}((\Delta^{n-1})^{\flat},
   Z),
}\]
\[\xymatrix{
   \cM(n-1,\tau) \ar[d]_-{\cM(s)}
   \ar[rr]^-{\Omega(n-1,\tau)}
   && \r{Map}^{\sharp}((\Delta^{n-1})^{\flat},
   Z) \ar[d]^-{\r{Res}^s} \\
   \cM(n,\sigma) \ar[rr]^-{\Omega(n,\sigma)}
   && \r{Map}^{\sharp}((\Delta^{n})^{\flat},
   Z).
}\]

By the induction hypothesis, the maps $\Omega(n-1,\rho)$ amalgamate into a
map $\cM(n,\sigma)\to \r{Map}^{\sharp}((\partial\Delta^{n})^{\flat},Z)$, and
the maps $\Omega(n-1,\tau)$ amalgamate into a map $\cM(n,\sigma)^{\deg}\to
\Map^\sharp((\Delta^n)^\flat,Z)$, where $\cM(n,\sigma)^{\deg}\subseteq
\cM(n,\sigma)$ is the union of the images of $\cM(s)\colon \cM(n-1,\tau)\to
\cM(n,\sigma)$. These maps amalgamate with $\Phi(n,\sigma)\colon
\cN(n,\sigma)\to\Map^\sharp((\Delta^{n})^\flat,Z)$ into a map $\Omega'\colon
A\to Z$, where
 \[
 A=(\cN(n,\sigma)\cup \cM(n,\sigma)^{\deg})^\sharp\times (\Delta^{n})^{\flat}
 \coprod_{(\cN(n,\sigma)\cup \cM(n,\sigma)^{\deg})^\sharp\times(\partial\Delta^n)^{\flat}}
 \cM(n,\sigma)^\sharp \times (\partial\Delta^{n})^{\flat},
 \]
fitting into the commutative square
\[
 \xymatrix{
   A \ar@{^(->}[d]_i \ar[rr]^-{\Omega'}
   && Z
    \ar[d]^-{f} \\
   \cM(n,\sigma)^\sharp\times (\Delta^{n})^\flat \ar[rr]_-{\Psi(n,\sigma)}
   \ar@{..>}[rru]^-{\Omega(n,\sigma)}
   && T.   }
\]
It suffices to show that $i$ is a trivial cofibration in $\Mset$ with respect
to the Cartesian model structure, so that there exists a dotted arrow
rendering the above diagram commutative.

Let us first remark that for every epimorphism $s\colon (n',\sigma')\to
(n'',\sigma'')$ of $\del_{/K}$, the left square of the commutative diagram
\[\xymatrix{\cN(n'',\tau'')\ar[d]\ar[r]^{\cN(s)} & \cN(n',\tau')\ar[d]\ar[r]^{\cN(d)} &\cN(n'',\tau'')\ar[d]\\
\cM(n'',\tau'')\ar[r]^{\cM(s)} & \cM(n',\tau')\ar[r]^{\cM(d)} & \cM(n'',\tau''),}
\]
is a pullback by Lemma \ref{2le:retract} below. Here $d$ is a section of
$s$.

Next we prove that the map $\cN(n,\sigma)^{\deg}\to \cM(n,\sigma)^{\deg}$ is
anodyne, where $\cN(n,\sigma)^{\deg}\subseteq \cN(n,\sigma)$ is the union of
the images of $\cN(s)$. More generally, we claim that, for pairwise distinct
epimorphisms $s_1,\dots ,s_m$, where $s_j\colon (n,\sigma)\to (n-1,\tau_j)$,
the inclusion $\cN(\tau_1)\cup\dots \cup
\cN(\tau_m)\hookrightarrow\cM(\tau_1)\cup\dots \cup \cM(\tau_m)$ is anodyne.
Here $\cN(\tau_j)\subseteq \cN(n,\sigma)$ denotes the image of the split
monomorphism $\cN(s_j)$ and similarly for $\cM(\tau_j)$. We proceed by
induction on $m$ (simultaneously for all $n$). The case $m=0$ is trivial and
we assume $m\ge 1$. For $1\le j\le m-1$, form the pushout
\[\xymatrix{(n,\sigma)\ar[r]^{s_j}\ar[d]_{s_m} & (n-1,\tau_j)\ar[d]\\
(n-1,\tau_m)\ar[r]^{s'_j} & (n-2,\tau'_j).}
\]
By Lemma \ref{2le:absolute} below, we have $\cN(\tau_j)\cap
\cN(\tau_m)=\cN(\tau'_j)$, where $\cN(\tau'_j)$ denotes the image of
$\cN(s'_js_m)$. The same holds for $\cM$. It follows that we have the
following pushout square
\[\xymatrix{f_0\ar[r]\ar[d] & f_1\ar[d] \\ f_2\ar[r]&f_3}\]
in the category $(\Sset)^{[1]}$, where
\begin{alignat*}{2}
f_0\colon&&\cN(\tau'_1)\cup \dots\cup\cN(\tau'_{m-1})&\to\cM(\tau'_1)\cup \dots\cup\cM(\tau'_{m-1}),\\
f_1\colon&&\cN(\tau_m)&\to\cM(\tau_m),\\
f_2\colon&&\cN(\tau_1)\cup \dots\cup\cN(\tau_{m-1})&\to\cM(\tau_1)\cup \dots\cup\cM(\tau_{m-1}),\\
f_3\colon&&\cN(\tau_1)\cup \dots\cup\cN(\tau_m)&\to\cM(\tau_1)\cup \dots\cup\cM(\tau_m).
\end{alignat*}
By assumption, $f_1$ is anodyne. By induction hypothesis, $f_0$ and $f_2$
are anodyne. Since $\cN(\tau_m)\cap\cM(\tau'_j)=\cN(\tau'_j)$ by the remark
of the preceding paragraph, Lemma \ref{2le:pushout_anodyne} implies that
$f_3$ is anodyne.

By the remark again, we have $\cN(n,\sigma)\cap
\cM(n,\sigma)^{\deg}=\cN(n,\sigma)^{\deg}$. Thus the inclusion
$\cN(n,\sigma)\subseteq\cN(n,\sigma)\cup \cM(n,\sigma)^{\deg}$ is a pushout
of $\cN(n,\sigma)^{\deg}\subseteq \cM(n,\sigma)^{\deg}$, hence is anodyne.
By assumption, the inclusion $\cN(n,\sigma)\subseteq \cM(n,\sigma)$ is
anodyne. By the two-out-of-three property for weak equivalences, it follows
that the inclusion $\cN(n,\sigma)\cup \cM(n,\sigma)^{\deg}\subseteq
\cM(n,\sigma)$ is anodyne, and consequently the inclusion
$(\cN(n,\sigma)\cup
\cM(n,\sigma)^{\deg})^\sharp\subseteq\cM(n,\sigma)^\sharp$ is a trivial
cofibration in $\Mset$ (see Remark \ref{3re:Quillen} below). The lemma then
follows from the fact that trivial cofibrations in $\Mset$ are stable under
smash products with cofibrations \cite{Lu1}*{Corollary 3.1.4.3}.
\end{proof}

We say that a square in a category $\cC$ is an \emph{absolute pullback}
(resp.\ \emph{absolute pushout}) if every functor $F\colon \cC\to \cD$
carries the square to a pullback (resp.\  pushout) square in $\cD$.

\begin{lemma}\label{2le:retract}
Let $\cC$ be a category. Given a commutative diagram in $\cC$
\[\xymatrix{X\ar[d]_f\ar[r]^{s} & Y\ar[r]^r\ar[d]^{g} & X\ar[d]^f\\
X'\ar[r]^{s'} & Y'\ar[r]^{r'} & X'}
\]
in which both horizontal compositions are identities and $g$ is a
monomorphism, then the square on the left is a pullback square. In
particular, if $g$ is a split monomorphism, then the square on the left is
an absolute pullback.
\end{lemma}

\begin{proof}
The second assertion follows immediately from the first one. To show the
first assertion, let $a\colon W\to X'$ and $b\colon W\to Y$ be morphisms
satisfying $s'a=gb$. If $c\colon W\to X$ is a morphism satisfying $fc=a$ and
$sc=b$, then we have $c=rsc=rb$. Conversely, we have $f(rb)=r'gb=r's'a=a$
and $s(rb)=b$. The last equality follows from $gsrb=s'frb=s'a=gb$, since $g$
is a monomorphism.
\end{proof}

\begin{lemma}\label{2le:absolute}
In $\del_{/K}$, pushouts of epimorphisms by epimorphisms are absolute
pushouts.
\end{lemma}

In the case of $\del\simeq \del_{/\Delta^0}$ the lemma is \cite{JT}*{Theorem
1.2.1} (see also \cite{GZ}*{II.3.2}). The proof in the general case is
similar. We include a proof for completeness.

\begin{proof}
Factorizing epimorphisms into compositions of $s^n_i$'s (for the notation
see the beginning of Section \ref{1ss}), we are reduced to the case of the
pushout of $s^n_i$ by $s^n_j$, where $i\le j$. This case follows from Lemma
\ref{2le:retract} applied to the diagram
\[\xymatrix{(n,\tau)\ar[d]_{s^{n-1}_{j-1}}\ar[r]^{d^{n+1}_i} & (n+1,\sigma)\ar[r]^{s^n_i}\ar[d]^{s^n_j} & (n,\tau)\ar[d]^{s^{n-1}_{j-1}}\\
(n-1,\tau')\ar[r]^{d^n_i} &(n,\sigma')\ar[r]^{s^{n-1}_i} & (n-1,\tau')}
\]
for $i<j$, and to the diagram
\[\xymatrix{(n,\tau)\ar[d]_{\id} \ar[r]^{d^{n+1}_i} & (n+1,\sigma)\ar[d]^{s^n_i}\ar[r]^{s^n_i} & (n,\tau)\ar[d]^{\id}\\
(n,\tau)\ar[r]^{\id} & (n,\tau)\ar[r]^{\id} & (n,\tau)}
\]
for $i=j$.
\end{proof}

\begin{lemma}\label{2le:pushout_anodyne}
Consider a pushout square
\[
\xymatrix{f_0\ar[r]^u\ar[d] & f_1\ar[d] \\ f_2\ar[r]&f_3}
\]
in $(\Sset)^{[1]}$, where $f_i\colon Y_i\to X_i$. Assume that $f_0$, $f_1$,
$f_2$ are anodyne (resp.\ right anodyne) and the map $X_0\coprod_{Y_0} Y_1
\to X_1$ induced by $u$ is a monomorphism. Then $f_3$ is anodyne (resp.\
right anodyne).
\end{lemma}

\begin{proof}
The square corresponds to a cube in $\Sset$, which can be decomposed into a
commutative diagram
\[\xymatrix{Y_0\ar[rr]\ar[rd]^{f_0}\ar[dd]&&Y_1\ar@{=}[rr]\ar[rd]^{g_0}\ar@{-->}[dd] && Y_1\ar[rd]^{f_1}\ar@{-->}[dd]\\
&X_0\ar[dd]\ar[rr] && Z_0\ar[rr]^(.3){a_0}\ar[dd] && X_1\ar[dd]\\
Y_2\ar[rd]^{f_2}\ar@{-->}[rr] && Y_3\ar[rd]^{g_2}\ar@{==}[rr] && Y_3\ar@{-->}[rd]^{f_3}\\
&Y_2\ar[rr] && Z_2\ar[rr]^{a_2} && X_3,}
\]
where the top and bottom squares on the left are pushout squares, and the
front and back squares are pushout squares. For $i=0,2$, the map $g_i$ is a
pushout of $f_i$, hence is anodyne (resp.\ right anodyne). Since $f_1$ is
anodyne (resp.\ right anodyne) and $a_0$ is a monomorphism by assumption,
$a_0$ is anodyne by the two-out-of-three property of weak equivalences
(resp.\ right anodyne by \cite{Lu1}*{Proposition 4.1.1.3}). Thus the pushout
$a_2$ of $a_0$ is anodyne (resp.\  right anodyne). Therefore, $f_3=a_2g_2$
is anodyne (resp.\ right anodyne).
\end{proof}

We now give the form of the construction technique as used in Sections
\ref{4ss} and \ref{5ss}.

\begin{proposition}\label{1pr:extension}
Let $K$ be a simplicial set, $\cC$ an $\infty$-category, and $i\colon
A\hookrightarrow B$ a monomorphism of simplicial sets. Denote by $f\colon
\Fun(B,\cC)\to\Fun(A,\cC)$ the map induced by $i$. Let $\cN$ be an object of
$(\Sset)^{(\del_{/K})^{op}}$ such that $\cN(\sigma)$ is weakly contractible
for all $\sigma\in\del_{/K}$, and let $\Phi\colon \cN\to \Map[K,\Fun(B,\cC)]$
be a morphism such that $\Map[K,f]\circ\Phi\colon \cN\to \Map[K,\Fun(A,\cC)]$
factorizes through $\Delta^0_{(\del_{/K})^{op}}$ to give a functor $a\colon
K\to\Fun(A,\cC)$. Then there exists $b\colon K\to\Fun(B,\cC)$ such that
$b\circ p=a$ and for every map $g\colon K'\to K$ and every global section
$\nu\in \Gamma(g^*\cN)_0$, the maps $b\circ g$ and $g^*\Phi\circ\nu\colon
K'\to\Fun(B,\cC)$ are homotopic over $\Fun(A,\cC)$. Here $g^*\Phi\colon
g^*\cN\to g^*\Map[K,\Fun(B,\cC)]=\Map[K',\Fun(B,\cC)]$.
\end{proposition}

In the statement we have implicitly used isomorphisms provided by Remark
\ref{2re:functors} such as $\Map^\sharp(K,\Fun(A,\cC)^\natural)\simeq
\Gamma(\Map[K,\Fun(A,\cC)])$.

\begin{proof}
Since $\Fun(-,\cC)^\natural=(\cC^\natural)^{(-)^\flat}$, the map
$f^\natural\colon \Fun(B,\cC)^\natural\to \Fun(A,\cC)^\natural$ is a
fibration in $\Mset$ for the Cartesian model structure by Lemma
\ref{1le:marked_fibration} below. Thus by Proposition
\ref{1le:injective_fibration}, $\Map[K,f^\natural]$ is an injective
fibration. We let $\Map[K,f^{\natural}]_a$ denote the fiber of
$\Map[K,f^{\natural}]$ at $a$, which is injectively fibrant. By Lemma
\ref{1le:component} (1), $\Gamma_\Phi(\Map[K,f^{\natural}]_a)$ is a
(nonempty) connected component of $\Gamma(\Map[K,f^{\natural}]_a)$. Note that
$\Gamma(\Map[K,f^\natural]_a)$ is the fiber of
$\Gamma(\Map[K,\Fun(B,\cC)])\to \Gamma(\Map[K,\Fun(A,\cC)])$ at $a$. Any
vertex of $\Gamma_\Phi(\Map[K,f^{\natural}]_a)$ then provides the desired
$b$. Indeed, for given $g$ and $\nu$, both $b\circ g$ and $g^*\Phi\circ\nu$
are given by vertices of the connected Kan complex
$\Gamma_{g^*\Phi}(\Map[K',f^\natural]_{g^*a})$, which are necessarily
equivalent.
\end{proof}

\begin{lemma}\label{1le:marked_fibration}
Let $X\to Y$ be a fibration and $i\colon A\to B$ be a cofibration in $\Mset$
with respect to the Cartesian model structure. Then the induced map
\[X^B\to X^A\times_{Y^A} Y^B\]
is a fibration in $\Mset$ with respect to the Cartesian model structure.
\end{lemma}

\begin{proof}
This follows immediately from the fact that trivial cofibrations in $\Mset$
are stable under smash product with cofibrations \cite{Lu1}*{Corollary
3.1.4.3}.
\end{proof}

\section{Restricted multisimplicial nerves}
\label{3ss}

In this section, we introduce several notions related to multisimplicial
sets. The restricted multisimplicial nerve (Definition
\ref{3de:restricted_nerve}) of a multi-tiled simplicial set (Definition
\ref{3de:multitiled_simplicial}) will play an essential role in the
statements of our theorems.

\begin{definition}[Multisimplicial set]\label{3de:multisimplicial}
Let $I$ be a set. We define the category of \emph{$I$-simplicial sets} to be
$\Sset[I]\colonequals\Fun((\del^I)^{op},\Set)$, where
$\del^I\colonequals\Fun(I,\del)$. For an integer $k\ge 0$, we define the
category of $k$-simplicial sets to be $\Sset[k]\colonequals\Sset[I]$, where
$I=\{1,\dots,k\}$. We identify $\Sset[1]$ with $\Sset$.
\end{definition}

We denote by $\Delta^{n_i\mid i\in I}$ the $I$-simplicial set represented by
the object $([n_i])_{i\in I}$ of $\del^I$. For an $I$-simplicial set $S$, we
denote by $S_{n_i\mid i\in I}$ the value of $S$ at the object $([n_i])_{i\in
I}$ of $\del^I$. An $(n_i)_{i\in I}$-\emph{simplex} of an $I$-simplicial set
$S$ is an element of $S_{n_i\mid i\in I}$. By Yoneda's lemma, there is a
canonical bijection between the set $S_{n_i\mid i\in I}$ and the set of maps
from $\Delta^{n_i\mid i\in I}$ to $S$.

For $J\subseteq I$, composition with the partial opposite functor $\del^I\to
\del^I$ sending $(\dots,P_{j'},\dots,P_j,\dots)$ to
$(\dots,P_{j'},\dots,P_j^{op},\dots)$ (taking $op$ for $P_j$ when $j\in J$)
defines a functor $\op^I_J\colon \Sset[I]\to \Sset[I]$. We put
$\Delta^{n_i\res i\in I}_J=\op^I_J\Delta^{n_i\res i\in I}$. Although
$\Delta^{n_i\res i\in I}_J$ is isomorphic to $\Delta^{n_i\res i\in I}$, it
will be useful in specifying the variance of many constructions. When
$J=\emptyset$, $\op^I_\emptyset$ is the identity functor so that
$\Delta^{n_i\res i\in I}_\emptyset=\Delta^{n_i\res i\in I}$.

\begin{remark}
The category $\Sset[I]$ is Cartesian-closed. For two $I$-simplicial sets
$X,Y$, the internal mapping object $\Map(Y,X)$ is an $I$-simplicial set such
that $\Hom_{\Sset[I]}(Z,\Map(Y,X))\simeq \Hom_{\Sset[I]}(Z\times Y,X)$ for
every $Z\in\Sset[I]$. We have $\op^I_J \Map(Y,X)\simeq \Map(\op^I_J
Y,\op^I_J X)$.
\end{remark}

\begin{definition}
Let $I,J$ be two sets.
\begin{enumerate}
  \item Let $f\colon J\to I$ be a map of sets. Composition with $f$
      defines a functor $\del^f\colon \del^I\to \del^J$. Composition with
      $(\del^f)^{op}$ induces a functor $(\del^f)^*\colon \Sset[J]\to
      \Sset[I]$, which has a right adjoint $(\del^f)_*\colon \Sset[I]\to
      \Sset[J]$. We will now look at two special cases.

  \item Let $f\colon J\to I$ be an injective map. The functor $\del^f$ has
      a right adjoint $c_f\colon \del^J\to \del^I$ given by $c_f(F)_i=F_j$
      if $f(j)=i$ and $c_f(F)_i=[0]$ if $i$ is not in the image of $f$.
      The functor $(\del^f)_*$ can be identified with the functor
      $\epsilon^f$ induced by composition with $(c_f)^{op}$. If
      $J=\{1,\dots,k'\}$, we write $\epsilon^I_{f(1)\dotsm f(k')}$ for
      $\epsilon^f$.

  \item Consider the map $f\colon I\to \{1\}$. Then
      $\delta_I\colonequals\del^f\colon \del\to\del^I$ is the diagonal
      functor, and composition with $(\delta_I)^{op}$ induces the
      \emph{diagonal functor} $\delta^*_I=(\del^f)^*\colon \Sset[I]\to
      \Sset$. We define
      \[\Delta^{[n_i]_{i\in I}}\colonequals
      \delta_I^*\Delta^{n_i\res i\in I} =\prod_{i\in I}\Delta^{n_i}.
      \] We define the \emph{multisimplicial nerve} functor to be the right
      adjoint $\delta^I_*\colon \Sset\to\Sset[I]$ of $\delta^*_I$. An
      $(n_i)_{i\in I}$-simplex of $\delta^I_*X$ is given by a map
      $\Delta^{[n_i]_{i\in I}}\to X$.

  \item For $J\subseteq I$, we define the \emph{twisted diagonal functor}
      $\delta^*_{I,J}$ as $\delta^*_I \circ \op^I_J\colon \Sset[I]\to
      \Sset$. We define
      \begin{align*}
      \Delta^{[n_i]_{i\in I}}_J\colonequals \delta_{I,J}^*\Delta^{n_i\res i\in I}=\delta_I^*\Delta^{n_i\res i\in I}_J
      =\left(\prod_{i\in I-J}\Delta^{n_i}\right)\times\left(\prod_{j\in J}(\Delta^{n_j})^{op}\right).
      \end{align*}
      When $J=\emptyset$, we have $\delta^*_{I,\emptyset}=\delta^*_I$ and
      $\Delta^{[n_i]_{i\in
     I}}_{\emptyset}=\Delta^{[n_i]_{i\in I}}$.
\end{enumerate}
\end{definition}

When $I=\{1,\dots,k\}$, we write $k$ instead of $I$ in the previous notation.
For example, in (2) we have $(\epsilon^k_jK)_n=K_{0,\dots ,n,\dots, 0}$,
where $n$ is at the $j$-th position and all other indices are $0$. In (3) we
have $\delta^*_k\colon \Sset[k]\to\Sset$ defined by
$(\delta^*_kX)_n=X_{n,\dots, n}$.

\begin{remark}\label{3re:eps}
For any map $f\colon J\to I$, we have $\del^f\circ \delta_I=\delta_J$, so
that $(\del^f)_*\circ \delta^I_*\simeq \delta^J_*$. In particular, for $f$
injective, we have $\epsilon^f\circ \delta^I_*\simeq \delta^J_*$. For
$\alpha\in I$, we have $\epsilon^I_\alpha\circ \delta^I_*\simeq \id_{\Sset}$.
\end{remark}

\begin{remark}
For $f\colon J\to I$ injective, we have $\del^f\circ c_f=\id_{\del^J}$, so
that $\epsilon^f\circ (\del^f)^*=\id_{\Sset[J]}$. The counit transformation
$(\del^f)^* \circ \epsilon^f\to \id_{\Sset[I]}$ is a monomorphism. Indeed,
for each object $P$ of $\del^I$, the unit morphism $P\to (c_f\circ \del^f)
(P)$ admits a section. Applying the functor $\delta_I^*$, we obtain a
monomorphism $\delta_J^*\circ \epsilon^f\to \delta_I^*$.
\end{remark}

\begin{remark}\label{3re:adjunction}
For every map $f\colon J\to I$, the adjunction formula for presheaves
provides a canonical isomorphism
\[\Map(Y,(\del^f)_*X)\simeq (\del^f)_*\Map((\del^f)^*Y,X)\]
for every $I$-simplicial set $X$ and every $J$-simplicial set $Y$. This map
is the composite map
\[\Map(Y,(\del^f)_*X)\xto{(\del^f)^*}(\del^f)_*\Map((\del^f)^*Y,(\del^f)^*(\del^f)_*X)\to (\del^f)_*\Map((\del^f)^*Y,X),\]
where the second map is induced by the counit map $(\del^f)^*(\del^f)_*X\to
X$.

Specializing to the case of $\delta_I$ and applying the functor
$\epsilon^I_\alpha$, where $\alpha\in I$, we get an isomorphism
\[\epsilon^I_\alpha\Map(X,\delta^I_*S)\simeq \Map(\delta_I^*X,S)\]
for every $I$-simplicial set $X$ and every simplicial set $S$, which is the
composite map
\[\epsilon^I_\alpha\Map(X,\delta^I_*S)\xto{\delta_I^*} \Map(\delta_I^*X,\delta_I^*\delta^I_*S)\to\Map(\delta_I^*X,S).\]
\end{remark}

\begin{definition}[Exterior product]
Let $I=\coprod_{j\in J}I_j$ be a partition. We define a functor
\[\boxtimes_{j\in J}\colon \prod_{j\in J}\Sset[I_j]\to
\Sset[I]
\] by the formula $\boxtimes_{j\in J}S^j=\prod_{j\in J}(\del^{\iota_j})^* S^j$,
where $\iota_j\colon I_j\hookrightarrow I$ is the inclusion. For
$J=\{1,\dots,m\}$, $I_j=\{1,\dots,k_j\}$, we define
\[
-\boxtimes\dots \boxtimes -\colon \Sset[k_1]\times\dots \times\Sset[k_m]\to\Sset[k].
\] by $(S^1\boxtimes \dots \boxtimes S^m)_{n^1_1,\dots,n^1_{k_1},\dots,
n^m_{1},\dots, n^m_{k_m}}=S^1_{n^1_1,\dots,n^1_{k_1}}\times \dots \times
S^m_{n^m_{1},\dots,n^m_{k_m}}$.
\end{definition}

We have the isomorphisms $\boxtimes_{i\in I}\Delta^{n_i}\simeq\Delta^{n_i\mid
i\in I}$ and $\delta_I^*\boxtimes_{j\in J}S^j\simeq \prod_{j\in J}
\delta_{I_j}^*S^j$.

\begin{remark}
For a map $f\colon J\to I$, we have $(\del^f)^*\Delta^{n_i\mid i\in
I}\simeq\boxtimes_{i\in I}\Delta^{[n_j]_{j\in f^{-1}(i)}}$, so that an
$(n_j)_{j\in J}$-simplex of $(\del^f)_* X$ is given by a map
$\boxtimes_{i\in I} \Delta^{[n_j]_{j\in f^{-1}(i)}}\to X$.
\end{remark}

We next turn to restricted variants of the multisimplicial nerve functor
$\delta^I_*$. We start with restrictions on edges.

\begin{definition}[Multi-marked simplicial set]\label{3de:multimarked_simplicial}
An \emph{$I$-marked simplicial set} (resp.\ \emph{$I$-marked
$\infty$-category}) is the data $(X,\cE=\{\cE_i\}_{i\in I})$, where $X$ is a
simplicial set (resp. an $\infty$-category) and, for all $i\in I$, $\cE_i$
is a set of edges of $X$ which contains every degenerate edge. The data
$\cE$ is sometimes called an \emph{$I$-marking} on $X$. A morphism $f\colon
(X,\{\cE_i\}_{i\in I})\to (X',\{\cE'_i\}_{i\in I})$ of $I$-marked simplicial
sets is a map $f\colon X\to X'$ having the property that $f(\cE_i)\subseteq
\cE'_i$ for all $i\in I$. We denote the category of $I$-marked simplicial
sets by $\Sset^{I+}$. It is the strict fiber product of $I$ copies of
$\Mset$ over $\Sset$.

For a simplicial set $X$ and a subset $J\subseteq I$, we define an
$I$-marked simplicial set $X^{\sharp^I_J} =(X,\cE)$ by $(X,\cE_j)=X^\sharp$
for $j\in J$ and $(X,\cE_i)=X^\flat$ for $i\in I-J$. We write
$X^{\sharp^I}=X^{\sharp^I_I}$ and $X^{\flat^I}=X^{\sharp^I_\emptyset}$. The
functor $\Sset\to \Sset^{I+}$ carrying $X$ to $X^{\sharp^I}$ (resp.\
$X^{\flat^I}$) is a right (resp.\ left) adjoint of the forgetful functor
$\Sset^{I+}\to \Sset$.
\end{definition}

Consider the functor $\delta_{I+}^*\colon \Sset[I]\to \Sset^{I+}$ sending $S$
to $(\delta^*_I S,\{\cE_i\}_{i\in I})$, where $\cE_i$ is the set of edges of
$\epsilon^I_i S\subseteq\delta_I^*S$. This functor admits a right adjoint
$\delta^{I+}_*\colon \Sset^{I+}\to \Sset[I]$. Since
$\delta_{I+}^*\Delta^{n_i\mid i\in I}=\prod_{i\in
I}(\Delta^{n_i})^{\sharp^I_{\{i\}}}$, the functor $\delta^{I+}_*$ carries
$(X,\{\cE_i\}_{i\in I})$ to the $I$-simplicial subset of $\delta^I_* X$ whose
$(n_i)_{i\in I}$-simplices are maps $\Delta^{[n_i]_{i\in I}}\to X$ such that
for every $j\in I$ and every map $\Delta^1\to \epsilon^I_j\Delta^{n_i\res
i\in I}$, the composition
\[\Delta^1\to\epsilon^I_j\Delta^{n_i\res i\in I} \to\Delta^{[n_i]_{i\in I}} \to X\]
is in $\cE_j$. We have $\delta^I_*(X)=\delta^{I+}_*(X^{\sharp I})$. When
$I=\{1,\dots,k\}$, we use the notation $\Sset^{k+}$ \footnote{In particular,
$\Sset^{2+}$ in our notation is $\Sset^{++}$ in \cite{Lu2}*{Definition
4.7.5.2}.}, $\delta_{k+}^*$ and $\delta^{k+}_*$.

\begin{definition}[Restricted multisimplicial nerve]
We define the \emph{restricted $I$-simplicial nerve} of an $I$-marked
simplicial set $(X,\cE=\{\cE_i\}_{i\in I})$ to be the $I$-simplicial set
\[X_{\cE}=X_{\{\cE_i\}_{i\in I}}\colonequals\delta^{I+}_*(X,\{\cE_i\}_{i\in I}).\]
In particular, for any marked simplicial set $(X,\cE)$, the simplicial set
$X_\cE$ is the simplicial subset of $X$ spanned by the edges in $\cE$.
\end{definition}

\begin{remark}\label{3re:Quillen}
The functor $\delta^*_{1+}\colon \Sset\to \Mset$ carries $S$ to $S^\sharp$.
The functor $\delta^{1+}_*\colon \Mset\to \Sset$ carries $(X,\cE)$ to the
simplicial subset of $X$ consisting of all simplices whose edges are all
marked edges. In other words, $X_\cE=\delta^{1+}_*(X,\cE)$ is the largest
simplicial subset $S\subseteq X$ such that $S^\sharp \subseteq (X,\cE)$. We
have $\delta^{1+}_*\simeq \Map^\sharp((\Delta^0)^\flat,-)$. For objects $X$
and $Y$ of $\Mset$, we have $\Map^\sharp(X,Y)=\delta^{1+}_*(Y^X)$.

The pair $(\delta^*_{1+},\delta^{1+}_*)$ is a Quillen adjunction for the Kan
model structure on $\Sset$ and the Cartesian model structure on $\Mset$.
This is a special case of Remark \ref{2re:Quillen} but we can also check
this easily as follows. Clearly $\delta^*_{1+}$ preserves cofibrations. To
see that it also preserves trivial cofibrations, note that for any anodyne
map of simplicial sets $T\to S$ and any $\infty$-category $\cC$, the induced
map $\Map^\sharp(S^\sharp,\cC^\natural)\to
\Map^\sharp(T^\sharp,\cC^\natural)$ is a trivial Kan fibration.
\end{remark}

Next we consider restrictions on squares. By a \emph{square} of a simplicial
set $X$, we mean a map $\Delta^1\times \Delta^1 \to X$. The transpose of a
square is obtained by swapping the two $\Delta^1$'s. Composition with the
maps $\id \times d^1_0,\id \times d^1_1\colon \Delta^1\simeq \Delta^1\times
\Delta^0 \to \Delta^1\times \Delta^1$ induce maps $\Hom(\Delta^1\times
\Delta^1,X)\to X_1$ and composition with the map $\id\times s^0_0 \colon
\Delta^1\times \Delta^1\to \Delta^1\times \Delta^0\simeq \Delta^1$ induces a
map $X_1\to \Hom(\Delta^1\times \Delta^1,X)$.

\begin{definition}[Multi-tiled simplicial set]\label{3de:multitiled_simplicial}
An \emph{$I$-tiled simplicial set} (resp.\ \emph{$I$-tiled
$\infty$-category}) is the data $(X,\cE=\{\cE_i\}_{i\in
I},\cQ=\{\cQ_{ij}\}_{i,j\in I, i\neq j})$, where $(X,\cE)$ is an $I$-marked
simplicial set (resp.\ $\infty$-category) and, for all $i,j\in I$, $i\neq
j$, $\cQ_{ij}$ is a set of squares of $X$ such that $\cQ_{ij}$ and
$\cQ_{ji}$ are obtained from each other by transposition of squares, and
$\id \times d^1_0$, $\id \times d^1_1$ induce maps $\cQ_{ij}\to \cE_i$, and
$\id \times s^0_0$ induces $\cE_i\to \cQ_{ij}$. A morphism $f\colon
(X,\cE,\cQ)\to (X',\cE',\cQ')$ of $I$-tiled simplicial sets is a map
$f\colon X\to X'$ having the property that $f(\cE_i)\subseteq f(\cE'_i)$ and
$f(\cQ_{ij})\subseteq\cQ'_{ij}$ for all $i,j$. We denote the category of
$I$-tiled simplicial sets by $\Sset^{I\square}$. The data $\cT=(\cE,\cQ)$ is
sometimes called an \emph{$I$-tiling} on $X$. For brevity, we adopt the
conventions $\cT_i=\cE_i$ and $\cT_{ij}=\cQ_{ij}$.
\end{definition}

\begin{remark}\label{3re:multitile}
Note that $\cE_i$ is the image of $\cQ_{ij}$ under either of the maps
$\cQ_{ij}\to \cE_i$ given by $\id\times d^1_0$ and $\id\times d^1_1$.
Moreover, $f(\cQ_{ij})\subseteq \cQ'_{ij}$ implies $f(\cE_i)\subseteq
\cE'_i$ and $f(\cE_j)\subseteq \cE'_j$.
\end{remark}

Consider the functor $\delta_{I\square}^*\colon \Sset[I]\to
\Sset^{I\square}$ carrying $S$ to $(\delta^*_{I+} S,\cQ)$, where $\cQ_{ij}$
is the image of the injection
\[(\epsilon^I_{ij}S)_{11}=\Hom_{\Sset[2]}(\Delta^{1,1},\epsilon^I_{ij}S)\xto{\delta^*_2}
\Hom_{\Sset}(\Delta^1\times \Delta^1,\delta^*_2\epsilon^I_{ij}S)\subseteq \Hom_{\Sset}(\Delta^1\times \Delta^1,\delta^*_I S).\]
This functor admits
a right adjoint $\delta^{I\square}_*\colon \Sset^{I\square}\to \Sset[I]$
carrying $(X,\cE,\cQ)$ to the $I$-simplicial subset of $\delta^{I+}_*
(X,\cE)\subseteq \delta^I_* X$ whose $(n_i)_{i\in I}$-simplices are maps
$\Delta^{[n_i]_{i\in I}}\to X$ satisfying the additional condition that for
every pair of elements $j,k\in I$, $j\neq k$, and every map
$\Delta^1\boxtimes \Delta^1\to \epsilon^I_{jk}\Delta^{n_i\res i\in I}$, the
composition
\[\Delta^1\times \Delta^1\to\delta_2^*\epsilon^I_{jk}\Delta^{n_i\res i\in I}\to\Delta^{[n_i]_{i\in I}} \to X\]
is in $\cQ_{jk}$. When $I=\{1,\dots,k\}$, we use the notation
$\Sset^{k\square}$, $\delta^*_{k\square}$, $\delta_*^{k\square}$.

\begin{definition}[Restricted multisimplicial nerve]\label{3de:restricted_nerve}
We define the \emph{restricted $I$-simplicial nerve} of an $I$-tiled
simplicial set $(X,\cT)$ to be the $I$-simplicial set
$\delta_*^{I\square}(X,\cT)$.
\end{definition}

\begin{notation}\label{3no:cartesian}
The underlying functor $\sfU\colon \Sset^{I\square}\to \Sset^{I+}$ carrying
$(X,\cE,\cQ)$ to $(X,\cE)$ admits a left adjoint $\sfV\colon \Sset^{I+}\to
\Sset^{I\square}$ and a right adjoint $\sfW\colon \Sset^{I+}\to
\Sset^{I\square}$, which can be described as follows.
\begin{itemize}
  \item We have $\sfV(X,\cE)=(X,\cE,\cQ)$, where $\cQ_{ij}$ is the union
      of the image of $\cE_i$ under $-\circ(\id \times s_0^0)$ and the
      image of $\cE_j$ under $-\circ (s_0^0\times \id)$.
  \item For sets of edges $\cE_1$ and $\cE_2$ of $X$, we denote by
      $\cE_1*_X\cE_2$ the set of squares $f\colon \Delta^1\times
      \Delta^1\to X$
      \[\xymatrix{f(0,0)\ar[r] \ar[d] &
      f(0,1)\ar[d]\\f(1,0)\ar[r] & f(1,1)}\] such that the vertical edges
      $f\circ (\id \times d^1_\alpha)$, $\alpha=0,1$ belong to $\cE_1$
      and the horizontal edges $f\circ (d^1_\alpha\times \id)$,
      $\alpha=0,1$ belong to $\cE_2$. We have $\sfW(X,\cE)=(X,\cE,\cQ)$,
      where $\cQ_{ij}=\cE_i*_X\cE_j$.
\end{itemize}
We have $\delta^*_{I+}\simeq\sfU\circ\delta^*_{I\square}$ and
$\delta_*^{I+}\simeq \delta_*^{I\square}\circ\sfW$.
\end{notation}

\begin{definition}[Cartesian multisimplicial nerve]\label{3de:cartesian_nerve}
If $\cC$ is an $\infty$-category and
$\cE_1$, $\cE_2$ are sets of edges of $\cC$, we denote by
$\cE_1*^\cart_{\cC}\cE_2$ the subset of $\cE_1*_{\cC}\cE_2$ consisting of
Cartesian squares. For an $I$-marked $\infty$-category $(\cC,\cE)$, we denote
by $(\cC,\cE^\cart)$ the $I$-tiled $\infty$-category such that
$\cE^\cart_i=\cE_i$ for $i\in I$ and $\cE^\cart_{ij}=\cE_i*_{\cC}^\cart\cE_j$
for $i,j\in I$ and $i\neq j$. We define the \emph{Cartesian $I$-simplicial
nerve} of an $I$-marked $\infty$-category $(\cC,\cE)$ to be
\[\cC^\cart_{\cE}\colonequals\delta^{I\square}_*(\cC,\cE^\cart).\]
\end{definition}

For reference in later sections, we define a few properties of sets of edges
and squares. As in the definition of marked simplicial sets, we are mainly
interested in those sets of edges that contain all degenerate edges.
However, many sets of squares of interest, when regarded as sets of edges in
suitable simplicial sets, do not contain all degenerate edges. For this
reason, we allow sets of edges not containing all degenerate edges in the
definitions below.

\begin{definition}\label{3de:edges}
Let $X$ be a simplicial set, and let $\cE$ be a set of edges of $X$. We say
that $\cE$ is
\begin{enumerate}
    \item \emph{composable} if every map $\Lambda^2_1\to X$ whose
        restrictions to $\Delta^{\{0,1\}}$ and to $\Delta^{\{1,2\}}$ are
        in $\cE$ extends to a $2$-simplex $\Delta^2\to X$ whose
        restriction to $\Delta^{\{0,2\}}$ is in $\cE$.
    \item \emph{stable under composition} for every $2$-simplex $\sigma$
        of $X$ such that $\sigma\circ d^2_0, \sigma\circ d^2_2\in \cE$,
        we have $\sigma\circ d^2_1\in \cE$.
\end{enumerate}
\end{definition}

If $\cE$ contains every degenerate edge, then (1) above is equivalent to
every one of the following conditions
\begin{itemize}
   \item $(X,\cE)$ has the extension property with respect to the
       inclusion $(\Lambda^2_1)^\sharp\subseteq (\Delta^2)^\sharp$;

   \item $X_\cE$ has the extension property with respect to the inclusion
       $\Lambda^2_1\subseteq \Delta^2$;
\end{itemize}
and (2) above is equivalent to every one of the following conditions
\begin{itemize}
   \item $(X,\cE)$ has the extension property with respect to the
       inclusion
       \[(\Lambda^2_1)^{\sharp}\coprod_{(\Lambda^2_1)^{\flat}}
       (\Delta^2)^{\flat}\subseteq (\Delta^2)^{\sharp};\]

   \item $X_\cE\to X$ has the right lifting property with respect to the
       inclusion $\Lambda^2_1\subseteq \Delta^2$;

   \item $X_\cE\to X$ is an inner fibration.
\end{itemize}
If $X$ has the extension property with respect to $\Lambda^2_1\subseteq
\Delta^2$, then (2) implies (1).

\begin{definition}\label{3de:admissible_edge}
Let $\cC$ be an $\infty$-category and let $\cE$, $\cF$ be two sets of edges
of $\cC$. We say that $\cE$ is
\begin{enumerate}
   \item \emph{stable under homotopy} if for $e\in\cE$ and $f\in\cC_1$
       that have the same image in $\H\cC$, we have $f\in\cE$;

   \item \emph{stable under equivalence} if for $e\in\cE$ and $f\in\cC_1$
       that are equivalent as objects of $\Fun(\Delta^1,\cC)$, we have
       $f\in\cE$;

   \item \emph{stable under pullback by $\cF$} if for every Cartesian
       square in $\cC$ of the form
      \begin{align*}
      \xymatrix{
      y' \ar[d]_{e'} \ar[r] & y \ar[d]^e \\
      x' \ar[r]^f & x   }
      \end{align*}
      with $e\in\cE$ and $f\in\cF$, we have $e'\in\cE$;

   \item \emph{stable under pullback} (see \cite{Lu1}*{Notation 6.1.3.4})
       if it is stable under pullback by $\cC_1$;

   \item \emph{admissible} if $\cE$ contains every degenerate edge of
       $\cC$, is stable under pullback, and for every $2$-simplex of
       $\cC$ of the form
      \begin{equation}\label{3eq:2cell}
      \xymatrix{&y\ar[rd]^p \\ z\ar[ru]^q\ar[rr]^r &&
      x}
\end{equation}
       with $p\in\cE$, we have $q\in\cE$ if and only if $r\in\cE$.
\end{enumerate}
\end{definition}

In the above definition, (5) implies (4); (4) implies (3); (2) implies (1).
Moreover, if $\cF$ contains every edge of $\cC$ that is an equivalence
(resp.\ degenerate), then (3) implies (2) (resp.\ (1)). If $\cE$ satisfies
(3) with $\cE$ and $\cF$ each containing all degenerate edges of $\cC$, then
$\cE$ contains all equivalences of $\cC$. The last condition in (5) is
equivalent to saying that $X_\cE\to X$ is a right fibration
\cite{Lu1}*{Definition 2.0.0.3}.

\begin{remark}\label{3re:admissible}
If $\cC$ admits pullbacks, then $\cE$ is admissible if and only if it
contains every degenerate edge of $\cC$ and is stable under composition,
pullback, and taking diagonal in $\cC$. The ``only if'' part is clear. For
the ``if'' part, note that in the $2$-simplex \eqref{3eq:2cell}, $q$ is a
composition of
\[z\to z\times_x y \to y,\]
where the first morphism is a pullback of the diagonal $y\to y\times_x y$ of
$p$ and the second morphism is a pullback of $r$ by $p$. Indeed, we have a
diagram with pullback squares
\[\xymatrix{z\ar[r]^q\ar[d] & y\ar[d]\\
z\times_x y\ar[r]\ar[d] & y\times_x y \ar[r]\ar[d] & y\ar[d]^p\\
z\ar[r]^q & y\ar[r]^p & x.}
\]
\end{remark}

In an $\infty$-category $\cC$, a set of edges $\cE$ is composable if and
only if its image in $\H\cC$ is stable under composition. Thus if $\cE$ is
composable and stable under homotopy, then $\cE$ is stable under
composition. The converse holds if $\cE$ contains every degenerate edge. In
the next section, we will need the following extension property of
composable sets of edges.

\begin{lemma}\label{3le:comp}
Let $I$ be a set. Let $(B,\cF)$ be an $I$-marked simplicial set and
$(\cC,\cE)$ an $I$-marked $\infty$-category.  Let $A\subseteq B$ be a
categorical equivalence such that for each $i\in I$, $\cF_i$ is contained in
the smallest set of edges of $B$ containing $\cG_i=A_1\cap \cF_i$ and stable
under composition. Assume $\cE_i$ composable for all $i\in I$ and $\cF_i\cap
\cF_j\subseteq A_1$ for all $i,j\in I$, $i\neq j$. Then $(\cC,\cE)$ has the
extension property with respect to $(A,\cG)\subseteq (B,\cF)$.
\end{lemma}

\begin{proof}
Let $f\colon (A,\cG)\to (\cC,\cE)$ be a map of $I$-marked simplicial sets.
Choose an extension $g\colon B\to \cC$ of $f$. For each $i\in I$, let
$\cE'_i$ denote the set of edges of $\cC$ that are homotopic to some edge of
$\cE_i$. Then $\cE'_i$ is stable under composition, and hence so is its
inverse image under $g$. Thus $g$ induces $(B,\cF)\to(\cC,\cE')$. Let
$D\subseteq B$ be the union of $A$ and the edges in $\cF$. We construct a map
$h_0\colon (D,\cF)\to (\cC,\cE)$ extending $f$ and a natural equivalence
$g\res D\to h_0$ extending $\id_f$, by choosing for each edge $e$ in $\cF_i$
but not in $A_1$, a homotopy from $g(e)$ to an edge $h_0(e)$ in $\cE_i$. By
\cite{Lu1}*{Lemma 2.4.6.3}, $h_0$ extends to a map $(B,\cF)\to (\cC,\cE)$, as
desired.
\end{proof}

\begin{definition}\label{3de:squares}
For a simplicial set $X$, the map $\Hom(\Delta^1\times \Delta^1,
X)\xrightarrow{\sim}\Hom(\Delta^1, \Map(\Delta^1,X))$ carrying $f$ to
$a\mapsto (b\mapsto f(a,b))$ (resp.\ $a\mapsto (b\mapsto f(b,a))$) is an
isomorphism.
\begin{enumerate}
  \item We say that a set of squares $\cQ$ of $X$ is \emph{stable under
      composition in the first (resp.\ second) direction} if the resulting
      set of edges of $\Map(\Delta^1,X)$ is stable under composition.
\end{enumerate}
Now let $\cQ$ and $\cQ'$ be sets of squares of an $\infty$-category $\cC$.
\begin{enumerate}\setcounter{enumi}{1}
  \item We say that $\cQ$ is \emph{stable under equivalence} if $\cQ$,
      when viewed as a set of edges of $\Map(\Delta^1,\cC)$ via the above
      isomorphism, is stable under equivalence.

  \item We say that $\cQ$ is \emph{stable under pullback by $\cQ'$ in the
      first (resp.\ second) direction}, if $\cQ$ is stable under pullback
      by $\cQ'$ in $\Map(\Delta^1,\cC)$, where $\cQ$ and $\cQ'$ are
      viewed as sets of edges via the above isomorphism.

  \item We say that $\cQ$ is \emph{stable under pullback in the first
      (resp.\ second) direction} if (3) holds for
      $\cQ'=\Fun(\Delta^1\times \Delta^1,\cC)$, the set of all squares of
      $\cC$.
\end{enumerate}
\end{definition}

By \cite{Lu1}*{Corollary 5.1.2.3}, condition (3) means that for any cube in
$\cC$ of the form
\begin{equation}\label{3eq:cube}
\xymatrix{y'(0)\ar[rr]\ar[rd]\ar[dd]&&y(0)\ar[rd]\ar@{-->}[dd]\\
&y'(1)\ar[dd]\ar[rr] && y(1)\ar[dd]\\
x'(0)\ar[rd]\ar@{-->}[rr] && x(0)\ar@{-->}[rd]\\
&x'(1)\ar[rr] && x(1),}
\end{equation}
such that the front and back squares are pullback, such that the right square
is in $\cQ$, and such that the bottom square is in $\cQ'$, the left square is
in $\cQ$. Here we interpret the horizontal and vertical arrows as in the
first (resp.\ second) direction and the oblique arrows as in the other
direction.

\begin{lemma}\label{3le:cart}
Let $\cC$ be an $\infty$-category. Let $\cQ^\cart$ be the set of all
pullback squares of $\cC$. Then the image of $\cQ^\cart$ under each of the
two isomorphisms in Definition \ref{3de:squares} is an admissible set of
edges. In particular, $\cQ^\cart$ is stable under equivalence, stable under
composition in both directions, and stable under pullback in both
directions.
\end{lemma}

\begin{proof}
The last condition in the definition of admissibility is \cite{Lu1}*{Lemma
4.4.2.1}. It remains to show the stability under pullback. Consider a cube
of the form \eqref{3eq:cube} in which the front, back, and right squares are
pullback. By the ``if'' part of (1), the square with vertices $y'(0)$,
$y(1)$, $x'(0)$, $x(1)$ is a pullback square. We conclude by the ``only if''
part of (1).
\end{proof}

\begin{remark}\label{3re:cart_square}
Let $\cC$ be an $\infty$-category and let $\cE_1$, $\cE_2$, $\cE_3$ be sets
of edges of $\cC$. Lemma \ref{3le:cart} has the following consequences.
\begin{enumerate}
  \item If $\cE_1$ is stable under composition, then
      $\cE_1*^\cart_\cC\cE_2$ is stable under composition in the first
      direction.

  \item If $\cE_2$ and $\cE_3$ are stable under pullback by $\cE_1$, then
      $\cE_2*_\cC\cE_3$ and $\cE_2*_\cC^\cart\cE_3$ are stable under
      pullback by $\cE_1*^\cart_\cC \cE_3$ in the first direction.

  \item If $\cE_3$ is stable under pullback by $\cE_2$, and $\cE_2$ is
      stable under pullback by $\cE_1$, then $\cE_2*_\cC^\cart\cE_3$ is
      stable under pullback by $\cE_1*_\cC \cE_3$ (and, in particular, by
      $\cE_1*_\cC^\cart \cE_3$) in the first direction.
\end{enumerate}
\end{remark}

\begin{remark}
Let $\cC$ be an ordinary category, and let $\cE_1,\dots,\cE_k$ be sets of
morphisms of $\cC$ stable under composition and containing identity
morphisms. Then $\N(\cC)_{\cE_1,\dots,\cE_k}$ and
$\N(\cC)^\cart_{\cE_1,\dots,\cE_k}$ can be interpreted as the $k$-fold nerves
in the sense of Fiore and Paoli \cite{FP}*{Definition 2.14} of suitable
$k$-fold categories. More generally, if $\cQ_{ij}$ are sets of squares stable
under composition in both directions such that $(\N(\cC),\cE,\cQ)$ is a
$k$-tiled $\infty$-category, then $\delta^{k\square}_*(\N(\cC),\cE,\cQ)$ is
the $k$-fold nerve of a suitable $k$-fold category.
\end{remark}

\section{Multisimplicial descent}
\label{4ss}

In this section, we study the map of simplicial sets obtained by composing
two directions in a multisimplicial nerve. Unlike in Theorem \ref{0th:main},
the two directions are not subject to the Cartesian restriction. The main
result is Theorem \ref{4th:multisimplicial_descent}, which is a general
criterion for the map to be a categorical equivalence. We then give more
specific sufficient conditions in two important special cases: Theorems
\ref{4co:multisimplicial_descent} and \ref{4pr:descent}. The latter can be
regarded as a generalization of Deligne's result
\cite{SGA4XVII}*{Proposition 3.3.2} (see Remark \ref{4re:deligne}).

In Deligne's theory, a fundamental role is played by the category of
compactifications of a morphism $f$, whose objects are factorizations of $f$
as $p\circ q$, where $p$, $q$ belong respectively to the two classes of
morphisms in question. To properly formulate compactifications of simplices
of higher dimensions, we introduce a bit of notation.

We identify \emph{partially ordered sets} with ordinary categories in which
there is at most one arrow between each pair of objects, by the convention
$p\le q$ if and only if there exists an arrow $p\to q$. For every element
$p\in P$, we identify the overcategory $P_{/p}$ (resp.\ undercategory
$P_{p/}$) with the full partially ordered subset of $P$ consisting of
elements $\leq p$ (resp.\ $\geq p$). For $p,p'\in P$, we identify
$P_{p//p'}$ with the full partially ordered subset of $P$ consisting of
elements both $\geq p$ and $\leq p'$, which is empty unless $p\leq p'$. For
a subset $Q$ of $P$, we write $Q_{p/}=Q\cap P_{p/}$, etc.

\begin{notation}
Let $n\geq0$ be an integer. We consider the bisimplicial set $\Delta^{n,n}$
and the partially ordered set $[n]\times [n]$, related by the natural
isomorphisms of simplicial sets
$\delta_2^*\Delta^{n,n}\simeq\Delta^n\times\Delta^n\simeq\N([n]\times[n])$.
We enumerate their vertices by coordinates $(i,j)$ for $0\leq i,j\leq n$. We
define $\Cpt^n\subseteq\Delta^{n,n}$ to be the bisimplicial subset obtained
by the vertices $(i,j)$ with $0\leq i\leq j\leq n$. We define
$\RCpt^n\subseteq[n]\times [n]$ to be the full partially ordered subset
spanned by $(i,j)$ with $0\leq i\leq j\leq n$. We have
\begin{align*}
\delta_2^*\Cpt^n\simeq \Box^n\subseteq\CCpt^n\colonequals\N(\RCpt^n),
\end{align*}
where we have put $\Box^n\colonequals\bigcup_{k=0}^n\Box_k^n$ and
$\Box_k^n\colonequals\N(\RCpt^n_{(0,k)//(k,n)})$ is the nerve of the full
partially ordered subset of $[n]\times[n]$ spanned by $(i,j)$ with $0\leq
i\leq k\leq j\leq n$.
\end{notation}

Below is the Hasse diagram of $\RCpt^3$, rotated so that the initial object
is shown in the upper-left corner. The dashed box represents $\Box^3_1$,
while bullets represent elements in the image of the diagonal embedding
$[3]\to \RCpt^3$.
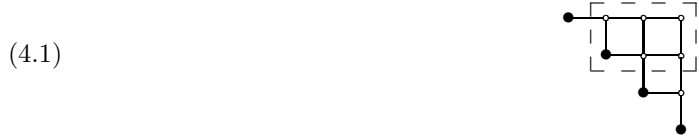
\begin{equation}\label{4eq:Cpt}
\begin{xy}
(0,5)="00"; (5,5)*\cir<1.8pt>{}="01"**\dir{-};
(10,5)*\cir<1.8pt>{}="02"**\dir{-}; (15,5)*\cir<1.8pt>{}="03"**\dir{-};
(5,0)="11"; (10,0)*\cir<1.8pt>{}="12"**\dir{-};
(15,0)*\cir<1.8pt>{}="13"**\dir{-};
(10,-5)="22";
(15,-5)*\cir<1.8pt>{}="23"**\dir{-};
(15,-10)="33"**\dir{-};
"01"; "11"**\dir{-};
"02"; "12"**\dir{-}; "22"**\dir{-};
"03"; "13"**\dir{-}; "23"**\dir{-};
"00"*{\bullet}; "11"*{\bullet}; "22"*{\bullet}; "33"*{\bullet};
(3,-2)="b"; (3,7)**\dir{--}; (17,7)**\dir{--}; (17,-2)**\dir{--}; "b"**\dir{--};
\end{xy}
\end{equation}
Note that the first coordinate is represented vertically and the second one
is represented horizontally.

The following lemma is crucial for our argument, whose proof will be given in
Lemma \ref{6le:cpt_inner}.

\begin{lemma}\label{4le:cpt_inner}
The inclusion $\Box^n\subseteq\CCpt^n$ is inner anodyne.
\end{lemma}

We now review compactifications in ordinary categories.

\begin{definition}\label{4de:comp}
Let $\cC$ be an ordinary category and let $\cE_1$, $\cE_2$ be two sets of
morphisms of $\cC$ containing all identity morphisms. Let $\tau\colon [n]\to
\cC$ be a functor, corresponding to a sequence of morphisms
\[c_0\to c_1\to \dots \to c_n.\]

We define a \emph{compactification of $\tau$} to be a functor $\sigma\colon
\RCpt^n\to \cC$ satisfying the following conditions:
\begin{enumerate}
  \item The functor $\sigma$ carries ``vertical'' morphisms $(i,j)\to
      (i',j)$ of $\RCpt^n$ into $\cE_1$ and ``horizontal'' morphisms
      $(i,j)\to (i,j')$ into $\cE_2$.

  \item The composition $[n]\to \RCpt^n \xto{\sigma} \cC$ is $\tau$. Here
      $[n]\to\RCpt^n$ is the diagonal functor carrying $i$ to $(i,i)$.
\end{enumerate}

Assume that $\cE_\alpha$ is stable under composition for $\alpha=1$ or
$\alpha=2$. The compactifications of $\tau$ can be organized into a category
$\Kpt^\alpha(\tau)$ as follows. Given two compactifications
$\sigma,\sigma'\colon \RCpt^n\to \cC$ of $\tau$, a morphism in
$\Kpt^\alpha(\tau)$ is a natural transformation $\gamma\colon \sigma\to
\sigma'$ satisfying the following conditions:
\begin{enumerate}
  \item For every $(i,j)\in \RCpt^n$, the morphism $\gamma(i,j)\colon
      \sigma(i,j)\to \sigma'(i,j)$ is in $\cE_\alpha$.

  \item The restriction of $\gamma$ to $[n]$ via the diagonal functor is
      $\id_\tau$.
\end{enumerate}
\end{definition}

For $\alpha=1$ (and $n\le 3$), $\Kpt^1(\tau)$ is the category of
compactifications considered by Deligne \cite{SGA4XVII}*{D\'efinition
3.2.5}.

In the language of $2$-marked simplicial sets, we can reformulate the two
conditions (1) in Definition \ref{4de:comp} as follows. Condition (1) in the
definition of compactifications means that the restriction of
$\N(\sigma)\colon \CCpt^n\to \N(\cC)$ to $\Box^n$ induces a map of $2$-marked
simplicial sets $\delta^*_{2+}\Cpt^n\to (\N(\cC),\cE_1,\cE_2)$. Condition (1)
in the definition of morphisms means that the restriction of
$\N(\gamma)\colon \Delta^1\times \RCpt^n\to \N(\cC)$ to $\Delta^1\times
\Box^n$, where $\gamma$ is regarded as a functor $[1]\times \RCpt^n\to \cC$,
induces a map of $2$-marked simplicial sets
$(\Delta^1)^{\sharp^{2}_{\{\alpha\}}}\times \delta^*_{2+}\Cpt^n\to
(\N(\cC),\cE_1,\cE_2)$. See Definition \ref{3de:multimarked_simplicial} for
the notation $(-)^{\sharp^{2}_{\{\alpha\}}}$.

We now define compactifications in $\infty$-categories, and more generally
in simplicial sets. Besides the need to deal with simplices of higher
dimensions, the definition is more complicated in two other ways: we
consider an extra set $K$ of ``directions'' and we consider restrictions not
only on edges, but also on squares, which leads to the use of multi-tiled
simplicial sets.

\begin{definition}\label{4de:compactification}
Let $K$ be a set and let $(X,\cT)$ be a $(\{1,2\}\coprod K)$-tiled simplicial
set. For $L\subseteq K$, integers $n,n_k\geq 0$ ($k\in K$), a map $\tau\colon
\Delta^{n,n_k\res k\in K}_L\to \delta_*^{\{0\}\amalg K}X$, and
$\alpha\in\{1,2\}\coprod K$, we define
$\Komp^\alpha(\tau)=\Komp^\alpha_{(X,\cT)}(\tau)$, the \emph{$\alpha$-th
simplicial set of compactifications of $\tau$}, to be the limit of the
diagram
\begin{equation}\label{4eq:limit}
\xymatrix{&&\epsilon^{\{1,2\}\amalg K}_\alpha
\Map(\Cpt^n\boxtimes\Delta^{n_k\res {k\in K}}_L,\delta^{(\{1,2\}\amalg K)\square}_*(X,\cT))\ar@{^{(}->}[d]^g\\
&\Map(\CCpt^n\times\Delta^{[n_k]_{k\in K}}_L,X)
\ar[r]^-{\RES_1}\ar[d]^-{\RES_2}
& \Map(\Box^n\times\Delta^{[n_k]_{k\in K}}_L,X)\\
\{\tau\}\ar@{^{(}->}[r]& \Map(\Delta^{[n,n_k]_{k\in K}}_L,X)}
\end{equation}
in the
category $\Sset$ of simplicial sets, where
\begin{itemize}
  \item we regard $\tau$ as a map $\Delta^{[n,n_k]_{k\in K}}_L\to X$,
      hence a vertex of $\Map(\Delta^{[n,n_k]_{k\in K}}_L,X)$;

  \item $\RES_1$ is induced by the inclusion $\Box^n\subseteq \CCpt^n$;

  \item $\RES_2$ is induced by the diagonal map $\Delta^n\to \CCpt^n$;
      and

  \item $g$ is the composition of maps
      \begin{align*}
      \epsilon^{\{1,2\}\amalg K}_\alpha\Map(\Cpt^n\boxtimes\Delta^{n_k\res {k\in K}}_L,\delta^{(\{1,2\}\amalg
      K)\square}_*(X,\cT))&\hookrightarrow\epsilon^{\{1,2\}\amalg K}_\alpha
      \Map(\Cpt^n\boxtimes\Delta^{n_k\res {k\in K}}_L,\delta^{\{1,2\}\amalg K}_*X)\\
      &\simeq\Map(\Box^n\times\Delta^{[n_k]_{k\in K}}_L, X),
      \end{align*}
      where the isomorphism is the adjunction formula of Remark
      \ref{3re:adjunction}.
\end{itemize}

If $(X,\cE)$ is a $(\{1,2\}\coprod K)$-marked simplicial set, then we put
$\Komp^\alpha_{X,\cE}(\tau)=\Komp^\alpha_{\sfW(X,\cE)}(\tau)$, where $\sfW$
is the functor in Notation \ref{3no:cartesian}. We put
$\Komp^\alpha(\tau)_L=\Komp^\alpha(\tau)$ if $\alpha\not\in L$, and
$\Komp^\alpha(\tau)_L=\Komp^\alpha(\tau)^{op}$ if $\alpha\in L$.
\end{definition}

For brevity, we sometimes write $I$ for $\{1,2\}\coprod K$.

\begin{remark}\label{4re:g}
Let us give a more explicit description of $g$ in \eqref{4eq:limit}. To
simplify notation, we let $Y$ denote the source of $g$. We let
$\iota_\alpha\colon \{1\}\to I$ denote the map with image $\alpha$. For any
simplicial set $S$, we have isomorphisms
\begin{align*}
\Hom_{\Sset}(S,Y)&\simeq \Hom_{\Sset[I]}((\del^{\iota_\alpha})^*S,\Map(\Cpt^n\boxtimes\Delta^{n_k\res {k\in K}}_L,\delta^{I\square}_*(X,\cT)))\\
&\simeq \Hom_{\Sset[I]}((\del^{\iota_\alpha})^*S\times (\Cpt^n\boxtimes\Delta^{n_k\res {k\in K}}_L),\delta^{I\square}_*(X,\cT))\\
&\simeq \Hom_{\Sset^{I\square}}(\delta_{I\square}^*(\del^{\iota_\alpha})^*S\times
\delta_{I\square}^*(\Cpt^n\boxtimes\Delta^{n_k\res {k\in K}}_L),(X,\cT))\\
&\simeq \Hom_{\Sset^{I\square}}(\sfV(S^{\sharp^I_{\{\alpha\}}})\times \sfW\delta_{I+}^*(\Cpt^n\boxtimes\Delta^{n_k\res {k\in K}}_L),(X,\cT)),
\end{align*}
where $\sfV$ and $\sfW$ are the functor in Notation \ref{3no:cartesian}.
Here in the last step we have used the isomorphisms
\[\delta_{I\square}^*(\del^{\iota_\alpha})^*S\simeq \sfV(S^{\sharp^I_{\{\alpha\}}}),
\quad \delta_{I\square}^*(\Cpt^n\boxtimes\Delta^{n_k\res {k\in K}}_L)\simeq \sfW\delta_{I+}^*(\Cpt^n\boxtimes\Delta^{n_k\res {k\in K}}_L).
\]
We define $\{\cF_\beta\}_{\beta\in I}$ by the isomorphism
\[\delta_{I+}^*(\Cpt^n\boxtimes\Delta^{n_k\res {k\in K}}_L)\simeq
(\square^n\times \Delta^{[n_k]_{k\in K}}_L,\{\cF_\beta\}_{\beta \in I}).
\]
In other words, $\cF_\beta$ is the set of edges of
$\epsilon^I_\beta(\Cpt^n\boxtimes\Delta^{n_k\res {k\in K}}_L)$, for all
$\beta\in I$. Then,
\begin{itemize}
  \item A vertex of $Y$ is precisely a map of $I$-marked simplicial sets
      $\delta_{I\square}^*(\Cpt^n\boxtimes\Delta^{n_k\res {k\in K}}_L)\to
      (X,\cT)$. In other words, a map $\sigma\colon \square^n\times
      \Delta^{[n_k]_{k\in K}}_L\to X$ is a vertex of $Y$ if and only if
      it carries $\cF_\beta*\cF_{\beta'}$ into $\cT_{\beta\beta'}$ for
      all $\beta,\beta'\in I$ with $\beta\neq \beta'$. As we observed in
      Remark \ref{3re:multitile}, the condition implies that $\sigma$
      carries $\cF_\beta$ into $\cT_\beta$ for all $\beta \in I$.

  \item Given vertices $\sigma$, $\sigma'$ of $Y$, an edge $\gamma\colon
      \sigma\to \sigma'$ of $\Map(\square^n\times\Delta^{[n_k]_{k\in
      K}}_L,X)$ is an edge of $Y$ if and only if, for every $\beta\in I$,
      $\beta\neq \alpha$ and for every square
      \begin{equation}\label{4eq:Komp}
      \xymatrix{y'\ar[d]\ar[r] & y\ar[d]\\x'\ar[r] & x}
      \end{equation}
      in $\cF_\alpha*\cF_\beta$, with vertical arrows in $\cF_\alpha$ and
      horizontal arrows in $\cF_\beta$, $\gamma$ carries the
      square
      \begin{equation}\label{4eq:Komp2}
      \xymatrix{(0,y')\ar[d]\ar[r] & (0,y)\ar[d]\\(1,x')\ar[r] & (1,x)}
      \end{equation}
      to a square in $\cT_{\alpha\beta}$. Here we have regarded $\gamma$
      as a map $\Delta^1\times (\square^n\times \Delta^{[n_k]_{k\in
      K}}_L)\to\cC$. We note two special cases of the condition:
      \begin{enumerate}
        \item For every edge $y\to x$ in $\cF_\alpha$, $\gamma$
            carries $(0,y)\to (1,x)$ to an edge in $\cT_\alpha$.

        \item For every $\beta\in I$, $\beta\neq \alpha$ and for
            every edge $x'\to x$ in $\cF_\beta$, $\gamma$ carries the
            square
            \[\xymatrix{(0,x')\ar[d]\ar[r] &
            (0,x)\ar[d]\\(1,x')\ar[r] & (1,x)}
            \] to a square in $\cT_{\alpha\beta}$.
      \end{enumerate}
      If $\cT_{\alpha\beta}$ is stable under composition in the first
      direction for every $\beta$, then Condition (2) is also a
      sufficient condition for $\gamma$ to be an edge of $Y$.

  \item For $m\ge 2$, an $m$-simplex $\gamma$ of
      $\Map(\square^n\times\Delta^{[n_k]_{k\in K}}_L,X)$ is an
      $m$-simplex of $Y$ if and only if each edge of $\gamma$ is an edge
      of $Y$.
\end{itemize}
In particular, $g$ satisfies the (unique) right lifting property with
respect to $\partial \Delta^m\subseteq \Delta^m$ for $m\ge 2$.
\end{remark}

\begin{remark}
In the situation of Definition \ref{4de:comp}, we have a canonical
isomorphism $\Komp^\alpha_{\N(\cC),\cE_1,\cE_2}(\tau)\simeq
\N(\Kpt^\alpha(\tau))$. We will see in Lemma \ref{4le:infinity_category} that
the simplicial set $\Komp^\alpha_{(X,\cT)}$ is an $\infty$-category under
mild hypotheses.
\end{remark}

\begin{remark}\label{4re:fibration}
We let $D^n$ denote the intersection of $\square^n$ and the diagonal
embedding $\Delta^n\to \CCpt^n$. Then $D^n$ is the disjoint union of $n+1$
points. Note that the diagram \eqref{4eq:limit} can be completed into a
commutative diagram
\begin{align*}
\xymatrix{
\Map(\CCpt^n\times\Delta^{[n_k]_{k\in K}}_L,X)\ar[rd]^-{\RES_4}\ar@/^1.5pc/[rrd]^-{\RES_1}\ar@/_1.5pc/[ddr]_-{\RES_2}
&&\epsilon^I_\alpha\Map(\Cpt^n\boxtimes \Delta^{n_k\res k\in K}_L, \delta^{I\square}_*(X,\cT))\ar[d]^g\\
&\Map((\Box^n\coprod_{D^n}
\Delta^n)\times\Delta^{[n_k]_{k\in K}}_L,X)\ar[r]\ar[d]
& \Map(\Box^n\times\Delta^{[n_k]_{k\in K}}_L,X)\ar[d]^-{\RES_3}\\
\{\tau\}\ar[r]& \Map(\Delta^{n}\times\Delta^{[n_k]_{k\in K}}_L,X)\ar[r] &
\Map(D^n\times\Delta^{[n_k]_{k\in K}}_L,X),}
\end{align*}
where the lower right square is a pullback. Here the maps in the lower right
square (including $\RES_3$) and $\RES_4$ are obvious restrictions. If $X$ is
an $\infty$-category, then $\RES_i$, $2\le i\le 4$ are Cartesian fibrations
(and coCartesian fibrations) by \cite{Lu1}*{Proposition 3.1.2.1} and
$\RES_1$ is a trivial Kan fibration by Lemma \ref{4le:cpt_inner} and
\cite{Lu1}*{Corollaries 2.3.2.4, 2.3.2.5}. Moreover, $\RES_1$ is an
isomorphism if $X$ is isomorphic to the nerve of an ordinary category.
\end{remark}

\begin{remark}
We have introduced $K$ in the definition mainly for convenience. In the case
where $\alpha\in\{1,2\}$, which is our main case of interest, we could reach
the same generality without $K$. In fact, we can define a $\{1,2\}$-tiled
simplicial set $(X',\cT')$, where $X'$ is the full simplicial subset of
$\Map(\Delta^{[n_k]_{k\in K}}_L,X)$ spanned by maps corresponding to maps
$\Delta^{n_k\res {k\in K}}_L\to \delta_*^{K\Box}(X,\cT_K)\subseteq
\delta_*^K X$ (where $\cT_K$ denotes the $K$-tiling induced by $\cT$), with
the following property: If $\tau$ defines an $n$-simplex $\tau'$ of $X'$,
then we have an isomorphism $\Komp^\alpha_{(X,\cT)}(\tau)\simeq
\Komp^\alpha_{(X',\cT')}(\tau')$; otherwise $\Komp^\alpha_{(X,\cT)}(\tau)$
is empty.
\end{remark}

Note that by Remark \ref{3re:adjunction}, the map $g$ is also equal to the
composition
\begin{align*}
&\quad\epsilon^{\{1,2\}\amalg K}_\alpha\Map(\Cpt^n\boxtimes\Delta^{n_k\res {k\in K}}_L,\delta^{(\{1,2\}\amalg K)\square}_*(X,\cT))\\
&\xto{\delta^*_{\{1,2\}\amalg K}}\Map(\Box^n\times\Delta^{[n_k]_{k\in K}}_L,\delta_{\{1,2\}\amalg
K}^*\delta^{(\{1,2\}\amalg K)\square}_*(X,\cT))\\
&\hookrightarrow\Map(\Box^n\times\Delta^{[n_k]_{k\in K}}_L,\delta_{\{1,2\}\amalg
K}^*\delta^{\{1,2\}\amalg K}_*X)\\
&\to\Map(\Box^n\times\Delta^{[n_k]_{k\in K}}_L,X),
\end{align*}
where the last map is induced by the counit map. We consider the composition
\begin{align}\label{4eq:compactification_phi}
\phi(\tau)\colon \Komp^\alpha(\tau)_L \to \epsilon^{I}_\alpha\op^{I}_L\Map(\Cpt^n\boxtimes\Delta^{n_k\res {k\in K}}_L,\delta^{I\square}_*(X,\cT))
\xto{\delta^*_{I}} \Map(\square^n\times\Delta^{[n_k]_{k\in K}},\delta_{I,L}^*\delta^{I\square}_*(X,\cT)),
\end{align}
which will be used in the proof of Theorem \ref{4th:multisimplicial_descent}
below.

\begin{remark}\label{4re:phi}
By construction, the composition
\[\Komp^\alpha(\tau)_L\xto{\phi(\tau)}
\Map(\Box^n\times\Delta^{[n_k]_{k\in K}},\delta_{I,L}^*\delta^{I\square}_*(X,\cT))\to \Map(D^n\times
\Delta^{[n_k]_{k\in K}},\delta_{I,L}^*\delta^{I\square}_*(X,\cT)),
\]
where the second map is induced by the inclusion $D^n\subseteq \square^n$
(see Remark \ref{4re:fibration} for the notation), is constant of value
$\delta_{I,L}^*\tau_0$, where
\[\tau_0\colon \delta^2_*(D^n)\boxtimes
\Delta^{n_k\res{k\in K}}_L\to \delta^{I}_*X
\]
is the restriction of $\tau$. If $\Komp^\alpha(\tau)$ is nonempty, then
$\tau_0$ factorizes through $\delta^{I\square}_*(X,\cT)$.
\end{remark}

Next we consider $(\{0\}\coprod K)$-tilings. Let $(X,\cT')$ be a
$(\{0\}\coprod K)$-tiled simplicial set. For brevity we sometimes write $J$
for $\{0\}\coprod K$. For $L\subseteq K$ and $\alpha'\in J$, we have the
commutative diagram
\begin{equation}\label{4eq:zeroK}
\xymatrix{\epsilon^J_{\alpha'}
\Map(\CCpt^n\boxtimes\Delta^{n_k\res {k\in K}}_L,\delta^{J\square}_*(X,\cT')) \ar@{^{(}->}[rr] \ar[d]_{\delta^*_J}
&&\epsilon^J_{\alpha'}
\Map(\CCpt^n\boxtimes\Delta^{n_k\res {k\in
K}}_L,\delta^J_*X)\ar[d]^-{\simeq}\\
\Map(\CCpt^n\times\Delta^{[n_k]_{k\in K}}_L,\delta_J^*\delta^{J\square}_*(X,\cT')) \ar@{^{(}->}[r] &
\Map(\CCpt^n\times\Delta^{[n_k]_{k\in K}}_L,\delta^*_J\delta_*^J X)\ar[r]&
\Map(\CCpt^n\times\Delta^{[n_k]_{k\in K}}_L, X)}
\end{equation}
by Remark \ref{3re:adjunction}. This is similar to the situation of the map
$g$ in Definition \ref{4de:compactification}.

To compare the restricted multisimplicial nerves of $(X,\cT)$ and of
$(X,\cT')$, we make some assumptions.

\begin{assumption}\label{4as:composition}
Let $(X,\cT)$ be a $(\{1,2\}\coprod K)$-tiled simplicial set and let
$(X,\cT')$ be a $(\{0\}\coprod K)$-tiled simplicial set. Consider the
following assumptions:
\begin{enumerate}
  \item For $\sigma\colon \Delta^2\to X$ with $\sigma\circ d^2_0\in
      \cT_1$, $\sigma\circ d^2_2\in \cT_2$, we have $\sigma\circ d^2_1\in
      \cT'_0$.

  \item For $k\in K$ and $\sigma \colon \Delta^2\times \Delta^1\to X$
      satisfying $\sigma\circ (d^2_0\times \id)\in \cT_{1k}$,
      $\sigma\circ (d^2_2\times \id)\in \cT_{2k}$, we have $\sigma\circ
      (d^2_1\times \id)\in \cT'_{0k}$.

  \item For $k\in K$, we have $\cT_k\subseteq \cT'_k$. For distinct
      elements $k,k'\in K$, we have $\cT_{kk'}\subseteq \cT'_{kk'}$.
\end{enumerate}
\end{assumption}

Note that (2) implies (1) if $K$ is nonempty.

\begin{remark}\label{4re:assump}
Assumption (1) implies $\cT_1,\cT_2\subseteq \cT'_0$. Conversely, if we have
$\cT_1,\cT_2\subseteq \cT'_0$ and $\cT'_0$ is stable under composition, then
Assumption (1) holds. Similarly, Assumption (2) implies
$\cT_{1k},\cT_{2k}\subseteq \cT'_{0k}$. Conversely, if we have
$\cT_{1k},\cT_{2k}\subseteq \cT'_{0k}$ and $\cT'_{0k}$ is stable under
composition in the first direction, then Assumption (2) holds.
\end{remark}

We consider the maps $\mu_0\colon \{1,2\}\to \{0\}$ and $\mu=\mu_0\coprod
\id_K \colon \{1,2\}\coprod K\to \{0\}\coprod K$. For brevity, we sometimes
write $I$ for $\{1,2\}\coprod K$ and $J$ for $\{0\}\coprod K$.

\begin{lemma}\label{4le:multisimp}
Suppose that Assumption \ref{4as:composition} is satisfied. Then
\begin{enumerate}
  \item The isomorphism $\delta_*^{\{1,2\}\amalg K}X\simeq
      (\del^\mu)_*\delta_*^{\{0\}\amalg K} X$ induces an inclusion
      \[\delta_*^{(\{1,2\}\amalg K)\square}(X,\cT)\subseteq
      (\del^\mu)_*\delta_*^{(\{0\}\amalg K)\square}(X,\cT').\]

  \item The pullback of $g$ by $\RES_1$ in Definition
      \ref{4de:compactification} factorizes through the upper-left corner
      of the diagram \eqref{4eq:zeroK} with $\alpha'=\mu(\alpha)$.
\end{enumerate}
\end{lemma}

\begin{proof}
(1) We have
\[\delta^*_{I\square}\Delta^{n_k\mid k\in
I}=\sfW\delta^*_{I+}\Delta^{n_k\mid k\in I},\quad
\delta^*_{J\square}(\del^\mu)^*\Delta^{n_k\mid k\in I}\simeq \sfW\delta^*_{J+}(\Delta^{[n_1,n_2]}\boxtimes\Delta^{n_k\mid k\in K}).
\]
We let $\cG_\beta$ denote the set of edges of
$\epsilon^I_\beta\Delta^{n_k\mid k\in I}$ for $\beta\in I$, and let $\cG_0$
denote the set of edges of
$\epsilon^J_0(\Delta^{[n_1,n_2]}\boxtimes\Delta^{n_k\mid k\in K})$. Then we
have
\[\delta^*_{I+}\Delta^{n_k\mid k\in I}\simeq (\Delta^{[n_k]_{k\in
I}},\{\cG_\beta\}_{\beta\in I}),\quad \delta^*_{J+}(\Delta^{[n_1,n_2]}\boxtimes\Delta^{n_k\mid k\in K})\simeq (\Delta^{[n_k]_{k\in
I}},\{\cG_\beta\}_{\beta\in J}).
\]
Thus an $(n_k)_{k\in I}$-simplex of $\delta_*^{I\square}(X,\cT)$ is given by
a map $\sigma\colon \Delta^{[n_k]_{k\in I}}\to X$ carrying $\cG_\beta$ into
$\cT_\beta$ for all $\beta\in I$ and carrying $\cG_\beta*\cG_{\beta'}$ into
$\cT_{\beta\beta'}$ for all $\beta,\beta'\in I$, $\beta\neq \beta'$.

Let us show first that $\sigma$ carries $\cG_\beta$ into $\cT'_\beta$ for all
$\beta\in J$. For $\beta\in K$, this follows from the assumption
$\cT_\beta\subseteq \cT'_\beta$. An edge $e$ in $\cG_0$ has the form
$(i,j,a)\to (i',j',a)$, where $a=(a_k)_{k\in K}$. Consider the $2$-simplex
\[\xymatrix{(i,j,a)\ar[r]^-{e''}\ar[rd]_-e&(i,j',a)\ar[d]^-{e'}\\&(i',j',a)}\]
of $\Delta^{[n_k]_{k\in I}}$, where $e'$ is in $\cG_1$ and $e''$ is in
$\cG_2$. By Assumption \ref{4as:composition} (1), $\sigma(e)$ is in
$\cT'_0$.

Next we show that $\sigma$ carries $\cG_{\beta}*\cG_{\beta'}$ into
$\cT'_{\beta\beta'}$ for all $\beta,\beta'\in J$, $\beta\neq \beta'$. For
$\beta,\beta'\in K$, this follows from the assumption
$\cT_{\beta\beta'}\subseteq \cT'_{\beta\beta'}$. It remains to show that
$\sigma$ carries $\cG_0*\cG_\beta$ into $\cT'_{0\beta}$ for $\beta\in K$.
Every square $\varsigma$ in $\cG_0*\cG_\beta$ can be extended to a map
$\Delta^2\times \Delta^1\to \Delta^{[n_k]_{k\in I}}$ as shown by the diagram
\[\xymatrix{(i,j,b)\ar[r]\ar[d] & (i,j,a)\ar[d]\\
(i,j',b)\ar[r]\ar[d] & (i,j',a)\ar[d]\\
(i',j',b)\ar[r] & (i',j',a)}
\]
with $\varsigma$ as the outer square. The upper square is in
$\cG_2*\cG_\beta$ and the lower square is in $\cG_1*\cG_\beta$. Thus, by
Assumption \ref{4as:composition} (2), $\sigma(\varsigma)$ is in
$\cT'_{0\beta}$.

(2) We let $Y'$ denote the pullback of $g$ by $\RES_1$ and let $Z$ denote
the upper-left corner of the diagram \eqref{4eq:zeroK}. We adopt the
notation of Remark \ref{4re:g}. Note that for $\beta\in K$, $\cF_\beta$ can
be identified with the set of edges of $\epsilon^J_\beta(\CCpt^n\boxtimes
\Delta^{n_k\mid k\in K}_L)$. We let $\cF_0$ denote the set of edges of
$\epsilon^J_0(\CCpt^n\boxtimes \Delta^{n_k\mid k\in K}_L)$. Then
\[\delta^*_{J+}(\CCpt^n\boxtimes \Delta^{n_k\mid k\in K}_L)\simeq
(\CCpt^n\times \Delta^{[n_k]_{k\in K}}_L,\{\cF_\beta\}_{\beta \in J}).\]
Note that $Z$ admits a description similar to the description of $Y$ in
Remark \ref{4re:g}. In particular, for $m\ge 2$, $Z\hookrightarrow
\Map(\CCpt^n\times \Delta^{[n_k]_{k\in K}}_L,X)$ has the right lifting
property with respect to $\partial\Delta^m\subseteq \Delta^m$. Thus it
suffices to check $Y'\subseteq Z$ on the level of vertices and edges.

Let $\sigma\colon \CCpt^n\times \Delta^{[n_k]_{k\in K}}_L\to X$ be a vertex
of $Y'$. To show that $\sigma$ is a vertex of $Z$, we need to check that
$\sigma$ carries $\cF_\beta$ to $\cT'_\beta$ for all $\beta\in J$ and
carries $\cF_\beta*\cF_{\beta'}$ to $\cT'_{\beta\beta'}$ for all
$\beta,\beta'\in J$, $\beta\neq \beta'$. The proof is similar to that of
(1). Note that for every edge $(i,j)\to (i',j')$ of $\CCpt^n$, $(i,j')$ is a
vertex of $\CCpt^n$.

Let $\gamma$ be an edge of $Y'$, regarded as a map $\Delta^1\times
(\CCpt^n\times \Delta^{[n_k]_{k\in K}}_L)\to X$. To show that $\gamma$ is an
edge of $Z$, we first check that for every edge $y\to x$ of
$\cF_{\mu(\alpha)}$, $\gamma$ carries the edge $e\colon (0,y)\to (1,x)$ to
an edge in $\cT'_{\mu(\alpha)}$. If $\alpha\in K$, then this follows from
the assumption $\cT_\alpha\subseteq \cT'_\alpha$. If $\alpha\in \{1,2\}$,
then $e$ can be completed into a $2$-simplex of the form
\[\xymatrix{(0,i,j,a)\ar[r]^-{e''}\ar[rd]_-e & (\alpha-1,i,j',a)\ar[d]^-{e'}\\&(1,i',j',a).}\]
Since $\gamma(e')$ is in $\cT_1$ and $\gamma(e'')$ is in $\cT_2$, we have
$\gamma(e)\in \cT'_0$ by Assumption \ref{4as:composition} (1). Finally we
check that for every $\beta\in J$, $\beta\neq \mu(\alpha)$ and every square
of the form \eqref{4eq:Komp} in $\cF_{\mu(\alpha)}*\cF_\beta$ with vertical
arrows in $\cF_{\mu(\alpha)}$ and horizontal arrows in $\cF_\beta$, $\gamma$
carries the square \eqref{4eq:Komp2} to a square in
$\cT'_{\mu(\alpha)\beta}$. If $\alpha,\beta\in K$, then this follows from
the assumption $\cT_{\alpha\beta}\subseteq \cT'_{\alpha\beta}$. In the
remaining cases we apply Assumption \ref{4as:composition} (2). If $\beta=0$,
we factorize the square horizontally. If $\alpha\in \{1,2\}$, we factorize
the square vertically, with the first component of the middle row given by
$\alpha-1$.
\end{proof}

\begin{construction}
The main result of this section, Theorem
\ref{4th:multisimplicial_descent} below, is about the composition
\begin{multline}\label{4eq:composition}
\delta_{\{1,2\}\amalg K,L}^*\delta^{(\{1,2\}\amalg K)\square}_*(X,\cT) \simeq
\delta_{\{0\}\amalg K,L}^*(\del^\mu)^*\delta^{(\{1,2\}\cup
K)\square}_*(X,\cT) \\
\hookrightarrow\delta_{\{0\}\amalg
K,L}^*(\del^\mu)^*(\del^\mu)_*\delta^{(\{0\}\amalg K)\square}_*(X,\cT')\to
\delta_{\{0\}\amalg K,L}^*\delta^{(\{0\}\amalg K)\square}_*(X,\cT'),
\end{multline}
where the inclusion in the middle is given by Lemma \ref{4le:multisimp} (1)
and the last map is the counit map. An $n$-simplex of the left hand side of
\eqref{4eq:composition} corresponds to a map $\Delta^n\times
\Delta^n\times(\Delta^n)^K\to X$. The map \eqref{4eq:composition} carries it
to the $n$-simplex corresponding to the composition $\Delta^n\times
(\Delta^n)^K\xto{\diag\times \id_{(\Delta^n)^K}} \Delta^n\times
\Delta^n\times(\Delta^n)^K\to X$, where $\diag\colon \Delta^n\to
\Delta^n\times \Delta^n$ is the diagonal map.

For any map $\tau\colon \Delta^{n,n_k\res k\in K}_L\to \delta_*^{\{0\}\amalg
K} X$, we consider the composition
\begin{align}\label{4eq:compactification_psi}
\psi(\tau)\colon \Komp^\alpha(\tau)_L \to \epsilon^{J}_{\mu(\alpha)}\op^{J}_L
\Map(\CCpt^n\boxtimes\Delta^{n_k\res {k\in K}}_L,\delta^{J\square}_*(X,\cT'))
\xto{\delta^*_{J}}
\Map(\CCpt^n\times\Delta^{[n_k]_{k\in K}},\delta_{J,L}^*\delta^{J\square}_*(X,\cT')),
\end{align}
where the first map is given by Lemma \ref{4le:multisimp} (2). We have a
commutative diagram
\[\xymatrix{\Komp^\alpha(\tau)_L \ar[r]^-{\psi(\tau)}\ar[d]_{\phi(\tau)}&
\Map(\CCpt^n\times\Delta^{[n_k]_{k\in K}},\delta_{J,L}^*\delta^{J\square}_*(X,\cT'))\ar[d]\\
\Map(\square^n\times\Delta^{[n_k]_{k\in K}},\delta_{I,L}^*\delta^{I\square}_*(X,\cT)\ar[r] &
\Map(\square^n\times\Delta^{[n_k]_{k\in K}},\delta_{J,L}^*\delta^{J\square}_*(X,\cT')),}
\]
where $\phi(\tau)$ is defined in \eqref{4eq:compactification_phi}, the lower
horizontal arrow is induced by \eqref{4eq:composition}, and the right
vertical arrow is the obvious restriction.
\end{construction}

\begin{remark}\label{4re:psi}
By construction, the composition
\[\Komp^\alpha(\tau)_L \xto{\psi(\tau)}\Map(\CCpt^n\times\Delta^{[n_k]_{k\in K}},\delta_{J,L}^*\delta^{J\square}_*(X,\cT'))
\to \Map(\Delta^{[n,n_k]_{k\in K}},\delta_{J,L}^*\delta^{J\square}_*(X,\cT')),
\]
where the second map is induced by the diagonal embedding $\Delta^n\to
\CCpt^n$, is constant of value $\delta_{J,L}^*\tau$. If $\Komp^\alpha(\tau)$
is nonempty, then $\tau$ factorizes through $\delta_*^{J\square}(X,\cT')$.
\end{remark}

\begin{theorem}[Multisimplicial descent]\label{4th:multisimplicial_descent}
Let $K$ be a set and let $\alpha\in\{1,2\}\coprod K$ be an element. Let
$(X,\cT)$ be a $(\{1,2\}\coprod K)$-tiled simplicial set and let $(X,\cT')$
be a $(\{0\}\coprod K)$-tiled simplicial set, satisfying Assumption
\ref{4as:composition}. We assume that $\Komp^\alpha_{(X,\cT)}(\tau)$ is
weakly contractible for every $n\ge 0$ and every $(n,n_k)_{k\in K}$-simplex
$\tau$ of $\delta^{(\{0\}\amalg K)\square}_*(X,\cT')$ with $n_k=n$. Then, for
every subset $L\subseteq K$, the map
\[
f\colon \delta_{\{1,2\}\amalg K,L}^*\delta^{(\{1,2\}\amalg K)\square}_*(X,\cT)
\to \delta_{\{0\}\amalg K,L}^*\delta^{(\{0\}\amalg K)\square}_*(X,\cT'),
\]
composition of \eqref{4eq:composition}, is a categorical equivalence.
\end{theorem}

Note that the assumption that the simplicial sets $\Komp^\alpha(\tau)$ for
those $\tau$ are nonempty implies that $\cT_k=\cT'_k$ for all $k\in K$, and
$\cT_{kk'}=\cT'_{kk'}$ for all $k,k'\in K$ with $k\neq k'$.

In the case $K=\emptyset$ and $\cT'_0=X_1$, Assumption \ref{4as:composition}
is clearly satisfied and the theorem takes the following form.

\begin{corollary}\label{4co:multi2}
Let $\alpha$ be $1$ or $2$. Let $(X,\cT)$ be a $2$-tiled simplicial set such
that $\Komp^\alpha_{(X,\cT)}(\tau)$ is weakly contractible for every simplex
$\tau$ of $X$. Then the map
\[f\colon \delta^*_2\delta_*^{2\square}(X,\cT)\to X\]
induced by the counit map $\delta^*_2\delta_*^2 X\to X$ is a categorical
equivalence.
\end{corollary}

\begin{proof}[Proof of Theorem \ref{4th:multisimplicial_descent}]
We let $Y$ and $Z$ denote the source and target, respectively, of the map
$f$ in the statement of the theorem. Consider a commutative diagram
\[\xymatrix{Y\ar[r]^-v\ar[d]_f & \Fun(\Delta^l,\cD)\ar[d]^p\\
Z\ar[r]^-w & \Fun(\partial \Delta^l,\cD)}\]
as in Lemma \ref{1le:categorical_equivalence}. Let $\sigma$ be an
$n$-simplex of $Z$, corresponding to a map $\tau\colon \Delta^{n,n_k\res
k\in K}_L\to \delta^{J\square}_*(X,\cT')$, where $n_k=n$. Consider the
commutative diagram
\begin{align}\label{4eq:descent}
\xymatrix{\cN(\sigma)\ar[d]\ar[r] & \Fun(\Delta^l\times\CCpt^n\times\Delta^{[n_k]_{k\in K}},\cD)
\ar[d]^-{\RES_1}\ar[rr]^-{\RES_2} && \Fun(\Delta^l\times \Delta^n,\cD)\ar[d]\ar@/^4pc/[dd]^{\RES_4}\\
\Komp^\alpha(\tau)_L\ar[r]^-{h}\ar[rd]_{v_*\phi(\tau)} &  \Fun(H\times\Delta^{[n_k]_{k\in
K}},\cD)\ar[r]\ar[d]&\Fun(\partial \Delta^l\times \CCpt^n\times\Delta^{[n_k]_{k\in
K}},\cD)\ar[r]^-{\RES_2} & \Fun(\partial\Delta^l\times\Delta^n,\cD)\\
&\Fun(\Delta^l\times \square^n\times\Delta^{[n_k]_{k\in
K}},\cD)\ar[rr]^-{\RES_3}&&\Fun(\Delta^l\times D^n,\cD).}
\end{align}
In the above diagram,
\begin{itemize}
  \item $\RES_1$ is induced by
      \[j\colon
      H=\Delta^l\times\Box^n\coprod_{\partial \Delta^l\times\Box^n
      } \partial\Delta^l\times\CCpt^n \hookrightarrow \Delta^l \times
      \CCpt^n;\]

  \item $h$ is the amalgamation of $v_*\phi(\tau)$ and $w_*\psi(\tau)$,
      where
      \[v_*\phi(\tau)\colon\Komp^\alpha(\tau)_L\to\Fun(\Delta^l\times\Box^n\times\Delta^{[n_k]_{k\in K}},\cD)\]
      is the composition of \eqref{4eq:compactification_phi} and the map
      induced by $v$, and
      \[w_*\psi(\tau)\colon\Komp^\alpha(\tau)_L\to\Fun(\partial\Delta^l\times\CCpt^n\times\Delta^{[n_k]_{k\in K}},\cD)\]
      is the composition of \eqref{4eq:compactification_psi} and the map
      induced by $w$;

  \item $\cN(\sigma)$ is defined so that the upper left square is a
      pullback square;

  \item the two maps $\RES_2$ are both induced by the diagonal embedding
      $\Delta^n\subseteq \CCpt^n\times \Delta^{[n_k]_{k\in K}}$;

  \item $D^n$ is as in Remark \ref{4re:fibration};

  \item $\RES_3$ is induced by the diagonal embedding $D^n\subseteq
      \square^n\times \Delta^{[n_k]_{k\in K}}$; and

  \item $\RES_4$ is induced by the inclusion $D^n\subseteq \Delta^n$;

  \item the unnamed arrows in the middle column and in the upper right
      square are obvious restrictions.
\end{itemize}

By Lemma \ref{4le:cpt_inner} and \cite{Lu1}*{Corollaries 2.3.2.4, 2.3.2.5},
the map $j\times\id_{\Delta^{[n_k]_{k\in K}}}$ is inner anodyne, and
consequently $\RES_1$ is a trivial Kan fibration. Thus $\cN(\sigma)$ is
weakly contractible.

We let $\Phi(\sigma)\colon\cN(\sigma)\to\Fun(\Delta^l\times \Delta^n,\cD)$
denote the composition of the upper horizontal arrows in \eqref{4eq:descent}.
We let $\sigma_0\colon D^n\to Z$ denote the restriction of $\sigma$. Since
$f$ induces a bijection on vertices, $\sigma_0$ factorizes uniquely through a
map $D^n\to Y$, which we still denote by $\sigma_0$. By Remark \ref{4re:phi},
$\RES_3\circ v_*\phi(\tau)$ is constant of value $v(\sigma_0)$. It follows
that $\RES_4\circ \Phi(\sigma)$ is constant of value $v(\sigma_0)$. In
particular, $\Phi(\sigma)$ induces a map
\[\cN(\sigma)^\sharp
\times(\Delta^n)^\flat\to\Fun(\Delta^l,\cD)^\flat\subseteq
\Fun(\Delta^l,\cD)^{\natural}.
\]
Thus $\Phi(\sigma)$ induces a map $\cN(\sigma)\to
\Map^\sharp((\Delta^n)^\flat, \Fun(\Delta^l,\cD)^{\natural})$, which we
still denote by $\Phi(\sigma)$. This construction is functorial in $\sigma$,
giving rise to a morphism $\Phi\colon \cN\to \Map[Z,\Fun(\Delta^l,\cD)]$ in
the category $(\Sset)^{(\del_{/Z})^{op}}$.

By Remark \ref{4re:psi}, the composition of the middle row of
\eqref{4eq:descent} is constant of value $w(\sigma)$. Thus $\Map[Z,p]\circ
\Phi\colon\cN\to \Map[Z,\Fun(\partial\Delta^l,\cD)]$ factorizes through the
morphism $\Delta^0_{(\del_{/Z})^{op}}\to \Map[Z,\Fun(\partial
\Delta^l,\cD)]$ corresponding to $w$ via Remark \ref{2re:functors}.

Now let $\sigma'$ be an $n$-simplex of $Y$ corresponding to a map
$\tau'\colon \Delta_L^{n,n,n_k\mid k\in K}\to \delta^{I\square}_*(X,\cT)$.
By restricting to $\CCpt^n\subseteq \Delta^{[n,n]}$ we obtain a vertex of
$\Komp^\alpha(\tau)$. By restricting the composition
\[\Delta^{n,n,n_k\mid k\in K}\xto{\tau'} \op^I_L\delta_*^{I\square}(X,\cT)\xto{v} \delta_*^I\Fun(\Delta^l,\cD),\]
we obtain a vertex of $\Fun(\Delta^l\times\RCpt^n\times\Delta^{[n_k]_{k\in
K}},\cD)$. The two vertices have the same image in $\Fun(H\times
\Delta^{[n_k]_{k\in K}},\cD)$ and hence provide a vertex $\nu(\sigma')$ of
$\cN(f(\sigma'))$, whose image under $\Phi(f(\sigma'))$ is $v(\sigma')$.
This construction is functorial in $\sigma'$, giving rise to $\nu\in
\Gamma(f^*\cN)_0$ such that $f^*\Phi\circ \nu=v$. Applying Proposition
\ref{1pr:extension} to $\Phi$, the map $f\colon Y\to Z$ and the global
section $\nu$ of $f^*\cN$, we obtain a map $u\colon Z\to \Fun(\Delta^l,\cD)$
satisfying $p\circ u=w$ such that $u\circ f$ and $v$ are homotopic over
$\Fun(\partial \Delta^l,\cD)$, as desired.
\end{proof}

Next we show that in a favorable case, the weak contractibility condition in
the theorem can be reduced to a weak contractibility condition on a
$2$-marked simplicial set.

\begin{theorem}\label{4co:multisimplicial_descent}
Let $\cC$ be an $\infty$-category and $K$ a \emph{finite} set. Consider a
$(\{0,1,2\}\coprod K)$-marked $\infty$-category
$(\cC,\cE_0,\cE_1,\cE_2,\{\cE_k\}_{k\in K})$ such that
\begin{enumerate}
  \item $\cE_1,\cE_2\subseteq \cE_0$;

  \item $\cE_0$ is stable under composition;

  \item $\cE_1$, $\cE_2$ are stable under pullback by $\cE_k$ for all
      $k\in K$;

  \item $\cE_k$ is stable under pullback by $\cE_1$ for all $k\in K$; and

  \item edges in $\cE_k$ admit pullbacks in $\cC$ by edges in $\cE_1$ for
      all $k\in K$.
\end{enumerate}
Then for every $(n,n_k)_{k\in K}$-simplex $\tau$ of the $(\{0\}\coprod
K)$-tiled $\infty$-category $\cC_{\cE_0,\{\cE_k\}_{k\in K}}^\cart$, the
restriction map $\Komp^\alpha_{(\cC,\cT)}(\tau)\to
\Komp^\alpha_{\cC,\cE_1,\cE_2}(\gamma)$, where $\gamma$ is the restriction
of $\tau$ to $\Delta^n\times \{(n_k)_{k\in K}\}$, is a trivial Kan fibration
for every $\alpha\in \{1,2\}$. Here
$(\cC,\cT)=(\cC,\cE_1,\cE_2,\{\cE_k\}_{k\in K},\cQ)$ is the $(\{1,2\}\coprod
K)$-tiled $\infty$-category in which $\cQ$ is determined by the conditions
\[\cQ_{12}=\cE_1*_{\cC}\cE_2,\quad \cQ_{ij}=\cE_i*_\cC^\cart\cE_j,\ (i,j)\neq (1,2),(2,1).\]
Moreover, if for some $\alpha\in \{1,2\}$ and for every simplex $\gamma$ of
$\cC_{\cE_0}\subseteq \cC$, the simplicial set
$\Komp^\alpha_{\cC,\cE_1,\cE_2}(\gamma)$ is weakly contractible, then, for
every subset $L\subseteq K$, the map
\[f\colon \delta^*_{\{1,2\}\amalg K,L}\delta_*^{(\{1,2\}\amalg K)\square}(\cC,\cT)\to
\delta^*_{\{0\}\amalg K,L} \cC_{\cE_0,\{\cE_k\}_{k\in K}}^\cart
\]
is a categorical equivalence.
\end{theorem}

\begin{proof}
By Lemma \ref{3le:cart} and Condition (2), $\cE_0*^\cart_\cC \cE_k$, $k\in K$
are stable under composition in the first direction. Thus by Remark
\ref{4re:assump} and Conditions (1) and (2), Assumption \ref{4as:composition}
is satisfied for $(\cC,\cT)$ and $\cC_{\cE_0,\{\cE_k\}_{k\in K}}^\cart$. By
Theorem \ref{4th:multisimplicial_descent}, it suffices to show the first
assertion. Indeed, the assumption that
$\Komp^\alpha_{\cC,\cE_1,\cE_2}(\gamma)$ is weakly contractible then implies
that $\Komp^\alpha_{(\cC,\cT)}(\tau)$ is weakly contractible.

We let $\infty=(n_k)_{k\in K}$ denote the final object of $[n_k]_{k\in
K}\colonequals \prod_{k\in K} [n_k]$. We have the following commutative
diagram
\begin{equation}\label{4eq:contractible}
\xymatrix{\Komp^\alpha_{(\cC,\cT)}(\tau)\ar[rr]\ar[d]&& \Komp^\alpha_{\cC,\cE_1,\cE_2}(\gamma)\ar[d]\\
\Fun(\CCpt^n\times\Delta^{[n_k]_{k\in K}},\cC)_\RKE\ar[r]^-{\rres_1}&
\cK'\ar[d]
\ar[r]&\cK\ar[d]\\
&\Fun(\N(Q\cup R),\cC)\ar[r]^-{\rres_2}& \Fun(\N(Q)\cup\N(R),\cC),}
\end{equation}
where
\begin{itemize}
  \item $Q=[n]\times[n_k]_{k\in K}\subseteq \RCpt^n\times[n_k]_{k\in K}$
      is induced by the diagonal inclusion $[n]\subseteq \RCpt^n$.

  \item $R=\RCpt^n\times\{\infty\}\subseteq \RCpt^n\times[n_k]_{k\in K}$.

  \item $\Fun(\CCpt^n\times\Delta^{[n_k]_{k\in K}},\cC)_\RKE
      \subseteq\Fun(\CCpt^n\times\Delta^{[n_k]_{k\in K}},\cC)$ is the full
      subcategory spanned by functors $F\colon
      \CCpt^n\times\Delta^{[n_k]_{k\in K}}\to\cC$ which are right Kan
      extensions of $F\res\N(Q\cup R)$.

  \item $\cK'\subseteq \Fun(\N(Q\cup R),\cC)$ is the full subcategory
      spanned by functors $F$ such that the composition $\N(Q\cup
      R)_{(i,j,p)/}\to \N(Q\cup R)\xto{F} \cC$ admits a limit for every
      vertex $(i,j,p)$ of $\CCpt^n\times \Delta^{[n_k]_{k\in K}}$.

  \item $\cK\subseteq \Fun(\N(Q)\cup \N(R),\cC)$ is the full subcategory
      spanned by functors $F$ such that the diagram
      \begin{equation}\label{4eq:pullback}
      \xymatrix{& F(i,j,\infty)\ar[d] \\F(j,j,p)\ar[r]&F(j,j,\infty).}
      \end{equation}
      admits a limit in $\cC$ for every vertex $(i,j,p)$ of
      $\CCpt^n\times\Delta^{[n_k]_{k\in K}}$.

  \item The horizontal arrows are restrictions. We will check that
      $\rres_2$ carries $\cK'$ into $\cK$ below.

  \item The lower vertical arrows are inclusions.

  \item The upper right vertical arrow is the amalgamation of the
      inclusion $\Komp^\alpha_{\cC,\cE_1,\cE_2}(\gamma)\subseteq
      \Fun(\N(R),\cC)$ with $\tau$, viewed as a vertex of
      $\Fun(\N(Q),\cC)$. The fact that the image is in $\cK$ follows from
      Conditions (3) and (5) (Condition (3) is needed if $\# K\ge 2$).

  \item The left vertical arrow is induced by the inclusion
      $\Komp^\alpha_{(\cC,\cT)}(\tau)\subseteq
      \Fun(\CCpt^n\times\Delta^{[n_k]_{k\in K}},\cC)$. We will check that
      the image is contained in $\Fun(\CCpt^n\times\Delta^{[n_k]_{k\in
      K}},\cC)_\RKE$ below.
\end{itemize}
For any vertex $(i,j,p)$ of $\CCpt^n\times \Delta^{[n_k]_{k\in K}}$, we let
$g\colon \Delta^1\times \Delta^1\to \CCpt^n\times\Delta^{[n_k]_{k\in K}}$
denote the square
\[\xymatrix{(i,j,p)\ar[r]\ar[d]& (i,j,\infty)\ar[d] \\(j,j,p)\ar[r]&(j,j,\infty).}\]
We have $\Delta^1\times \Delta^1\simeq ((\Lambda^2_0)^{op})^\triangleleft$.
The induced map $\Lambda^2_0\to (\N(Q\cup R)_{(i,j,p)/})^{op}$ is cofinal by
Lemma \ref{3le:cofinal} below. Thus a functor $G\colon
\CCpt^n\times\Delta^{[n_k]_{k\in K}}\to\cC$ is a right Kan extension of
$G\res\N(Q\cup R)$ if and only if $G\circ g$ is a pullback square, for all
$(i,j,p)$. For any vertex $G$ of $\Komp^\alpha_{(\cC,\cT)}(\tau)$, regarded
as a functor $G\colon \CCpt^n\times\Delta^{[n_k]_{k\in K}}\to\cC$, the
square $G\circ g$ is obtained by a finite sequence of compositions from
squares in $\cT_{1k}=\cE_{1}*^\cart_\cC\cE_k$, $k\in K$. Therefore, the
image of $\Komp^\alpha_{(\cC,\cT)}(\tau)\subseteq
\Fun(\CCpt^n\times\Delta^{[n_k]_{k\in K}},\cC)$ is contained in
$\Fun(\CCpt^n\times\Delta^{[n_k]_{k\in K}},\cC)_\RKE$. Moreover, if $F\colon
\N(Q\cup R)\to \cC$ is a functor, then the composition $\N(Q\cup
R)_{(i,j,p)/}\to \N(Q\cup R)\xto{F} \cC$ admits a limit if and only if the
diagram \eqref{4eq:pullback} admits a limit. Thus $\rres_2$ carries $\cK'$
into $\cK$ and the lower right square in a pullback.

By \cite{Lu1}*{Proposition 4.3.2.15}, $\rres_1$ is a trivial Kan fibration.
We apply Lemma \ref{6le:union} to show that the inclusion $\N(Q)\cup
\N(R)\subseteq \N(Q\cup R)$ is inner anodyne. For this we need to check that
$Q\cup R=Q\coprod_{Q\cap R} R$ is a pushout in the category of partially
ordered sets (see Remark \ref{6re:order}). Let $(i,i,p)$ be in $Q$ and
$(i',j',\infty)$ in $R$. If we have $(i',j',\infty)\le (i,i,p)$, then
$p=\infty$ so that $(i,i,p)$ is in $Q\cap R$. On the other hand, if we have
$(i,i,p)\le (i',j',\infty)$, then we have $(i,i,p)\le (i',i',\infty)\le
(i',j',\infty)$. It follows that $\rres_2$ is a trivial Kan fibration.

To show that the upper horizontal arrow is a trivial Kan fibration, it
remains to show that, ignoring the middle term in the second row, the upper
square of \eqref{4eq:contractible} is also a pullback square. This amounts
to saying that for every $m$-simplex $\sigma$ of
$\Fun(\CCpt^n\times\Delta^{[n_k]_{k\in K}},\cC)_\RKE$, if the restriction of
$\sigma$ to $\N(Q)$ is $\tau$ and the restriction of $\sigma$ to $\N(R)$ is
in $\Komp^\alpha_{\cC,\cE_1,\cE_2}(\gamma)$, then $\sigma$ is a simplex of
$\Komp^\alpha_{(\cC,\cT)}(\tau)$. By Remark \ref{4re:g}, it suffices to
treat the cases $m=0$ and $m=1$.

Case $m=0$. Consider integers $0\le i\le i'\le j\le j'\le n$ and a morphism
$p\le q$ of $[n_k]_{k\in K}$. Since $\sigma\colon
\CCpt^n\times\Delta^{[n_k]_{k\in K}}\to\cC$ is a right Kan extension of
$\sigma \res \N(Q\cup R)$, it carries the outer and right squares of the
diagram
\[\xymatrix{(i,j,p)\ar[r]\ar[d]&(i,j,q)\ar[r]\ar[d] & (i,j,\infty)\ar[d]\\
(j,j,p)\ar[r] &(j,j,q)\ar[r] & (j,j,\infty)}
\]
to pullback squares. It follows that $\sigma$ carries the left square to a
pullback square. Thus, since the restriction of $\sigma$ to $\N(Q)$ is
$\tau$, $\sigma$ carries $\cF_k$ to $\cE_k$ for all $k\in K$ by Condition
(4), where $\cF_k$ is defined in Remark \ref{4re:g}. Moreover, since
$\sigma$ carries the outer and lower squares of the diagram
\[\xymatrix{(i,j,p)\ar[r]\ar[d] & (i,j,q)\ar[d]\\
(i',j,p)\ar[r]\ar[d] & (i',j,q)\ar[d]\\
(j,j,p)\ar[r] & (j,j,q)}
\]
to pullbacks, it carries the upper square to a pullback. Taking $q=\infty$,
Condition (3) then implies that $\sigma$ carries $(i,j,p)\to (i',j,p)$ to a
morphism in $\cE_1$. It follows that $\sigma$ carries $\cF_1*\cF_k$ into
$\cE_1*^\cart\cE_k$ for every $k\in K$. Consider the cube
\[\xymatrix{(i,j,p)\ar[rr]\ar[rd]\ar[dd] && (i,j,q)\ar[rd]\ar@{-->}[dd]\\
&(i,j',p)\ar[rr]\ar[dd] &&(i,j',q)\ar[dd]\\
(j,j,p)\ar@{-->}[rr]\ar[rd] && (j,j,q)\ar@{-->}[rd]\\
&(j',j',p)\ar[rr] &&(j',j',q).}
\]
The image of the bottom square under $\sigma$ can be obtained by a finite
sequence of compositions from squares in $\cE_0*_\cC^\cart \cE_k$, $k\in K$.
Since $\sigma$ carries the front and back squares to pullbacks as well,
$\sigma$ carries the top square to pullback. Taking $q=\infty$, Condition
(3) then implies that $\sigma$ carries $(i,j,p)\to (i,j',p)$ to a morphism
in $\cE_2$. It follows that $\sigma$ carries $\cF_2*\cF_k$ into
$\cE_2*^\cart \cE_k$ for every $k\in K$. Finally, given a square $S$ in
$\cF_k*\cF_{l}$ for distinct $k,l\in K$, let $(i,j)$ be its projection in
$\CCpt^n$ and $T$ its projection in $\Delta^{[n_k]_{k\in K}}$. Then $S$ can
be identified with the top face of a cube, product of the edge $(i,j)\to
(j,j)$ and the square $T$. Since $\sigma$ carries the other five faces of
the cube to pullback squares, it carries $S$ to a pullback as well.

Case $m=1$. We check Condition (2) in Remark \ref{4re:g}. For $0\le i\le
j\le n$ and $p\le q$ in $[n_k]_{k\in K}$, consider the following cube in
$\Delta^1\times\CCpt^n\times\Delta^{[n_k]_{k\in K}}$:
\[\xymatrix{(0,i,j,p)\ar[rd]\ar[dd]\ar[rr]&&(0,i,j,q)\ar[rd]\ar@{-->}[dd]\\
&(0,j,j,p)\ar[rr]\ar[dd]&&(0,j,j,q)\ar[dd]\\
(1,i,j,p)\ar@{-->}[rr]\ar[rd]&& (1,i,j,q)\ar@{-->}[rd]\\
&(1,j,j,p)\ar[rr]&&(1,j,j,q).}
\]
Since $\sigma\colon \Delta^1\times\CCpt^n\times\Delta^{[n_k]_{k\in
K}}\to\cC$ carries the top and bottom squares to pullbacks and carries the
front square to the identity on $\tau(j,p)\to \tau(j,q)$, it carries the
back square to a pullback. Taking $q=\infty$, Condition (3) then implies
that $\sigma$ carries $(0,i,j,p)\to (1,i,j,p)$ to a morphism in
$\cE_\alpha$.
\end{proof}

\begin{lemma}\label{3le:cofinal}
Let $P$ be a partially ordered set and $f\colon \Lambda^2_0\to \N(P)$ a map.
Assume that $f(0)$ is the product (namely, supremum) of $f(1)$ and $f(2)$ in
$P$, and $P_{/f(1)}\cup P_{/f(2)}=P$. Then $f$ is cofinal
\cite{Lu1}*{Definition 4.1.1.1}.
\end{lemma}

\begin{proof}
By \cite{Lu1}*{Theorem 4.1.3.1}, it suffices to show that for every $p\in
P$, the simplicial set $S=\Lambda^2_0\times_{\N(P)}\N(P_{p/})$ is weakly
contractible. By the second assumption, either $p\le f(1)$ or $p\le f(2)$.
If exactly one of the two inequalities holds, then $S$ is a point. If both
inequalities hold, then $p\le f(0)$ by the first assumption, and hence
$S=\Lambda^2_0$.
\end{proof}

\begin{remark}
In Theorem \ref{4co:multisimplicial_descent}, the Cartesian restriction on
$\cE_k*_\cC\cE_l$ for $k,l\in K$ is not essential. To be more precise, under
the assumptions of the theorem, consider an $I$-tiling $\cT=((\cE_{i})_{i\in
I},(\cQ_{ij})_{i,j\in I,i\neq j})$ and a $J$-tiling $\cT'=((\cE_i)_{i\in
J},(\cQ_{ij})_{i,j\in J,i\neq j})$ such that $\cQ_{12}=\cE_1*_\cC \cE_2$,
$\cQ_{ik}=\cE_i*^\cart_\cC \cE_k$ for $i=0,1,2$ and $k\in K$, and
$\cQ_{kl}\subseteq \cE_k*_\cC \cE_l$ is stable under pullback by
$\cQ_{1l}=\cE_1*_\cC^\cart \cE_l$ in the first direction or stable under
pullback by $\cQ_{k1}=\cE_k*_\cC^\cart\cE_1$ in the second direction for
$k,l\in K$, $k\neq l$. Then the proof shows that the restriction map
$\Komp^\alpha_{(\cC,\cT)}(\tau)\to \Komp^\alpha_{\cC,\cE_1,\cE_2}(\gamma)$
is a trivial Kan fibration for every $(n,n_k)_{k\in K}$-simplex $\tau$ of
$\delta^{(\{0\}\amalg K)\square}_*(X,\cT')$, and if
$\Komp^\alpha_{\cC,\cE_1,\cE_2}(\gamma)$ is weakly contractible for every
$\gamma$, then
\[
f\colon \delta_{\{1,2\}\amalg K,L}^*\delta^{(\{1,2\}\amalg K)\square}_*(X,\cT)
\to \delta_{\{0\}\amalg K,L}^*\delta^{(\{0\}\amalg K)\square}_*(X,\cT')
\]
is a categorical equivalence.
\end{remark}

As promised, we give sufficient conditions for the simplicial set
$\Komp^\alpha(\tau)$ to be an $\infty$-category.

\begin{lemma}\label{4le:infinity_category}
In the situation of Definition \ref{4de:compactification},
\begin{enumerate}
  \item Assume that $\cT_{\alpha \beta}$ is stable under composition in
      the first direction (Definition \ref{3de:squares}) for all
      $\beta\in \{1,2\}\coprod K$, $\beta\neq \alpha$. Then the map $g$
      is an inner fibration. Moreover, if $X$ is an $\infty$-category,
      then $\Komp^\alpha_{(X,\cT)}(\tau)$ is an $\infty$-category.

  \item If we have $(X,\cT)=\sfW(X,\cE)$ and $\cE_\alpha$ is composable
      (Definition \ref{3de:edges}) and $X$ is an $\infty$-category, then
      $\Komp^\alpha_{X,\cE}(\tau)=\Komp^\alpha_{(X,\cT)}(\tau)$ is an
      $\infty$-category.
\end{enumerate}
\end{lemma}

The assumption in (1) implies that $\cT_\alpha$ is stable under composition.
The assumption in (1) is satisfied if we have $(X,\cT)=\sfW(X,\cE)$ and
$\cE_\alpha$ is stable under composition.

\begin{proof}
By Remark \ref{4re:g}, $g$ satisfies the right lifting property with respect
to every horn inclusion $\Lambda^m_i\subseteq \Lambda^m$ for $m\ge 3$. Thus,
for the first assertion of (1), it suffices to show that $g$ satisfies the
right lifting property with respect to $\Lambda^2_1\subseteq \Delta^2$. We
use the notation of Remark \ref{4re:g}. Let $\gamma$ be a $2$-simplex of
$\Map(\square^n\times \Delta^{[n_k]_{k\in K}},X)$ such that the restriction
of $\gamma$ to $\Lambda^2_1$ factorizes through $Y$. We regard $\gamma$ as a
map $\Delta^2\times (\square^n\times \Delta^{[n_k]_{k\in K}})\to X$. For any
square in $\cF_\alpha*\cF_\beta$ of the form \eqref{4eq:Komp}, consider the
map $\Delta^2\times \Delta^1\to \Delta^2\times (\square^n\times
\Delta^{[n_k]_{k\in K}})$ as shown by the diagram
\[\xymatrix{(0,y')\ar[r]\ar[d] & (0,y)\ar[d]\\
(1,y')\ar[r]\ar[d] & (1,y)\ar[d]\\
(2,x')\ar[r] & (2,x).}
\]
By assumption, $\gamma$ carries the upper and lower squares to squares in
$\cT_{\alpha\beta}$. (We could replace the second row by $(1,x')\to (1,x)$
without affecting the validity of the argument.) Since $\cT_{\alpha\beta}$
is stable under composition in the first direction, $\gamma$ carries the
outer square to a square in $\cT_{\alpha\beta}$. Therefore, the restriction
of $\gamma$ to $\Delta^{\{0,2\}}$ is an edge of $Y$.

The second assertion of (1) follows from the first assertion of (1) and the
fact that $\RES_2$ is an inner fibration if $X$ is an $\infty$-category
(Remark \ref{4re:fibration}).

For (2), note that by Remark \ref{4re:fibration}, we have a diagram with
pullback square
\[\xymatrix{\Komp^\alpha(\tau)\ar[r]
&Z\ar[r]\ar[d] & Y\ar[d]^-{\RES_3\circ g}\\
&\{\tau\}\ar[r] & \Map(D^n\times\Delta^{[n_k]_{k\in K}}_L,X),}
\]
where $Z$ denotes the fiber of the map $\RES_3\circ g$ at $\tau$, and the
map $\Komp^\alpha(\tau)\to Z$ is a pullback of the map $\RES_4$ in Remark
\ref{4re:fibration}, hence an inner fibration. Thus it suffices to show that
$Z$ is an $\infty$-category. Since $\RES_3$ is an inner fibration and $g$
satisfies the right lifting property with respect to every horn inclusion
$\Lambda^m_i\subseteq \Lambda^m$ for $m\ge 3$, it suffices to check that $Z$
satisfies the extension property with respect to $\Lambda^2_1\subseteq
\Delta^2$. Let $f\colon \Lambda^2_1\to Z$ be a map. Unwinding the
definition, to show that $f$ extends to a map $\Delta^2\to Z$, we are
reduced to showing the extension property
\[\xymatrix{(A,A_1\cap \cG)\ar@{^{(}->}[d]\ar[r]^-{f'} & (X,\cE_\alpha),\\
(B,\cG)\ar@{..>}[ru]}
\]
where we have $(B,\cG)=(\Delta^2)^\sharp\times (\square^n\times
\Delta^{[n_k]_{k\in K}}_L,\cF_\alpha)$ and
\[A=\Lambda^2_1\times
(\square^n\times \Delta^{[n_k]_{k\in K}}_L)\coprod_{\Lambda^2_1\times
(D^n\times \Delta^{[n_k]_{k\in K}}_L)} \Delta^2\times (D^n\times
\Delta^{[n_k]_{k\in K}}_L),
\]
and $f'$ is the amalgamation of $f$ and $\tau$. Every edge in $\cG$ that is
not in $A$ has the form $(0,y)\to (2,x)$ with $y\to x$ in $\cF_\alpha$, and
can be extended to a $2$-simplex of $B$
\[\xymatrix{&(1,y)\ar[rd]\\
(0,y)\ar[ru]\ar[rr] &&(2,x),}
\]
where the oblique edges are in $A_1\cap \cG$. (Again we could replace
$(1,y)$ by $(1,x)$.) Therefore, it suffices to apply Lemma \ref{3le:comp}.
\end{proof}

We now give a criterion for the weak contractibility of certain
$\infty$-categories of compactifications.

\begin{theorem}\label{4pr:descent}
Let $(\cC,\cE_1,\cE_2)$ be a $2$-marked $\infty$-category. Suppose that the
following conditions are satisfied:
\begin{enumerate}
  \item $\cE_1$ and $\cE_2$ are composable (Definition \ref{3de:edges}).

  \item The $\infty$-category $\cC_{\cE_1}$ admits pullbacks and
      pullbacks are preserved by the inclusion $\cC_{\cE_1}\subseteq
      \cC$.

  \item For every morphism $f$ of $\cC$, there exists a $2$-simplex of
      $\cC$ of the form
      \begin{equation}\label{4eq:2cell}
      \xymatrix{&y\ar[rd]^p\\z\ar[ru]^q\ar[rr]^f && x}
      \end{equation}
      with $p\in\cE_1$ and $q\in\cE_2$.
\end{enumerate}
Then, for every $n$-simplex $\tau$ of $\cC$, the simplicial set
$\Komp^1_{\cC,\cE_1,\cE_2}(\tau)^{op}$ is a filtered $\infty$-category and
is weakly contractible. Moreover, the natural map
\[\delta_2^*\cC_{\cE_1,\cE_2} \to \cC\]
is a categorical equivalence.
\end{theorem}

Recall that an $\infty$-category is said to be \emph{filtered}
\cite{Lu1}*{Definition 5.3.1.7} if it satisfies the extension property with
respect to the inclusion $A\subseteq A^\triangleright$ for every finite
simplicial set $A$. Recall also that an ordinary category is filtered if and
only if its nerve is a filtered $\infty$-category \cite{Lu1}*{Proposition
5.3.1.13}. Thus in the case where $\cC$ is the nerve of an ordinary
category, the first assertion of Theorem \ref{4pr:descent} generalizes
\cite{SGA4XVII}*{Proposition 3.2.6}.

\begin{remark}
Condition (2) of Theorem \ref{4pr:descent} is satisfied if the following
conditions are satisfied:
\begin{enumerate}[(a)]
  \item morphisms in $\cE_1$ admit pullbacks in $\cC$ by morphisms in
      $\cE_1$;

  \item $\cE_1$ is stable under pullback by $\cE_1$;

  \item for every $2$-simplex of $\cC$ of the form \eqref{4eq:2cell} such
      that $f$ and $p$ are in $\cE_1$, $q$ is in $\cE_1$.
\end{enumerate}
Indeed, Condition (c) implies that for every diagram $a\colon A\to \cC$,
where $A$ is a nonempty simplicial set, the overcategory $(\cC_{\cE_1})_{/a}$
is a full subcategory of $\cC_{/a}$, so that a diagram $\bar a\colon
A^\triangleleft\to \cC_{\cE_1}$ is a limit diagram if the composition
$A^\triangleleft \xto{\bar a}\cC_{\cE_1}\to \cC$ is a limit diagram. Note
that Conditions (b) and (c) hold if $\cE_1$ is admissible (Definition
\ref{3de:admissible_edge}).
\end{remark}

\begin{proof}[Proof of Theorem \ref{4pr:descent}]
For brevity we write $\Komp^1(\tau)$ for $\Komp^1_{\cC,\cE_1,\cE_2}(\tau)$.
Since $\cE_1$ is composable, $\Komp^1(\tau)$ is an $\infty$-category by Lemma
\ref{4le:infinity_category}. It suffices to show that $\Komp^1(\tau)^{op}$ is
filtered. In fact, every filtered $\infty$-category is weakly contractible
\cite{Lu1}*{Lemma 5.3.1.18}. The last assertion of the proposition then
follows from Corollary \ref{4co:multi2}.

By \cite{Lu1}*{Remark 5.3.1.10}, $\Komp^1(\tau)^{op}$ is filtered if and
only if $\Komp^1(\tau)$ has the extension property with respect to the
inclusion $A\subseteq A^\triangleleft$ whenever $A$ is the nerve of a finite
partially ordered set. We fix such an $A$ and proceed by induction on $n$.
For $n=0$, $\Komp^1(\tau)$ is a point and the assertion holds trivially.

For $n\ge 1$, by the induction hypothesis, the composite map $f_{-1}\colon
A\xrightarrow{f}\Komp^1(\tau)\to \Komp^1(\tau\circ d^n_n)$ extends to
$g_{-1}\colon A^\triangleleft\to \Komp^1(\tau\circ d^n_n)$. We identify
$\CCpt^{n-1}$ with its image under $d^n_n$, hence with the full subcategory
of $\CCpt^n$ spanned by the objects $(i,j)$, $1\le i\le j\le n-1$. For
$0\leq k\leq n$, consider the full subcategory $\CCpt^n_k$ of $\CCpt^n$
spanned by $\CCpt^{n-1}$ and the objects $(i,n)$ with $n-k\leq i\leq n$. We
have $\CCpt^{n-1}\subseteq \CCpt^n_0 \subseteq \dots \subseteq
\CCpt^n_n=\CCpt^n$. Similarly we define $\Cpt^n_k\subseteq \Cpt^n$. Define
$\Komp^1_k(\tau)$ similarly to $\Komp^1(\tau)$ but with $\CCpt^n$, $\Cpt^n$
and $\Box^n=\delta_2^*\Cpt^n$ replaced by $\CCpt^n_k$, $\Cpt^n_k$ and
$\delta_2^*\Cpt^n_k$, respectively. We show by induction on $k$ that there
exists a map $g_k\colon A^\triangleleft \to \Komp^1_k(\tau)$ compatible with
$f_k$ and $g_{k-1}$, where $f_k$ is the composition of $f$ and the natural
map $\Komp^1(\tau)\to\Komp^1_k(\tau)$, rendering the following diagram
commutative:
\[\xymatrix{
A \ar[dd] \ar[rd]^(.6){f=f_n} \ar@/^0.6pc/[rrrd]_-{f_k}\ar@/^0.6pc/[rrrrd]^(.7){f_{k-1}}\ar@/^1pc/[drrrrrr]^(.7){f_{-1}} \\
& \Komp^1(\tau) \ar[r] & \cdots \ar[r] &\Komp^1_k(\tau) \ar[r] & \Komp^1_{k-1}(\tau)\ar[r] &\cdots \ar[r] & \Komp^1(\tau\circ d^n_n). \\
A^\triangleleft \ar@/_1pc/[urrrrrr]_(.7){g_{-1}}
\ar@{..>}@/_0.6pc/[urrr]^-{g_k}\ar@{..>}@/_0.6pc/[urrrr]_(.7){g_{k-1}}  \ar@{..>}[ur]_(.6){g_n} }
\]
The map $g_n$ will allow us to conclude the proof of the proposition.

Below are the Hasse diagrams of (the homotopy categories of) $\CCpt^3_0$ and
$\CCpt^3_2$, respectively. Bullets in the first diagram represent vertices in
the image of the diagonal embedding $\Delta^3\subseteq \CCpt^3_0$. Bullets in
the second diagram represent vertices in the image of the embedding
$\Delta^2\to \CCpt^3_2$ defined later in the proof.
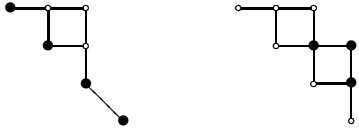
\[\begin{xy}
(0,5)="00"; (5,5)*\cir<1.8pt>{}="01"**\dir{-};
(10,5)*\cir<1.8pt>{}="02"**\dir{-};
(5,0)="11"; (10,0)*\cir<1.8pt>{}="12"**\dir{-};
(10,-5)="22";
(15,-10)="33"**\dir{-};
"01"; "11"**\dir{-};
"02"; "12"**\dir{-}; "22"**\dir{-};
"00"*{\bullet}; "11"*{\bullet}; "22"*{\bullet}; "33"*{\bullet};
\end{xy}\qquad\qquad\begin{xy}
(0,5)*\cir<1.8pt>{}="00"; (5,5)*\cir<1.8pt>{}="01"**\dir{-};
(10,5)*\cir<1.8pt>{}="02"**\dir{-};
(5,0)*\cir<1.8pt>{}="11"; (10,0)="12"**\dir{-};
(15,0)="13"**\dir{-};
(10,-5)*\cir<1.8pt>{}="22";
(15,-5)="23"**\dir{-};
(15,-10)*\cir<1.8pt>{}="33"**\dir{-};
"01"; "11"**\dir{-};
"02"; "12"**\dir{-}; "22"**\dir{-};
"13"; "23"**\dir{-};
"12"*{\bullet}; "13"*{\bullet}; "23"*{\bullet};
\end{xy}
\]

We first consider the case $k=0$. The map $f_0$ (resp.\ $g_{-1}$)
corresponds to a map $\tilde{f}_0\colon A\times\CCpt^n_0\to\cC$ (resp.\
$\tilde{g}_{-1}\colon A^\triangleleft\times\CCpt^{n-1}\to\cC$). To find the
desired map $g_0$, it suffices to construct a map $\tilde{g}_0\colon
A^\triangleleft\times\CCpt^n_0\to\cC$, extending $\tilde{f}_0$ and
$\tilde{g}_{-1}$ and the composition
$A^\triangleleft\times\Delta^n\to\Delta^n\xrightarrow{\tau}\cC$, where the
first map is the projection, via the diagonal embedding $\Delta^n\subseteq
\CCpt^n_0$. This follows if $\cC$ has the extension property with respect to
the smash product of $A\subseteq A^\triangleleft$ and
$\CCpt^{n-1}\coprod_{\Delta^{n-1}}\Delta^n\subseteq\CCpt^n_0$. However, the
latter inclusion is inner anodyne by Lemma \ref{6le:union} applied to
$Q=[n]^{op}$ and $R=(\RCpt^{n-1})^{op}$. Thus we may find the map
$\tilde{g}_0$ by \cite{Lu1}*{Corollary 2.3.2.4} as $\cC$ is an
$\infty$-category.

For $1\le k\le n$, consider the full subcategory $\Delta^2\subseteq
\CCpt^n_k$ spanned by $\{(n-k,n-1),(n-k,n),(n-k+1,n)\}$. We identify
$\Delta^{\{0,2\}}$ with the subcategory of $\CCpt^n_{k-1}$ spanned by
$\{(n-k,n-1),(n-k+1,n)\}$. The inclusion
$\CCpt^n_{k-1}\coprod_{\Delta^{\{0,2\}}}\Delta^2\subseteq \CCpt^n_k$ is
inner anodyne by Lemma \ref{6le:cut}, and so is its smash product
\[S\colonequals \left(A^\triangleleft \times
\left(\CCpt^n_{k-1}\coprod_{\Delta^{\{0,2\}}}\Delta^2\right)\right)\cup(A\times \CCpt^n_k)\subseteq A^\triangleleft \times \CCpt^n_k\]
with $A\subseteq A^\triangleleft$. We define $\cG_1$ and $\cG_2$ by
\[(A^\triangleleft\times \delta_2^*\Cpt^n_k,\cG_1,\cG_2)\simeq
(A^\triangleleft)^{\sharp^2_{\{1\}}} \times \delta_{2+}^{*}\Cpt^n_k.
\]
We let $-\infty$ denote the cone point of $A^\triangleleft$. Any edge in
$\cG_1$ but not in $S$ has the form $(-\infty,n-k,n)\to (l,i,n)$ with $l$ in
$A^\triangleleft$ and $i>n-k+1$, and can be extended to a $2$-simplex
\[\xymatrix{&(l,n-k+1,n)\ar[rd]\\
(-\infty,n-k,n)\ar[ru]\ar[rr]
&&(l,i,n)}
\]
with oblique edges in $S_1\cap \cG_1$. Any edge in $\cG_2$ but not in $S$
has the form $(-\infty,n-k,j)\to (-\infty,n-k,n)$ with $j<n-1$ and can be
extended to a $2$-simplex
\[\xymatrix{&(-\infty,n-k,n-1)\ar[rd]\\
(-\infty,n-k,j)\ar[ru]\ar[rr]
&&(-\infty,n-k,n)}
\]
with oblique edges in $S_1\cap \cG_2$. Thus, by Condition (1) and Lemma
\ref{3le:comp}, it suffices to construct a map $(S,S_1\cap \cG_1,S_1\cap
\cG_2)\to (\cC,\cE_1,\cE_2)$ extending the amalgamation $v\colon
V\colonequals A\times \CCpt^n_k\coprod_{A\times \CCpt^n_{k-1}}
A^\triangleleft\times \CCpt^n_{k-1}\to \cC$ of $\tilde f_k$ and $\tilde
g_k$, where $\tilde{f}_k\colon A\times\CCpt^n_k\to\cC$ (resp.\
$\tilde{g}_{k-1}\colon A^\triangleleft\times\CCpt^{n}_{k-1}\to\cC$) is the
map given by $f_k$ (resp.\ $g_{k-1}$). For this, it suffices to construct a
map $(A^\triangleleft)^{\sharp^2_{\{1\}}}\times T\to(\cC,\cE_1,\cE_2)$
extending the amalgamation of $\tilde{f}_k\res A\times\Delta^2$ and
$\tilde{g}_{k-1}\res A^\triangleleft\times\Delta^{\{0,2\}}$. Here
$T=(\Delta^2,\cF_1,\cF_2)$ is the $2$-marked simplicial set with $\cF_1$
(resp.\ $\cF_2$) consisting of the degenerate edges and the edge $1\to 2$
(resp.\ $0\to 1$).

We now lift $v$ to a map $V\to \cC_{/\tau(n)}$, corresponding to a map
$(V^\triangleright,\cG'_1,\cG'_2)\to (\cC,\cE_1,\cE_2)$, where $\cG'_1$ is
the union of $(V_1\cap \cG_1)\cup \{\id_{+\infty}\}$ and all edges
$(l,i,n)\to +\infty$ in $V^\triangleright$ for $l\in A^\triangleleft$, and
$\cG'_2\colonequals (V_1\cap \cG_2)\cup \{\id_{+\infty}\}$. Here $+\infty$
denotes the cone point of $V^\triangleright$. Consider the inclusion
$\iota\colon A^\triangleleft \to V$ induced by the inclusion
$\{(n,n)\}\subseteq \CCpt^n_{k-1}$. Since the restriction of $v$ to
$A^\triangleleft$ is constant of value $\tau(n)$, the amalgamation of $v$
and the constant map $A^{\triangleleft\triangleright}\to \cC$ of value
$\tau(n)$ provides a map $v'\colon
(C^\triangleright(\iota),\cG''_1,\cG'_2)\to (\cC,\cE_1,\cE_2)$, where we
have $C^\triangleright(\iota)\colonequals V\coprod_{A^\triangleleft}
A^{\triangleleft\triangleright}$, and $\cG''_1$ is the intersection of
$\cG'_1$ and the set of edges of $C^\triangleright(\iota)$. Since the
inclusions $\{(n,n)\}\subseteq \CCpt^n_{k-1}$ and $\{(n,n)\}\subseteq
\CCpt^n_{k}$ are right anodyne by \cite{Lu1}*{Lemma 4.2.3.6}, and so are
their products with identity maps \cite{Lu1}*{Corollary 2.1.2.7}, the
inclusion $A^\triangleleft=A\coprod_A A^\triangleleft\subseteq V$ is right
anodyne by Lemma \ref{2le:pushout_anodyne}. By \cite{Lu1}*{Lemma 2.1.2.3},
it follows that the inclusion $C^\triangleright(\iota)\subseteq
V^\triangleright$ is inner anodyne. Every edge in $\cG'_1$ that is not in
$\cG''_1$ has the form $(l,i,n)\to +\infty$ and can be extended to a
$2$-simplex
\[\xymatrix{&(l,n,n)\ar[rd]\\(l,i,n)\ar[rr]\ar[ru]&&+\infty}\]
with oblique edges in $\cG''_1$. Lemma \ref{3le:comp} then provides the
desired extension of $v'$ and hence $v$.

We are therefore reduced to showing that every map
\[a\colon A^{\sharp^2_{\{1\}}}\times T \coprod_{A^{\sharp^2_{\{1\}}}\times (\Delta^{\{0,2\}})^{\flat^2}}
(A^\triangleleft)^{\sharp^2_{\{1\}}}\times (\Delta^{\{0,2\}})^{\flat^2}\to (\cC_{/x},\cE'_1,\cE'_2)
\]
whose restriction to $A\times\Delta^{\{1,2\}}\coprod_{A\times
\Delta^{\{2\}}} A^\triangleleft \times \Delta^{\{2\}}$ factorizes through
$(\cC_{\cE_1})_{/x}$ extends to a map
$(A^\triangleleft)^{\sharp^2_{\{1\}}}\times T\to (\cC_{/x},\cE'_1,\cE'_2)$.
Here $x$ is an object of $\cC$ and $\cE'_i$ denotes the inverse image of
$\cE_i$ via the map $\cC_{/x}\to \cC$ for $i=1,2$. Recall that $A$ is the
nerve of a partially ordered set. We let $B\subseteq A^\triangleleft\times
\Delta^2$ denote the full subcategory spanned by all vertices except
$(-\infty,1)$. Consider the commutative diagram of inclusions
\[\xymatrix{A\times\Delta^{\{0,2\}}\ar[r]\ar[d] & A\times \Delta^2\ar[d]\ar[rd]\\
(A\times \Delta^{\{0,2\}})^\triangleleft \ar[r]\ar[d] &
A\times \Delta^2\coprod_{A\times\Delta^{\{0,2\}}}(A\times \Delta^{\{0,2\}})^\triangleleft\ar[r]^-{h}\ar[d]
&(A\times \Delta^2)^\triangleleft\ar[d]\\
A^\triangleleft\times \Delta^{\{0,2\}}\ar[r] & A\times \Delta^2\coprod_{A\times\Delta^{\{0,2\}}}A^\triangleleft\times \Delta^{\{0,2\}}\ar[r]^-{h'} & B}
\]
where the lower left (resp.\ right) vertical arrow carries the cone point of
$(A\times\Delta^{\{0,2\}})^\triangleleft$ (resp.\
$(A\times\Delta^2)^\triangleleft$) to $(-\infty,0)$, and the squares on the
left are clearly pushouts. For any simplex $\sigma$ of $B$, if $\sigma$ is
not a simplex of $(A\times \Delta^2)^\triangleleft$, then $(-\infty,2)$ is a
vertex of $\sigma$, so that $\sigma$ is a simplex of $A^\triangleleft\times
\Delta^{\{0,2\}}$. Thus $h'$ is a pushout of $h$, which is inner anodyne by
\cite{Lu1}*{Lemma 2.1.2.3}, since the inclusion $A\times \Delta^{\{0,2\}}
\subseteq A\times \Delta^2$ is left anodyne by \cite{Lu1}*{Corollary
2.1.2.7}. Thus $a$ extends to a map $a'\colon B\to\cC_{/x}$. We would like
to apply \cite{Lu1}*{Lemma 4.3.2.13} to conclude that there exists a right
Kan extension $b\colon A^\triangleleft\times \Delta^2\to \cC_{/x}$ of $a'$.
The only condition we need to check for this is that the induced diagram
$B_{(-\infty,1)/}\to B\xrightarrow{a'}\cC_{/x}$ has a limit. However, the
composite map factorizes through $a_0\colon
B_{(-\infty,1)/}\to(\cC_{\cE_1})_{/x}$. By Condition (2) and Lemma
\ref{4le:over_pull} below, the $\infty$-category $(\cC_{\cE_1})_{/x}$ admits
finite limits and such limits are preserved by the inclusion
$(\cC_{\cE_1})_{/x}\subseteq\cC_{/x}$. We therefore obtain a limit diagram
$b_0\colon A^\triangleleft \times \Delta^{\{1,2\}}\to (\cC_{\cE_1})_{/x}$
extending $a_0$ and a right Kan extension $b$ of $a'$. The restriction of
$b$ to $(A^\triangleleft \times \Delta^{\{1,2\}})\cup B$ is equivalent to
the amalgamation $b_1$ of $b_0$ and $a'$. Thus, by \cite{Lu1}*{Lemma
2.4.6.3}, up to replacing $b$ by an extension of $b_1$, we may assume that
$b\res A^\triangleleft \times \Delta^{\{1,2\}}$ factorizes through
$(\cC_{\cE_1})_{/x}$.

Note that $b$ does not necessarily carry the edge
$(-\infty,0)\to(-\infty,1)$ into $\cE'_2$, which is the last requirement to
conclude that $b$ gives rise to the desired extension
$(A^\triangleleft)^{\sharp^2_{\{1\}}}\times T\to (\cC_{/x},\cE'_1,\cE'_2)$.
To overcome this problem, we apply Condition (3) to the arrow
$b((-\infty,0)\to (-\infty,1))$ to get a $2$-simplex $\gamma$ of $\cC_{/x}$.
Consider the totally ordered set $I=\{0<1^-<1<2\}$, which contains
$[2]=\{0<1<2\}$. The amalgamation of $\gamma$ and $b$ is a map $c\colon K\to
\cC_{/x}$, where
\[K\colonequals A^\triangleleft \times\Delta^2\coprod_{\{-\infty\}\times \Delta^1}
\{-\infty\}\times \Delta^{\{0,1^-,1\}}\subseteq A^\triangleleft \times \Delta^I,
\]
with $c((-\infty,0)\to (-\infty,1^-))\in\cE'_2$ and $c((-\infty,1^-)\to
(-\infty,1))\in\cE'_1$. We let $\cF'_1$ (resp.\ $\cF'_2$) denote the set of
all degenerate edges of $\Delta^I$ and all edges of $\Delta^{\{1^-,1,2\}}$
(resp.\ $\Delta^{\{0,1^-\}}$). Consider the pushout
\[(L,\cH_1,\cH_2)=(A^\triangleleft)^{\sharp^2_{\{1\}}} \times
(\Delta^I,\cF'_1,\cF'_2)\coprod_{(A\times \Delta^I)^{\flat^2}} (A\times \Delta^2)^{\flat^2}
\]
given by the degeneracy
map $I\to [2]$ identifying $1^-$ and $1$. The inclusion $K\subseteq L$
induced by the inclusion $K\subseteq A^\triangleleft \times \Delta^I$ is a
pushout of the inclusion
\begin{multline*}
r\colon (\{-\infty\}\times \Delta^0)\star (A^\triangleleft \times \Delta^{\{1,2\}})
\coprod_{(\{-\infty\}\times \Delta^0)\star (\{-\infty\}\times
\Delta^{\{1\}})} (\{-\infty\}\times \Delta^{\{0,1^-\}})\star
(\{-\infty\}\times \Delta^{\{1\}}) \\
\to (\{-\infty\}\times\Delta^{\{0,1^-\}})\star (A^\triangleleft \times
\Delta^{\{1,2\}}).
\end{multline*}
Indeed, for any simplex $\sigma$ of $L$, if $(-\infty,1^-)$ is a vertex of
$\sigma$, then $\sigma$ is a simplex of the target of $r$; otherwise
$\sigma$ is a simplex of $A^\triangleleft\times \Delta^2$. Moreover, $r$ is
inner anodyne by \cite{Lu1}*{Lemma 2.1.2.3}, since the inclusion
$\{-\infty\}\times \Delta^{\{1\}}\subseteq A^\triangleleft\times
\Delta^{\{1,2\}}$ is left anodyne by \cite{Lu1}*{Lemma 4.2.3.6}. Note that
we have $\cH_2\subseteq K_1$ and $c$ induces a map $(K,K_1\cap
\cH_1,\cH_2)\to (\cC_{/x},\cE'_1,\cE'_2)$. Moreover, any edge in $\cH_1$
that is not in $K$ has the form $(-\infty,1^-)\to (l,m)$ with $m\ge 1$ and
can be extended to a $2$-simplex
\[\xymatrix{&(-\infty,1)\ar[rd]\\
(-\infty,1^-)\ar[rr]\ar[ru]&&(l,m)}
\]
with oblique arrows in $K_1\cap \cH_1$. Thus, by Condition (1) and Lemma
\ref{3le:comp}, $c$ extends to a map $c'\colon (L,\cH_1,\cH_2)\to
(\cC_{/x},\cE'_1,\cE'_2)$. The restriction of $c'$ to $A^\triangleleft \times
\Delta^{\{0,1^-,2\}}\simeq A^\triangleleft \times \Delta^2$ provides the
desired extension.
\end{proof}

\begin{lemma}\label{4le:over_pull}
Let $\cC$ and $\cD$ be $\infty$-categories and $f\colon \cC\to \cD$ a
functor. Assume that $\cC$ admits pullbacks and pullbacks are preserved by
$f$. Then, for any object $x$ of $\cC$, the overcategory $\cC_{/x}$ admits
finite limits and such limits are preserved by the functor $f'\colon
\cC_{/x}\to \cD_{/f(x)}$.
\end{lemma}

\begin{proof}
The morphism $\id_x$ is a final object of $\cC_{/x}$ and
$f(\id_x)=\id_{f(x)}$ is a final object of $\cD_{/f(x)}$. By Lemma
\ref{4le:over} below, $\cC_{/x}$ admits pullbacks and the functors
$\cC_{/x}\to \cC$ and $\cD_{/f(x)}\to \cD$ preserve pullbacks. Since the
latter is conservative, the functor $f'$ preserves pullbacks. We conclude by
\cite{Lu1}*{Corollaries 4.4.2.4, 4.4.2.5}.
\end{proof}

\begin{lemma}\label{4le:over}
Let $A$ and $B$ be simplicial sets. Assume that $B$ is weakly contractible.
Let $\cC$ be an $\infty$-category and $p\colon A\to \cC$ a diagram. Then a
diagram $f\colon B\to \cC_{/p}$ admits a limit if and only if the
composition $B\xto{f} \cC_{/p}\to \cC$ admits a limit. Moreover, $\bar
f\colon B^\triangleleft\to \cC_{/p}$ is a limit diagram if and only if the
composition $B^\triangleleft \xto{\bar f} \cC_{/p}\to \cC$ is a limit
diagram.
\end{lemma}

This applies in particular to the case where $B=\Lambda^2_2$. In this case
we have $B^\triangleleft\simeq \Delta^1\times \Delta^1$.

\begin{proof}
We let $q\colon B\star A\to \cC$ denote the diagram corresponding to $f$. We
let $q_0$ denote the restriction of $q$ to $B$. Since the inclusion
$B\subseteq B\star A$ is left anodyne by \cite{Lu1}*{Lemma 4.2.3.6}, the map
$\cC_{/q}\to \cC_{/q_0}$ is a trivial Kan fibration by
\cite{Lu1}*{Proposition 2.1.2.5}. Therefore, $\cC_{/q}$ admits a final
object if and only if $\cC_{/q_0}$ admits a final object, and an object of
$\cC_{/q}$ is a final object if and only if its image in $\cC_{/q_0}$ is a
final object.
\end{proof}

\begin{remark}\label{4re:deligne}
In the situation of Theorem \ref{4pr:descent}, for every $\infty$-category
$\cD$, the functor
\begin{equation}\label{4eq:fun}
\Fun(\cC,\cD)\to \Fun(\delta_2^*\cC_{\cE_1,\cE_2},\cD)
\end{equation}
is an equivalence of $\infty$-categories. This generalizes Deligne's gluing
result \cite{SGA4XVII}*{Proposition 3.3.2}, which can be interpreted as
saying that \eqref{4eq:fun} induces a bijection between the sets of
equivalence classes of objects when $\cC$ is the nerve of an ordinary
category and $\cD=\N(\cat)$.
\end{remark}

In the remaining part of this section, we will study a variant of the
diagonal functor $\delta_2^*\colon \Sset[2]\to \Sset$, which will allow,
among other things, to express the $\infty$-category of correspondences in
\cite{Gait} in terms of our multisimplicial nerves. This will not be used in
the later sections of this article. Therefore, the uninterested reader may
safely skip the remaining part of this section and proceed to Section
\ref{5ss}.

\begin{definition}\label{4de:correspondence}
Let $X$ be a bisimplicial set. We let $\delta_{2\nabla}^*X$ denote the
simplicial set defined by
$(\delta_{2\nabla}^*X)_n=\Hom_{\Sset[2]}(\Cpt^n,X)$. This defines a functor
$\delta_{2\nabla}^*\colon \Sset[2]\to \Sset$.
\end{definition}

Recall that we have $(\delta_{2}^*X)_n\simeq \Hom_{\Sset[2]}(\Delta^{n,n},
X)$.

\begin{theorem}\label{4th:correspondence}
The map
\[f\colon \delta_2^* X\to \delta_{2\nabla}^* X\]
induced by the inclusions $\Cpt^n\subseteq\Delta^{n,n}$ is a categorical
equivalence.
\end{theorem}

Under our convention of representing the first direction vertically and
second direction horizontally as in \eqref{4eq:Cpt}, the map can be
described as ``forgetting the lower-left corner''. Before proving the
theorem, let us look at a few examples.

\begin{example}
For $X=\Cpt^n$, we have a canonical isomorphism $\CCpt^n\simeq
\delta_{2\nabla}^*\Cpt^n$. An $m$-simplex $\alpha$ of $\CCpt^n$ is given by
a sequence $(i_0,j_0)\le \dots \le (i_m,j_m)$ in $\RCpt^n$. The isomorphism
carries $\alpha$ to the $m$-simplex of $\delta_{2\nabla}^*\Cpt^n$ given by
the map of bisimplicial sets $\Cpt^m\to \Cpt^n$ carrying $(a,b)$ to
$(i_a,j_b)$. The map $f$ can be identified with the inclusion
$\square^n\subseteq \CCpt^n$, which is inner anodyne (Lemma
\ref{4le:cpt_inner}), and in particular a categorical equivalence.
\end{example}

\begin{example}
In the situation of Theorem \ref{4pr:descent}, there exists a non-canonical
categorical equivalence $\delta_{2\nabla}^*\cC_{\cE_1,\cE_2}\to \cC$ by
Theorem \ref{4th:correspondence} applied to the bisimplicial set
$\cC_{\cE_1,\cE_2}$.
\end{example}

\begin{example}\label{4ex:corr}
Given a $2$-marked $\infty$-category $(\cC,\cE_1,\cE_2)$ satisfying certain
conditions, Gaitsgory defined an $\infty$-category of correspondences
$\cC_{\corr:\cE_1,\cE_2}$ \cite{Gait}*{5.1.2} ($\cE_1=vert$, $\cE_2=horiz$
in his notation) following an idea of Lurie. More generally, given an
\emph{arbitrary} $2$-marked $\infty$-category $(\cC,\cE_1,\cE_2)$, using the
above functor $\delta_{2\nabla}^*$, one can define the \emph{simplicial set
of correspondences} to be
\[\cC_{\corr:\cE_1,\cE_2}\colonequals
\delta_{2\nabla}^*(\op^2_{\{2\}}\cC^\cart_{\cE_1,\cE_2}).
\]
In other words,
we have
\[(\cC_{\corr:\cE_1,\cE_2})_n=\Hom_{\Sset[2]}(\Cpt^n,\op^2_{\{2\}}\cC^\cart_{\cE_1,\cE_2}).\]
Applying Theorem \ref{4th:correspondence} to the bisimplicial set
$\op^2_{\{2\}}\cC^\cart_{\cE_1,\cE_2}$, we know that the natural map
\[\delta^*_{2,\{2\}}\cC^\cart_{\cE_1,\cE_2}\to\cC_{\corr:\cE_1,\cE_2},\]
given by ``forgetting the lower-right corner'', is a categorical
equivalence.
\end{example}

\begin{proof}[Proof of Theorem \ref{4th:correspondence}]
The proof is very similar to that of Theorem
\ref{4th:multisimplicial_descent}. Consider a commutative diagram
\[\xymatrix{\delta_2^* X\ar[r]^-v\ar[d]_f & \Fun(\Delta^l,\cD)\ar[d]^p\\
\delta_{2\nabla}^* X\ar[r]^-w & \Fun(\partial \Delta^l,\cD)}\]
as in Lemma \ref{1le:categorical_equivalence}. Let $\sigma$ be an $n$-simplex
of $\delta_{2\nabla}^* X$, corresponding to a map $\tau\colon \Cpt^n\to X$.
Consider the commutative diagram
\begin{align}\label{7eq:descent}
\xymatrix{\cN(\sigma)\ar[d]\ar[r] & \Fun(\Delta^l\times\CCpt^n,\cD)
\ar[d]^-{\RES_1}\ar[rr]^-{\RES_2} && \Fun(\Delta^l\times \Delta^n,\cD)\ar[d]\ar@/^4pc/[dd]^{\RES_4}\\
\Delta^0\ar[r]^-{h}\ar[rd]_{v\circ \delta_2^* \tau} &  \Fun(H,\cD)\ar[r]\ar[d]&
\Fun(\partial \Delta^l\times \CCpt^n,\cD)\ar[r]^-{\RES_2} & \Fun(\partial\Delta^l\times\Delta^n,\cD)\\
&\Fun(\Delta^l\times \square^n,\cD)\ar[rr]^-{\RES_3}&&\Fun(\Delta^l\times D^n,\cD).}
\end{align}
In the above diagram,
\begin{itemize}
  \item $H$ and the maps $\RES_i$, $1\le i\le 4$ are defined as in the
      proof of Theorem \ref{4th:multisimplicial_descent};

  \item $h$ is the amalgamation of $v\circ \delta_2^*\tau\colon
      \square^n\to \Fun(\Delta^l,\cD)$ and $w\circ
      \delta_{2\nabla}^*\tau\colon \CCpt^n\to
      \Fun(\partial\Delta^l,\cD)$;

  \item $\cN(\sigma)$ is defined so that the upper left square is a
      pullback square;

  \item the unnamed arrows in the middle column and in the upper right
      square are obvious restrictions.
\end{itemize}

By \cite{Lu1}*{Corollaries 2.3.2.4, 2.3.2.5}, the map $j\colon
H\hookrightarrow \Delta^l\times \CCpt^n$ is inner anodyne, and consequently
$\RES_1$ is a trivial Kan fibration. It follows that $\cN(\sigma)$ is a
contractible Kan complex.

We let $\Phi(\sigma)\colon\cN(\sigma)\to\Fun(\Delta^l\times \Delta^n,\cD)$
denote the composition of the upper horizontal arrows in \eqref{7eq:descent}.
Then $\Phi(\sigma)$ induces a map
\[\cN(\sigma)^\sharp
\times(\Delta^n)^\flat\to\Fun(\Delta^l,\cD)^\flat\subseteq
\Fun(\Delta^l,\cD)^{\natural}.
\]
Thus $\Phi(\sigma)$ induces a map $\cN(\sigma)\to
\Map^\sharp((\Delta^n)^\flat, \Fun(\Delta^l,\cD)^{\natural})$, which we still
denote by $\Phi(\sigma)$. This construction is functorial in $\sigma$, giving
rise to a morphism $\Phi\colon \cN\to \Map[ \delta_{2\nabla}^*
X,\Fun(\Delta^l,\cD)]$ in the category $(\Sset)^{(\del_{/ \delta_{2\nabla}^*
X})^{op}}$.

The composition $\Delta^n\hookrightarrow \CCpt^n\xto{\delta^*_{2\nabla}\tau}
X$, where the first map is the diagonal embedding, is $\sigma$. Thus the
composition of the middle row of \eqref{4eq:descent} is given by $w(\sigma)$.
Thus $\Map[ \delta_{2\nabla}^* X,p]\circ \Phi\colon\cN\to \Map[
\delta_{2\nabla}^* X,\Fun(\partial\Delta^l,\cD)]$ factorizes through the
morphism $\Delta^0_{(\del_{/ \delta_{2\nabla}^* X})^{op}}\to \Map[
\delta_{2\nabla}^* X,\Fun(\partial \Delta^l,\cD)]$ corresponding to $w$ via
Remark \ref{2re:functors}.

Now let $\sigma'$ be an $n$-simplex of $\delta_2^* X$ corresponding to a map
$\tau'\colon \Delta^{n,n}\to X$. The restriction of $v\circ
\delta^2\tau'\colon \Delta^{[n,n]}\to \Fun(\Delta^l,\cD)$ to
$\CCpt^n\subseteq \Delta^{[n,n]}$ provides a vertex of $\nu(\sigma')$ of
$\cN(f(\sigma'))$, whose image under $\Phi(f(\sigma'))$ is $v(\sigma')$.
This construction is functorial in $\sigma'$, giving rise to $\nu\in
\Gamma(f^*\cN)_0$ such that $f^*\Phi\circ \nu=v$. Applying Proposition
\ref{1pr:extension} to $\Phi$, the map $f\colon \delta_2^*X\to
\delta_{2\nabla}^* X$ and the global section $\nu$ of $f^*\cN$, we obtain a
map $u\colon \delta_{2\nabla}^* X\to \Fun(\Delta^l,\cD)$ satisfying $p\circ
u=w$ such that $u\circ f$ and $v$ are homotopic over $\Fun(\partial
\Delta^l,\cD)$, as desired.
\end{proof}

\section{Cartesian gluing}\label{5ss}
In Section \ref{4ss}, we gave a general criterion for multisimplicial
descent (Theorem \ref{4th:multisimplicial_descent}). It is often impossible
to apply the theorem directly to Cartesian multisimplicial nerves, as the
simplicial set of compactifications for Cartesian tilings is often empty for
$n\ge 2$. However, we have seen that certain bigger multisimplicial nerves
do satisfy multisimplicial descent (Theorems
\ref{4co:multisimplicial_descent} and \ref{4pr:descent}). In this section,
we complete the picture by comparing Cartesian multisimplicial nerves with
bigger multisimplicial nerves. The basic idea is to decompose a square
$\sigma$ in an $\infty$-category
\begin{equation}\label{5eq:square}
\xymatrix{w \ar[r]\ar[d]& y\ar[d]\\ z\ar[r]&x}
\end{equation}
into a diagram $\sigma'$
\begin{equation}\label{5eq:decompose}
\xymatrix{w\ar[rd]\\&w' \ar[r]\ar[d]&y\ar[d]\\&z\ar[r]&x,}
\end{equation}
where the inner square is Cartesian. More precisely, $\sigma'$ is a right
Kan extension of $\sigma$ along the full embedding $\Delta^1\times
\Delta^1\to (\Delta^1\times \Delta^1)^\triangleleft$ carrying $(0,0)$ to the
cone point $-\infty$ and carrying every other vertex $(i,j)$ to $(i,j)$. To
deal with the oblique arrow $f\colon w'\to w$, we consider the square
\[\xymatrix{w\ar[r]^{\id_w}\ar[d]_{\id_w} & w\ar[d]^f\\
w\ar[r]^f & w'.}
\]
If this square is a pullback square (which happens exactly when $f$ is a
monomorphism), we stop. Otherwise, we apply the above procedure recursively,
which leads to the diagonal map $\delta\colon w\to w\times_{w'}w$ of $f$,
and the diagonal of $\delta$, and so on.

To state our result, we introduce a bit of notation. For sets of edges
$\cE_1$, $\cE_2$, $\cE$ of an $\infty$-category $\cC$, we let
$\cE_1*_\cC^\cE\cE_2\subseteq \cE_1*_\cC\cE_2$ denote the set of squares
that admit a decomposition as above with $w\to w'$ in $\cE$. We have
$\cE_1*^\cart_\cC\cE_2=\cE_1*_\cC^\cE\cE_2$, where $\cE$ is the set of
equivalences of $\cC$ (or the set of degenerate edges of $\cC$).

The main result of this section is the following.

\begin{theorem}[Cartesian gluing]\label{5th:cartesian_gluing}
Let $\cC$ be an $\infty$-category and $K$ a finite set. Let
$(\cC,\cT)\subseteq(\cC,\cT')$ be two $(\{1,2\}\coprod K)$-tiled
$\infty$-categories such that $\cT_j=\cT'_j$ for all $j\in \{1,2\}\coprod
K$, and $\cT_{jj'}=\cT'_{jj'}$ for all $j,j'\in \{1,2\}\coprod K$ with
$j\neq j'$, except when $(j,j')=(1,2)$ or $(2,1)$, we have
$\cT_{12}=\cT_1*_\cC^\cart \cT_2$ and $\cT'_{12}=\cT_1*_\cC^{\cE}\cT_2$,
where $\cE\subseteq \cT_1\cap \cT_2$ is a set of edges of $\cC$. Suppose
that the following conditions are satisfied:
\begin{enumerate}
  \item $\cT_1*_{\cC}\cT_2=\cT_1*^{\cC_1}_{\cC}\cT_2$; $\cT_1$ (resp.\
      $\cT_2$) is stable under composition and pullback by $\cT_2$
      (resp.\ $\cT_1$).

  \item Every morphism $f$ in $\cE$ is $n$-truncated for some integer
      $n\ge -2$ (which may depend on $f$) \cite{Lu1}*{Definition 5.5.6.8}.
      Moreover, $\cE$ is stable under composition, pullback by $\cT_1\cup
      \cT_2$, and taking diagonals: for every edge $y\to x$ in $\cE$, its
      diagonal $y\to y\times_x y$ is in $\cE$ (the pullback $y\times_x y$
      exists in $\cC$ by the first part of Condition (1)).

  \item For every $k\in K$, the set $\cT_{1k}$ (resp.\ $\cT_{2k}$) is
      stable under composition and pullback by $\cT_{2k}$ (resp.\
      $\cT_{1k}$) in the first direction, and $\cT_{1k}\cap \cT_{2k}$ is
      stable under pullback by $\cT_{1k}\cup \cT_{2k}$ in the first
      direction. Moreover, we have
      \begin{equation}\label{5eq:cond3}
      \cT_{1k}*_{\Fun(\Delta^1,\cC)}^{\cE*_\cC
      \cT_k}\cT_{2k}=\cT_{1k}*_{\Fun(\Delta^1,\cC)}^{
      (\cE*_\cC\cT_k)\cap\cT_{1k}\cap \cT_{2k}}\cT_{2k}.
      \end{equation}
      See Remark \ref{5re:explicit} (3) below for an explicit description
      of the meaning of \eqref{5eq:cond3}.

  \item For every pair $k,k'\in K$ with $k\neq k'$, and every
      \emph{Cartesian} square of the form \eqref{5eq:square} of the
      $\infty$-category $\Fun(\Delta^1\times \Delta^1,\cC)$ (whose
      vertices are regarded as squares of $\cC$ in directions $k,k'$),
      with $y\to x$ given by a $(1,1,1)$-simplex of
      $\delta_*^{\{1,k,k'\}\square}(\cC,\cT)$ and $z\to x$ given by a
      $(1,1,1)$-simplex of $\delta_*^{\{2,k,k'\}\square}(\cC,\cT)$ (where
      the obvious restrictions of $\cT$ are still denoted by $\cT$), we
      have $w\in\cT_{kk'}$.
\end{enumerate}
Then, for any subset $L\subseteq K$, the inclusion map
\[\iota\colon \delta^*_{\{1,2\}\amalg K,L}\delta_*^{(\{1,2\}\amalg K)\square}(\cC,\cT)
\hookrightarrow\delta^*_{\{1,2\}\amalg K,L}\delta_*^{(\{1,2\}\amalg K)\square}(\cC,\cT')
\] is a categorical equivalence.
\end{theorem}

We note that unlike the theorems in the last section, Theorem
\ref{5th:cartesian_gluing} is symmetric in $\cE_1$ and $\cE_2$.

\begin{remark}
Let us recall some facts about $n$-truncated morphisms, $n\ge -2$, in an
$\infty$-category $\cC$.
\begin{itemize}
  \item A morphism $f$ of $\cC$ is $(-2)$-truncated (resp.\
      $(-1)$-truncated) if and only if $f$ is an equivalence (resp.\ a
      monomorphism).

  \item The set of $n$-truncated morphisms of $\cC$ is admissible. Indeed,
      the set is stable under pullback by \cite{Lu1}*{Remark 5.5.6.12}. It
      follows from the long exact sequence of homotopy groups that the set
      is stable under composition. Moreover, given a $2$-simplex $\sigma$
      of $\cC$ of the form \eqref{3eq:2cell}, if $r=\sigma\circ d^2_1$ is
      $n$-truncated and $p=\sigma\circ d^2_0$ is $(n+1)$-truncated, then
      $q=\sigma\circ d^2_2$ is $n$-truncated.

  \item Given a morphism $f\colon y\to x$ of $\cC$ such that the fiber
      product $y\times_x y$ exists, $f$ is $(n+1)$-truncated if and only
      if its diagonal $y\to y\times_x y$ is $n$-truncated
      (\cite{Lu1}*{Lemma 5.5.6.15} assumes that $\cC$ admits finite
      limits, but the proof only uses the existence of $y\times_x y$).

  \item In an $(n+1)$-truncated category \cite{Lu1}*{Definition 2.3.4.1},
      every morphism is $n$-truncated by \cite{Lu1}*{Proposition
      2.3.4.18}.
\end{itemize}
\end{remark}

\begin{remark}\label{5re:explicit}
We have the following remarks concerning the conditions in the above
theorem.
\begin{enumerate}
  \item The conditions of the theorem imply that the sets $\cT_j$,
      $\cT_{ij}$, $\cT'_{ij}$ and $\cE$ are all stable under equivalence.
      Indeed, the second part of Condition (1) implies that $\cT_1$ and
      $\cT_2$ are stable under equivalence. The second part of Condition
      (2) implies that $\cE$ is stable under equivalence. It follows that
      $\cT_{12}$ and $\cT'_{12}$ are stable under equivalence. The first
      part of Condition (3) implies that $\cT_{1k}$ and $\cT_{2k}$ are
      stable under equivalence. It follows that $\cT_k$ is stable under
      equivalence. Finally, Condition (4) implies that $\cT_{kk'}$ is
      stable under equivalence.

  \item The first part of Condition (1) is satisfied if morphisms in
      $\cT_1$ admits pullback in $\cC$ by morphisms in $\cT_2$.

  \item The left hand side of \eqref{5eq:cond3} clearly contains the right
      hand side. Since $\cT_{1k}$ and $\cT_{2k}$ are stable under
      equivalence, the meaning of the equality is as follows. Consider a
      square of the form \eqref{5eq:square} in the $\infty$-category
      $\Fun(\Delta^1,\cC)$ (whose vertices are regarded as edges of $\cC$
      in direction $k$), such that $y\to x,w\to z\in\cT_{1k}$ and $z\to x,
      w\to y\in\cT_{2k}$. If it has a decomposition of the form
      \eqref{5eq:decompose} with $w\to w'$ in $\cE*_\cC \cT_k$, then $w\to
      w'$ is in $\cT_{1k}\cap\cT_{2k}$.

  \item Suppose that we have $\cT_{jj'}=\cT_j*^\cart_{\cC}\cT_{j'}$ for
      all $j,j'\in \{1,2\}\coprod K$ with $j\neq j'$. Then the identity
      \eqref{5eq:cond3} holds automatically, by (the dual of)
      \cite{Lu1}*{Lemma 4.4.2.1}. Moreover, the first part of Condition
      (3) implies Condition (4). To see this, consider a square $\sigma$
      as in Condition (4). Applying Lemma \ref{3le:cart} to the
      corresponding cube (whose vertices are edges of $\cC$ in direction
      $k'$, say), we get $w\in \cC_1*^\cart_\cC \cC_1$. Applying the first
      part of Condition (3) to the images of $\sigma$ under the maps
      $\Fun(\Delta^1\times \Delta^1,\cC)\to \Fun(\Delta^1,\cC)$ induced by
      $d^1_0\times \id$, $d^1_1\times \id$, we get $w\in
      \cC_1*_\cC\cT_{k'}$. Similarly, we have $w\in \cT_k*_\cC\cC_1$.

  \item Suppose that we have $\cT_{jj'}=\cT_j*^\cart_{\cC}\cT_{j'}$ for
      all $j,j'\in \{1,2\}\coprod K$ with $j\neq j'$, and moreover that
      $\cT_k$ is stable under pullback by either $\cT_1$ or $\cT_2$ for
      each $k\in K$. Then, by Remark \ref{3re:cart_square}, Conditions (1)
      and (2) imply Condition (3), which in turn implies Condition (4).
\end{enumerate}
\end{remark}

Combining Theorem \ref{5th:cartesian_gluing} with Theorems
\ref{4co:multisimplicial_descent} and \ref{4pr:descent}, we obtain the
following.

\begin{theorem}\label{5th:cartesian_descent}
Let $\cC$ be an $\infty$-category and let $K$ be a finite set. We are given
a $(\{0,1,2\}\coprod K)$-marked $\infty$-category
$(\cC,\cE_0,\cE_1,\cE_2,\{\cE_k\}_{k\in K})$ such that
\begin{enumerate}
  \item $\cE_1,\cE_2\subseteq \cE_0$; $\cE_0$ is stable under
      composition. Moreover, for every morphism $f$ in $\cE_0$, there
      exists a $2$-simplex of $\cC$ of the form
      \[\xymatrix{&y\ar[rd]^p\\z\ar[ru]^q\ar[rr]^f && x}\]
      with $p\in\cE_1$ and $q\in\cE_2$.

  \item Every morphism $f$ in $\cE_1\cap \cE_2$ is $n$-truncated for some
      integer $n\ge -2$ (which may depend on $f$).

  \item $\cE_k$ is stable under pullback by $\cE_1$ for every $k\in K$.

  \item Edges in $\cE_1$ admit pullbacks in $\cC$ by edges in $\cE_k$ for
      all $k\in K$.

  \item $\cE_1*_\cC \cE_2=\cE_1 *_\cC^{\cE_1\cap \cE_2} \cE_2$. Moreover,
      $\cE_1$ (resp.\ $\cE_2$) is stable under composition and pullback
      by $\cE_k$ and $\cE_2$ (resp.\ $\cE_1$); $\cE_1\cap \cE_2$ is
      stable under pullback by $\cE_1\cup \cE_2$.

  \item $\cC_{\cE_1}$ admits pullbacks and pullbacks are preserved by the
      functor $\cC_{\cE_1}\to \cC_{\cE_0}$.
\end{enumerate}
Then, for every subset $L\subseteq K$, the natural map
\[g\colon \delta^*_{\{1,2\}\amalg K,L}\cC_{\cE_1,\cE_2,\{\cE_k\}_{k\in K}}^\cart
\to \delta^*_{\{0\}\amalg K,L}\cC_{\cE_0,\{\cE_k\}_{k\in K}}^\cart
\]
is a categorical equivalence (see Definition \ref{3de:cartesian_nerve} for
the notation).
\end{theorem}

\begin{remark}
If $\cC$ admits pullbacks and $\cE_1$, $\cE_2$ are admissible, then
Conditions (4), (5), and (6) of Theorem \ref{5th:cartesian_descent} hold.
Moreover, in this case, Condition (1) of Theorem \ref{5th:cartesian_descent}
implies that $\cE_0$ is admissible by Remark \ref{3re:admissible}. Indeed,
$\cE_0$ is clearly stable under pullback, and given a $2$-simplex as in
Condition (1), we have a diagram
\[\xymatrix{z\ar[r]^-{d_q}\ar[rd]_-{d_f} & z\times_y z\ar[r]\ar[d]&y\ar[d]^-{d_p}\\
&z\times_x z\ar[r] & y\times_x y,}
\]
where the square is a pullback by Lemma \ref{5le:pullback} below, so that
the diagonal $d_f$ of $f$ belongs to $\cE_0$.
\end{remark}

\begin{lemma}\label{5le:pullback}
Let $\cC$ be an $\infty$-category admitting pullbacks. Consider two
$2$-simplices of $\cC$ sharing an edge as depicted by the diagram
\[\xymatrix{z\ar[r]\ar[rd]&x'\ar[d]&y\ar[l]\ar[ld]\\
&x.}
\]
Then we have a pullback square
\[\xymatrix{y\times_{x'} z\ar[d]\ar[r]& x'\ar[d]\\
y\times_{x} z\ar[r] & x'\times_x x',}
\]
where the right vertical arrow is the diagonal of $x'\to x$.
\end{lemma}

\begin{proof}
Indeed, we have a diagram
\[\xymatrix{y\times_{x'}z\ar[rr]\ar[dd]\ar@{-->}[rd]&& y\ar[dd]\ar[rd]\\
&y\times_{x} z\ar@{-->}[rr]\ar@{-->}[dd] && y\times_x x'\ar[r]\ar[dd] & y\ar[dd]\\
z\ar[rr]\ar[rd] && x'\ar[rd]\\
&x'\times_x z\ar[rr]\ar[d] && x'\times_x x'\ar[r]\ar[d] & x'\ar[d]\\
&z\ar[rr] && x'\ar[r] & x}
\]
where the front face of the cube and the squares on the back page are
pullbacks. It follows that the other two faces of the cube containing $x'$
are pullbacks. Therefore, all the faces of the cube are pullbacks.
\end{proof}

\begin{proof}[Proof of Theorem \ref{5th:cartesian_descent}]
Denote by $(\cC,\cT)$ the $(\{1,2\}\coprod K)$-tiled simplicial set as in
Theorem \ref{4co:multisimplicial_descent}. Then the map $g$ factorizes as
\[\delta^*_{\{1,2\}\amalg K,L}\cC_{\cE_1,\cE_2,\{\cE_k\}_{k\in K}}^\cart
\xrightarrow{\iota}\delta^*_{\{1,2\}\amalg K,L}\delta_*^{(\{1,2\}\amalg
K)\square}(\cC,\cT) \xrightarrow{f}\delta^*_{\{0\}\amalg
K,L}\cC_{\cE_0,\{\cE_k\}_{k\in K}}^\cart.
\]
By Theorem \ref{5th:cartesian_gluing} applied to the inclusion
$(\cC,(\cE_1,\cE_2,\{\cE_k\}_{k\in K})^\cart)\subseteq(\cC,\cT)$ (see
Definition \ref{3de:cartesian_nerve} for the notation) and $\cE=\cE_1\cap
\cE_2$, the inclusion $\iota$ is a categorical equivalence. Indeed, by
Condition (3) of Theorem \ref{5th:cartesian_descent} and Remark
\ref{5re:explicit} (5), it suffices to check Conditions (1) and (2) of
Theorem \ref{5th:cartesian_gluing}. The first part of Condition (2) of
Theorem \ref{5th:cartesian_gluing} is Condition (2) of Theorem
\ref{5th:cartesian_descent}. Condition (1) and the second part of Condition
(2) of Theorem \ref{5th:cartesian_gluing} follow from Condition (5) of
Theorem \ref{5th:cartesian_descent}. To show that $f$ is a categorical
equivalence as well, we use Theorem \ref{4co:multisimplicial_descent} (with
$\alpha=1$). Conditions (1) and (2) of Theorem
\ref{4co:multisimplicial_descent} follow from Condition (1) of Theorem
\ref{5th:cartesian_descent}. Condition (3) of Theorem
\ref{4co:multisimplicial_descent} follows from Condition (5) of Theorem
\ref{5th:cartesian_descent}. Conditions (4) and (5) of Theorem
\ref{4co:multisimplicial_descent} are Conditions (3) and (4) of Theorem
\ref{5th:cartesian_descent}, respectively. It remains to check that
$\Komp^1_{\cC,\cE_1,\cE_2}(\tau)$ is weakly contractible for every simplex
$\tau$ of $\cC_{\cE_0}$, which follows from Theorem \ref{4pr:descent}
applied to $(\cC_{\cE_0},\cE_1,\cE_2)$. Conditions (1), (2), (3) of Theorem
\ref{4pr:descent} follow from Conditions (5), (6), (1) of Theorem
\ref{5th:cartesian_descent}, respectively.
\end{proof}

The rest of this section is devoted to the proof of Theorem
\ref{5th:cartesian_gluing}. A key ingredient in the proof is an analogue of
the diagram \eqref{5eq:decompose} for decompositions of simplices of higher
dimensions. Such decompositions are naturally encoded by certain lattices.
Let us review some basic terminology.

\begin{definition}[Lattice]
By a \emph{lattice} we mean a nonempty partially ordered set admitting
products (namely, infima) and coproducts (namely, suprema) of pairs of
elements, or equivalently, admitting finite nonempty products and
coproducts. In a lattice, we denote products by $\wedge$ and coproducts by
$\vee$. A lattice $P$ is said to be \emph{distributive} if $p\wedge (q\vee
r)=(p\wedge q)\vee (p\wedge r)$ for all $p,q,r\in P$, or equivalently,
$p\vee (q\wedge r)=(p\vee q)\wedge (p\vee r)$  for all $p,q,r\in P$
\cite{DP}*{Lemma 4.3}.

A map between lattices preserving finite nonempty products and coproducts is
called a \emph{morphism} of lattices. A morphism of lattices necessarily
preserves order.
\end{definition}

Note that a finite lattice admits arbitrary products and coproducts.

\begin{definition}[Sublattice]\label{5de:interval}
A nonempty subset of a lattice is called a \emph{sublattice} if it is stable
under finite nonempty products and coproducts. We endow the subset with the
induced lattice structure.
\end{definition}

Subsets of a lattice $P$ of the forms $P_{p/}$, $P_{/q}$, $P_{p//q}$ for
$p\le q$ in $P$ are necessarily sublattices of $P$.

\begin{definition}[Up-set lattice]
Let $P$ be a partially ordered set. A subset $Q$ of $P$ is called an
\emph{up-set} if $q\in Q$ and $p\ge q$ with $p\in P$ imply $p\in Q$. We
order the set $\cU(P)$ of up-sets of $P$ by \emph{inverse inclusion}: $Q\le
Q'$ if and only if $Q\supseteq Q'$. Then $\cU(P)$ becomes a distributive
lattice admitting arbitrary products and coproducts. In fact, we have $Q\vee
Q'=Q\cap Q'$ and $Q\wedge Q'=Q\cup Q'$. We call $\cU(P)$ the \emph{up-set
lattice} of $P$.

We let $\varsigma^P\colon P\to \cU(P)$ denote the map carrying $p$ to
$P_{p/}$, which is a fully faithful functor (namely, an order embedding)
since we have chosen the inverse inclusion order on $\cU(P)$. Note that
$\varsigma^P$ preserves coproducts whenever they exist in $P$. On the other
hand, $\varsigma^P$ does not preserve the product of any family of elements,
unless the family admits a minimum.
\end{definition}

\begin{remark}
Although we do not need it in the sequel, let us recall the correspondence
between finite partially ordered sets and finite distributive lattices
\cite{DP}*{Chapter 5} via up-set lattices. An element $p$ of a lattice $L$
is said to be \emph{product-irreducible} if $p$ is not a final object
(namely, maximum) of $L$ and $p=a\wedge b$ implies $p=a$ or $p=b$ for all
$a,b\in L$. We let $\cI(L)\subseteq L$ denote the subset of
product-irreducible elements of $L$. The map $\varsigma^P$ factorizes to
give an embedding $P\to \cI(\cU(P))$, which is an isomorphism if $P$ is
finite. The map $\eta_L\colon L\to \cU(\cI(L))$ carrying $x$ to
$\cI(L)_{x/}$ is a morphism of lattices preserving initial and final
objects. Birkhoff's representation theorem states that $\eta_L$ is an
isomorphism for any finite distributive lattice $L$.
\end{remark}

We will need the following properties of up-set lattices.

\begin{remark}\label{5re:U1}
We have an isomorphism $\cU(P^\triangleright)\simeq \cU(P)^\triangleright$
carrying $Q\neq \emptyset$ to $Q\cap P$ and carrying $\emptyset$ to the cone
point of $\cU(P)^\triangleright$. In particular, $\cU(P)$ can be identified
with the sublattice of $\cU(P^\triangleright)$ spanned by nonempty up-sets
of $P^\triangleright$, or equivalently, up-sets of $P^\triangleright$ that
contain the cone point.
\end{remark}

\begin{remark}\label{5re:U2}
For $Q\in \cU(P)$, we have $Q\subseteq P$ and $\varsigma^P(Q)\subseteq
\varsigma^P(P)\subseteq \cU(P)$. Moreover, we have
$\varsigma^P(P)_{Q/}=\varsigma^P(Q)$. Thus a diagram $F\colon \N(\cU(P))\to
\cC$ in an $\infty$-category $\cC$ is a right Kan extension along
$\N(\varsigma^P)$ if and only if for every $Q\in \cU(P)$, the restriction of
$F$ to $\N(\varsigma^P(Q))^\triangleleft$ exhibits $F(Q)$ as the limit of
$F\res \N(\varsigma^P(Q))$. Note that when $Q\in\varsigma^P(P)$, the last
condition is automatic. To alleviate notation, we will write $\varsigma^P$
for $\N(\varsigma^P)$.
\end{remark}

\begin{definition}
Let $P$ and $P'$ be partially ordered sets and let $f\colon P'\to P$ be an
order-preserving map. The map $\cU^f\colon \cU(P)\to \cU(P')$ carrying $Q$
to $f^{-1}(Q)$ is a morphism of lattices preserving products and coproducts.
The functor $\cU^f$ admits a right adjoint $\cU_f\colon \cU(P')\to \cU(P)$
carrying an up-set $Q'$ of $P'$ to the up-set of $P$ generated by $f(Q')$.
In other words, $\cU_f(Q')=\bigcup_{q\in Q'} P_{f(q)/}$. The functor $\cU_f$
preserves products.
\end{definition}

We will need the following properties of the functor $\cU_f$.

\begin{remark}\label{5re:U3}
The following diagram commutes:
\[\xymatrix{P'\ar[r]^-{\varsigma^{P'}}\ar[d]_f &
\cU(P')\ar[d]^{\cU_f}\\ P\ar[r]^-{\varsigma^P}& \cU(P).}\]
\end{remark}

\begin{remark}\label{5re:U5}
Suppose $P'$ admits nonempty coproducts and $f$ preserves such coproducts.
For $Q'\in \cU(P')$, the map $f$ restricts to a map $Q'\to \cU_f(Q')$. We
claim that the induced map $\N(Q')^{op}\to \N(\cU_f(Q'))^{op}$ is cofinal.
Indeed, for every $Q\in \cU_f(Q')$, the partially ordered set
$Q'\times_{\cU_f(Q')}\cU_f(Q')_{/Q}$ is nonempty and admits nonempty
coproducts, hence admits a final object. Thus
$\N(Q')\times_{\N(\cU_f(Q'))}\N(\cU_f(Q'))_{/Q}$ is weakly contractible and
we apply the criterion of cofinality \cite{Lu1}*{Theorem 4.1.3.1}.

In this case, if $F\colon \N(\cU(P))\to \cC$ is a right Kan extension along
$\varsigma^P$, then $F\circ \N(\cU_f)\colon \N(\cU(P'))\to \cC$ is a right
Kan extension along $\varsigma^{P'}$. Indeed, by Remark \ref{5re:U2}, it
suffices to check that for every $Q'\in \cU(P')$ and every limit diagram
$\N(\cU_f(Q'))^\triangleleft\to \cC$, the induced map
$\N(Q')^\triangleleft\to \cC$ is a limit diagram, which follows from the
above cofinality by \cite{Lu1}*{Proposition 4.1.1.8}.
\end{remark}

\begin{lemma}\label{5le:U4}
If $P'$ admits coproducts indexed by a set $I$ and $f\colon P'\to P$
preserves such coproducts, then $\cU_f$ preserves coproducts indexed by $I$.
In particular, if $P$ admits coproducts of pairs of elements and $f$
preserves such coproducts, then $\cU_f$ is a morphism of lattices.
\end{lemma}

\begin{proof}
Let $Q'_i$, $i\in I$ be up-sets of $P'$. We have $\bigcap_{i\in
I}\cU_f(Q'_i)\supseteq \cU_f(\bigcap_{i\in I}Q'_i)$. To show the inclusion
in the other direction, let $y\in \bigcap_{i\in I}\cU_f(Q'_i)$. For each
$i\in I$, there exists $x_i\in Q'_i$ such that $f(x_i)\le y$. Thus
$f(\bigvee_{i\in I} x_i)= \bigvee_{i\in I}f(x_i)\le y$. This implies
$y\in\cU_f(\bigcap_{i\in I} Q'_i)$ since we have $\bigvee_{i\in I} x_i\in
\bigcap_{i\in I}Q'_i$.
\end{proof}

\begin{definition}[Exact square]\label{5de:exact}
By an \emph{exact square} in a lattice, we mean a square that is both a
pushout square and a pullback square, or, equivalently, a square of the form
\[\xymatrix{x\wedge y\ar[r]\ar[d] & x\ar[d]\\y \ar[r] & x\vee y.}\]
The left vertical arrow is called an \emph{exact pullback} of the right
vertical arrow.
\end{definition}

Exact squares in $\cU(P)$ correspond to pushout squares of sets. The
relevance of such squares is shown by the following lemmas.

\begin{lemma}\label{5le:U7}
Every right Kan extension $F\colon \N(\cU(P))\to \cC$ along $\varsigma^P$
carries exact squares to pullback squares. More generally, for every full
subcategory $R\subseteq \cU(P)$ containing $\varsigma^P(P)$, every functor
$F\colon \N(R)\to \cC$ that is a right Kan extension of $F\res
\N(\varsigma^P(P))$ carries exact squares to pullback squares.
\end{lemma}

\begin{proof}
Let
\begin{equation}\label{5eq:exact}
\xymatrix{Q\cup Q'\ar[r]\ar[d] & Q\ar[d]\\Q'\ar[r] & Q\cap Q'}
\end{equation}
be an exact square in $R$. We consider $S=\varsigma^P(P)\cup \{Q,Q',Q\cap
Q'\}$, satisfying $\varsigma^P(P)\subseteq S\subseteq R$. By
\cite{Lu1}*{Proposition 4.3.2.8}, $F$ is a right Kan extension of $F\res
\N(S)$. In particular, the restriction of $F$ exhibits $F(Q\cup Q')$ as a
limit of $F\res \N(S_{Q\cup Q'/})$. By Lemma \ref{3le:cofinal}, the map
$\Lambda^{2}_0\to \N(S_{Q\cup Q'/})^{op}$ induced by the square
\eqref{5eq:exact} is cofinal. Thus by \cite{Lu1}*{Proposition 4.1.1.8}, $F$
carries the square to a pullback square in $\cC$.
\end{proof}

\begin{lemma}\label{5le:exact}
Let $P$ be a finite partially ordered set. Every morphism $Q\to Q'$ in
$\cU(P)$ is the composition of a finite sequence of exact pullbacks of the
morphisms $\omega^P(x)\colon \varsigma^P(x)\to \varsigma^P(x)-\{x\}$ for
$x\in Q-Q'$.
\end{lemma}

\begin{proof}
We may choose a (finite) sequence of morphisms $Q=Q_0\to\cdots\to Q_m=Q'$
such that for $1\leq i\leq m$, $Q_{i-1}=Q_i\cup\{x_i\}$, where $x_i\in
Q-Q_i$ is a maximal element. For each $i$, the following diagram
\[\xymatrix{
Q_{i-1} \ar[r]\ar[d] & \varsigma^P(x_i) \ar[d]^-{\omega^P(x_i)} \\
Q_i \ar[r]& \varsigma^P(x_i)-\{x_i\}
}
\]
is an exact square. Thus the lemma follows.
\end{proof}

The following lattices encode generalizations of the diagram
\eqref{5eq:decompose}.

\begin{notation}\label{5no:cart}
For $n\ge 0$, we let $\Crt^n$ denote the sublattice of $\cU([n]\times [n])$
spanned by nonempty up-sets of $[n]\times [n]$ and we let $\varsigma^n\colon
[n]\times [n]\to \Crt^n$ denote the map induced by $\varsigma^{[n]\times
[n]}$ carrying $(p,q)$ to $([n]\times [n])_{(p,q)/}$. For an
order-preserving map $d\colon [m]\to [n]$, we let $\Crt(d)\colon \Crt^m\to
\Crt^n$ denote the map induced by $\cU_{d\times d}$. Put
$\Cart^n=\N(\Crt^n)$ and $\Cart(d)=\N(\Crt(d))$. We still write
$\varsigma^n$ for $\N(\varsigma^n)$.
\end{notation}

By Remark \ref{5re:U1}, we have $\Crt^n\simeq \cU([n]\times [n]-\{(n,n)\})$.
The definition of $\Crt^n$ given above has the advantage of being functorial
with respect to $[n]$. Every up-set of $[n]\times [n]$ has the form
$\{(p,q)\in [n]\times [n]\mid q\ge a_p\}$ for a sequence of integers $-1\le
a_0\le\dots\le a_n\le n$. Thus the cardinality of $\Crt^n$ is
$\binom{2n+2}{n+1}-1$.

Below are the Hasse diagrams of $\Crt^1$ and $\Crt^2$, rotated so that the
initial objects are shown in the upper-left corners. Bullets represent
elements in the images of $\varsigma^1$ and $\varsigma^2$. The dashed boxes
represent $\Crt^1_{0,1}$ and $\Crt^2_{1,2}$ (see Construction
\ref{5cs:boxplus} (1) below).
\begin{equation}\label{5eq:Hasse}
\begin{xy}
(0,5)="0000";
(5,0)*\cir<1.8pt>{}="00"**\dir{-}; (10,0)="01"**\dir{-};
(10,-5)="11"**\dir{-};
(5,-5)="10"**\dir{-};
"00"**\dir{-};
"0000"*{\bullet}; "01"*{\bullet}; "11"*{\bullet}; "10"*{\bullet};
(-2,-2)="b"; (-2,7)**\dir{--}; (12,7)**\dir{--}; (12,-2)**\dir{--}; "b"**\dir{--};
\end{xy}\qquad\qquad
\begin{xy}
(0,5)="000000"; (5,0)*\cir<1.8pt>{}="0000"**\dir{-};
(10,0)*\cir<1.8pt>{}="0001"**\dir{-}; (15,0)="0101"**\dir{-};
(5,-5)*\cir<1.8pt>{}="0010"; (10,-5)*\cir<1.8pt>{}="0011"**\dir{-};
(15,-5)*\cir<1.8pt>{}="0111"**\dir{-};
(5,-10)="1010";
(10,-10)*\cir<1.8pt>{}="1011"**\dir{-};
(15,-10)="1111"**\dir{-};
(12.5,-7.5)*\cir<1.8pt>{}="00";
(17.5,-7.5)*\cir<1.8pt>{}="01"**\dir{-};
(22.5,-7.5)="02"**\dir{-};
(12.5,-12.5)*\cir<1.8pt>{}="10";
(17.5,-12.5)*\cir<1.8pt>{}="11"**\dir{-};
(22.5,-12.5)="12"**\dir{-};
(12.5,-17.5)="20";
(17.5,-17.5)="21"**\dir{-};
(22.5,-17.5)="22"**\dir{-};
"00"; "10"**\dir{-}; "20"**\dir{-};
"01"; "11"**\dir{-}; "21"**\dir{-};
"02"; "12"**\dir{-}; "22"**\dir{-};
"0000"; "0010"**\dir{-}; "1010"**\dir{-};
"0001"; "0011"**\dir{-}; "1011"**\dir{-}; "10"**\dir{-};
"0101"; "0111"**\dir{-}; "1111"**\dir{-}; "11"**\dir{-};
"0011"; "00"**\dir{-};
"0111"; "01"**\dir{-};
"000000"*{\bullet}; "0101"*{\bullet}; "1010"*{\bullet}; "1111"*{\bullet};
"02"*{\bullet}; "12"*{\bullet}; "22"*{\bullet}; "21"*{\bullet}; "20"*{\bullet};
(3,-3)="b"; (3,-14.5)**\dir{--}; (24.5,-14.5)**\dir{--}; (24.5,-3)**\dir{--}; "b"**\dir{--};
\end{xy}
\end{equation}

The map $\Crt(d)$ is a morphism of lattices by Lemma \ref{5le:U4}. Moreover,
$\varsigma^n$ preserves coproducts and final objects. In particular,
$\varsigma^n(p,q)=\varsigma^n(p,0)\vee \varsigma^n(0,q)$. By Remark
\ref{5re:U3}, the maps $\varsigma^n$ for different $n$ are compatible with
$d$ in the sense that we have
$\Crt(d)(\varsigma^m(p,q))=\varsigma^n(d(p),d(q))$ for all $(p,q)\in
[m]\times [m]$.

By Remark \ref{5re:U2}, a diagram $F\colon \Cart^n\to \cC$ in an
$\infty$-category $\cC$ is a right Kan extension along $\varsigma^n$ if and
only if for every $Q\in \Crt^n$, the restriction of $F$ to
$\N(\varsigma^n(Q))^\triangleleft$ exhibits $F(Q)$ as the limit of $F\res
\N(\varsigma^n(Q))$. By Remark \ref{5re:U5}, if $F\colon \Cart^n\to \cC$ is
a right Kan extension along $\varsigma^n$, then $F\circ \Cart(d)\colon
\Cart^m\to \cC$ is a right Kan extension along $\varsigma^m$.

\begin{definition}\label{5de:cartesian}
Let $\cC$, $\cD$ be $\infty$-categories and let $\tau\colon \Delta^n\times
\Delta^n\times \cD\to \cC$ be a functor.  We define $\Kart(\tau)$, the
simplicial set of \emph{Cartesianizations of $\tau$}, to be the fiber of the
restriction map
\begin{align*}
\Fun(\Cart^n\times\cD,\cC)_\RKE \xrightarrow{\RES}
\Fun(\Delta^n\times\Delta^n\times\cD,\cC)
\end{align*}
at $\tau$. Here,
$\Fun(\Cart^n\times\cD,\cC)_\RKE\subseteq\Fun(\Cart^n\times\cD,\cC)$ is the
full subcategory spanned by functors $F\colon \Cart^n\times\cD\to\cC$ that
are right Kan extensions of $F\res\Delta^n\times\Delta^n\times\cD$ along
$\varsigma^n\times \id_\cD$.
\end{definition}

\begin{remark}
By \cite{Lu1}*{Proposition 4.3.2.9}, $\RES$ is the composition
\[\Fun(\Cart^n\times\cD,\cC)_\RKE\to \cK \hookrightarrow \Fun(\Delta^n\times\Delta^n\times\cD,\cC)\]
of a trivial Kan fibration with the inclusion of the full subcategory $\cK$
spanned by functors $\tau$ that admit right Kan extensions along
$\varsigma^n\times \id_\cD$. In particular, $\Kart(\tau)$ is a contractible
Kan complex if $\tau$ admits a right Kan extension along $\varsigma^n\times
\id_\cD$ and $\Kart(\tau)$ is empty otherwise.

If $\cC$ admits pullbacks, then $\RES$ is a trivial Kan fibration. Indeed,
in this case, every diagram $\N(Q)\to \cC$, where $Q\in \Crt^n$, admits a
limit by Lemma \ref{4le:over_pull}.
\end{remark}

The following projection map will play an important role.

\begin{notation}\label{5no:pi}
Let $n\ge 0$ be an integer. We define a morphism of lattices
\[\pi^n=(\pi^n_1,\pi^n_2)\colon \Crt^n\to [n]\times[n]\]
to be the composite of the morphism of lattices $-\vee \xi^n(n,n)\colon
\Crt^n\to \Crt^n_{n,n}$, where
$\xi^n(n,n)=\varsigma^n(n,0)\wedge\varsigma^n(0,n)$ and
$\Crt^n_{n,n}=\Crt^n_{\xi^n(n,n)/}$, and the isomorphism $\Crt^n_{n,n}\simeq
[n]\times[n]$ carrying $\xi^n(n,n)_{(p,q)/}=\varsigma^n(p,n)\wedge
\varsigma^n(n,q)$ to $(p,q)$. We still write $\pi^n$ for $\N(\pi^n)$.
\end{notation}

Note that $\varsigma^n$ is a left adjoint of $\pi^n$, hence a section of
$\pi^n$. We have the following characterizations of $\pi^n$: for
$Q\in\Crt^n$, we have
\begin{gather*}
\varsigma^n(\pi^n_1(Q),n)=Q\vee \varsigma^n(0,n),\quad \varsigma^n(n,\pi^n_2(Q))=Q\vee\varsigma^n(n,0),\\
\pi^n(Q)=\left(\min_{(p,q)\in Q}p,\min_{(p,q)\in Q} q\right).
\end{gather*}
The last equation implies that for every order-preserving map $d\colon
[m]\to [n]$, we have $\pi^n\circ \Crt(d)=(d\times d)\circ \pi^m$. Indeed,
for $Q\in \Crt^m$, we have
\[\pi^n(\Crt(d)(Q))=\left(\min_{(p,q)\in Q} d(p),\min_{(p,q)\in Q} d(q)\right)=(d\times d)(\pi^n(Q)).\]

\begin{lemma}\label{5le:kan}
Let $\cC$ be an $\infty$-category and $F\colon \Delta^n\times \Delta^n\to
\cC$ a diagram. The following conditions are equivalent:
\begin{enumerate}
\item $F$ is obtained from a map of bisimplicial sets $\Delta^{n,n}\to
    \cC^\cart_{\cC_1,\cC_1}$.

\item $F$ is a right Kan extension of $F\res \N(\xi^n(n,n))$.

\item $F\circ \pi^n\colon \Cart^n\to \cC$ is a right Kan extension along
    $\varsigma^n$.
\end{enumerate}
\end{lemma}

\begin{proof}
By Lemma \ref{3le:cofinal}, the map $\Lambda^2_0\to
\N(\xi^n(n,n)_{(p,q)/})^{op}$ induced by the square
\[
\xymatrix{(p,q)\ar[r]\ar[d] & (p,n)\ar[d]\\(n,q)\ar[r] & (n,n)}
\]
is cofinal. Thus, by \cite{Lu1}*{Proposition 4.1.1.8}, (2) is equivalent to
the condition that $F$ carries the above square to a pullback. This
condition is a special case of (1), and is equivalent to (1) by
\cite{Lu1}*{Lemma 4.4.2.1}.

Next we show that (2) implies (3). Assume that $F\circ \pi^n\colon
\Cart^n\to \cC$ is a right Kan extension along $\varsigma^n$. Then
$F(p,q)=F(\pi^n(\varsigma^n(p,n)\wedge \varsigma^n(n,q)))$ is a limit of
$F\res \N(\varsigma^n(p,n)\wedge \varsigma^n(n,q))$ by Remark \ref{5re:U2}.
This implies that $F$ is a right Kan extension of $F\res \N(\xi^n(n,n))$.

Finally we show that (3) implies (2). Assume that $F$ is a right Kan
extension of $F\res \N(\xi^n(n,n))$. Then, for every $Q\in \Crt^n$, the
restriction $F\res \N(Q)$ is a right Kan extension of $F\res \N(Q\vee
\xi^n(n,n))$. Indeed, for any $(p,q)\in Q$, we have $(Q\vee
\xi^n(n,n))_{(p,q)/}=\xi^n(n,n)_{(p,q)/}$. Moreover, the restriction of $F$
exhibits $F(\pi^n(Q))$ as the limit of $F\res Q\vee \xi^n(n,n)$ since
$\xi^n(n,n)_{\pi^n(Q)/}=Q\vee \xi^n(n,n)$. It follows that the restriction
of $F\circ \pi^n$ exhibits $(F\circ \pi^n)(Q)$ as a limit of $F\circ
\pi^n\res \N(\varsigma^n(Q))$. Therefore, $F\circ \pi^n\colon \Cart^n\to
\cC$ is a right Kan extension along $\varsigma^n$ by Remark \ref{5re:U2}.
\end{proof}

We now introduce a crucial $2$-marking on $\Cart^n$.

\begin{notation}
Let $n\geq 0$ be an integer. We define a $2$-marking $\cF=(\cF_1,\cF_2)$ on
$\Cart^n$ as follows. For $i=1,2$, we let $\bar\cF_i$ denote the set of edges
of $\epsilon^2_i \Delta^{n,n}$, so that $\delta^*_{2+}\Delta^{n,n}\simeq
(\Delta^n\times\Delta^n,\bar\cF_1,\bar \cF_2)$. We define
$\cF_i=(\pi^n)^{-1}(\bar\cF_i)$ for $i=1,2$. Graphically, $\cF_1$ (resp.\
$\cF_2$) consists of edges whose image under $\pi^n$ are vertical (resp.\
horizontal). Recall that $\cF$ induces a $2$-tiling $\cF^\cart$ defined by
$\cF^\cart_{12} =\cF_1*_{\Cart^n}^\cart \cF_2$.
\end{notation}

For an order-preserving map $d\colon [m]\to [n]$, the map $\Cart(d)$ induces
a map $(\Cart^m,\cF)\to(\Cart^n,\cF)$ of $2$-marked $\infty$-categories, and
a map $(\Cart^m,\cF^\cart)\to(\Cart^n,\cF^\cart)$ of $2$-tiled
$\infty$-categories.

\begin{construction}\label{5cs:YZ}
Consider a $(\{1,2\}\coprod K)$-tiled $\infty$-category $(\cC,\cT)$ and a
subset $L\subseteq K$. For brevity, we write $I$ for $\{1,2\}\coprod K$. We
consider the following two simplicial sets
\begin{align*}
Y^n(\cT)&=\epsilon^I_1\Map(\delta_*^{2+}(\Cart^n,\cF)\boxtimes
\Delta^{n_k\res k\in K}_L,\delta_*^{I\square} (\cC,\cT)),\\
Z^n(\cT)&=\epsilon^I_1\Map(\delta_*^{2\square}(\Cart^n,\cF^\cart)\boxtimes
\Delta^{n_k\res k\in K}_L,\delta_*^{I\square} (\cC,\cT)).
\end{align*}
We have a natural commutative diagram
\[\xymatrix{
& \Fun(\delta^*_2\delta_*^2\Cart^n\times\Delta^{[n_k]_{k\in K}}_L,\cC) \ar[d]
& \Fun(\Cart^n\times\Delta^{[n_k]_{k\in K}}_L,\cC)\ar[l]_-{f^n}\ar@/^1pc/[ld]_-{g^n}\ar@/^2pc/[ldd]_-{h^n} \\
Y^n(\cT) \ar@{^(->}[r]\ar[d] & \Fun(\delta^*_2\delta_*^{2+}(\Cart^n,\cF)\times\Delta^{[n_k]_{k\in K}}_L,\cC) \ar[d] \\
Z^n(\cT) \ar@{^(->}[r] & \Fun(\delta^*_2\delta_*^{2\square}(\Cart^n,\cF^\cart)\times\Delta^{[n_k]_{k\in K}}_L,\cC), }
\]
where
\begin{itemize}
  \item The vertical arrows are induced by the inclusions
      $\delta_*^{2\square}(\Cart^n,\cF^\cart)\subseteq
      \delta_*^{2+}(\Cart^n,\cF)\subseteq\delta_*^2\Cart^n$;

  \item $f^n$ is induced by the adjunction
      $\delta^*_2\delta_*^2\Cart^n\to \Cart^n$;

  \item $g^n$ and $h^n$ are compositions of $f^n$ and the vertical
      arrows;

  \item In the inclusion on the second row we have used the isomorphism
      \[\epsilon^I_1\Map(\delta_*^{2+}(\Cart^n,\cF)\boxtimes
      \Delta^{n_k\res k\in K}_L ,\delta_*^I\cC)\simeq
      \Fun(\delta^*_2\delta_*^{2+}(\Cart^n,\cF)\times \Delta^{[n_k]_{k\in
      K}}_L,\cC)\]
      in Remark \ref{3re:adjunction} and similarly for the inclusion on the third row.
\end{itemize}
Moreover, we have a commutative diagram
\[\xymatrix{
Y^n(\cT) \ar[r]^-{y^n(\cT)}\ar[d] & \Map(\delta^*_2\delta_*^{2+}(\Cart^n,\cF)\times\Delta^{[n_k]_{k\in K}},
\delta^*_{I,L}\delta_*^{I\square}(\cC,\cT)) \ar[d] \\
Z^n(\cT) \ar[r]^-{z^n(\cT)} & \Map(\delta^*_2\delta_*^{2\square}(\Cart^n,\cF^\cart)\times\Delta^{[n_k]_{k\in K}},
\delta^*_{I,L}\delta_*^{I\square} (\cC,\cT)), }
\]
where the vertical arrows are induced by the inclusion
$\delta_*^{2\square}(\Cart^n,\cF^\cart)\subseteq\delta_*^{2+}(\Cart^n,\cF)$
and the horizontal arrows are induced by $\delta^*_I$ in Remark
\ref{3re:adjunction}. In the above notation, we have kept the datum $\cT$ as
we will now let it vary.
\end{construction}

\begin{lemma}\label{5le:restriction}
Assume that we are in the situation of Theorem \ref{5th:cartesian_gluing}.
Let $\cE^i\subseteq \cE$ be the subset of $i$-truncated edges, and let
$\cT^i$ be the $(\{1,2\}\coprod K)$-tiling between $\cT$ and $\cT'$
determined by $\cT^i_{12}=\cT_1*_{\cC}^{\cE^i}\cT_2$. Then, for every map
$\tau\colon\Delta^{n,n,n_k\res k\in K}_L\to \delta_*^{(\{1,2\}\amalg
K)\square}(\cC,\cT^i)$, the simplicial set $\Kart(\tau)$ is a contractible
Kan complex and the restriction of the map $g^n$ (resp.\ $h^n$) to
$\Kart(\tau)\subseteq\Fun(\Cart^n\times\Delta^{[n_k]_{k\in K}}_L,\cC)$ has
image contained in $Y^n(\cT^i)$ (resp.\ $Z^n(\cT^{i-1})$ for $i\ge -1$).
Here $\Kart(\tau)$ is the simplicial set of Cartesianizations of $\tau$
(where $\tau$ is regarded as a functor
$\Delta^n\times\Delta^n\times\Delta^{[n_k]_{k\in K}}_L\to\cC$) in Definition
\ref{5de:cartesian}.
\end{lemma}

In particular, we have induced maps
\begin{align*}
g(\tau)\colonequals y^n(\cT^i)\circ g^n&\colon\Kart(\tau)\to\Map(\delta^*_2\delta_*^{2+}(\Cart^n,\cF)\times\Delta^{[n_k]_{k\in K}},
\delta^*_{\{1,2\}\amalg K,L}\delta_*^{(\{1,2\}\amalg K)\square} (\cC,\cT^i)),\\
h(\tau)\colonequals z^n(\cT^{i-1})\circ h^n&\colon\Kart(\tau)\to\Map(\delta^*_2\delta_*^{2\square}(\Cart^n,\cF^\cart)\times\Delta^{[n_k]_{k\in K}},
\delta^*_{\{1,2\}\amalg K,L}\delta_*^{(\{1,2\}\amalg K)\square} (\cC,\cT^{i-1})).
\end{align*}

\begin{proof}
Consider an equivalence $e$ in the $\infty$-category
\begin{equation}\label{5eq:restr}
\Fun(\delta^*_2\delta_*^{2+}(\Cart^n,\cF)\times
\Delta^{[n_k]_{k\in K}}_L,\cC)
\end{equation}
with one vertex in $Y^n(\cT^i)$. By Remark \ref{5re:explicit} (1), we know
that the other vertex is also in $Y^n(\cT^i)$. Moreover, we have
$\cE^{-2}*^\cart_\cC\cT_\alpha\subseteq \cT^i_{1\alpha}$ for $\alpha\in
\{2\}\coprod K$. It follows that $e$ is in $Y^n(\cT^i)$. Thus, for any
connected Kan complex $S$ contained in \eqref{5eq:restr}, either
$Y^n(\cT^i)\cap S=\emptyset$ or $S\subseteq Y^n(\cT^i)$. The same holds for
$Z^n(\cT^{i-1})$. As $\Kart(\tau)$ is either empty or a contractible Kan
complex, its images in $\infty$-categories are contained in connected Kan
complexes. Therefore, it suffices to find one vertex $F$ of $\Kart(\tau)$
satisfying $g^n(F)\in Y^n(\cT^i)$ and $h^n(F)\in Z^n(\cT^{i-1})$.

For clarity, let $G\colon\Delta^{[n,n,n_k]_{k\in K}}\to\cC$ be the functor
corresponding to $\tau$ (we have till now denoted $G$ by $\tau$). Note that
$G$ underlies a map of $(\{1,2\}\coprod K)$-tiled simplicial sets
$\delta^*_{I\square}\Delta^{n,n,n_k\res k\in K}_L\to (\cC,\cT^i)$. We let
$\bar \cG_\alpha$ denote the set of edges of $\epsilon^I_\alpha
\Delta^{n,n,n_k\res k\in K}_L$. Then
\[\delta^*_{I\square}\Delta^{n,n,n_k\res k\in K}_L\simeq\sfW\delta^*_{I+}\Delta^{n,n,n_k\res k\in K}_L,\quad
\delta^*_{I+}\Delta^{n,n,n_k\res k\in K}_L\simeq (\Delta^{[n,n,n_k]_{k\in
K}}_L,\{\bar\cG_\alpha\}_{\alpha\in I}).\] We define an $I$-marked
simplicial set $(\Cart^n\times \Delta^{[n_k]_{k\in
K}}_L,\{\cG_\alpha\}_{\alpha\in I})$ by $\cG_\alpha=(\varsigma^n\times
\id)^{-1}\bar \cG_\alpha$. The goal is to show that $G$ admits a right Kan
extension $F\colon\Cart^n\times \Delta^{[n_k]_{k\in K}}_L\to\cC$ along
$\varsigma^n\times \id$ such that $F$ sends squares in
$\cG_\alpha*\cG_\beta$ to squares in $\cT^i_{\alpha\beta}$ for
$\alpha,\beta\in I$, $\alpha\neq \beta$, and, for $i\ge -1$, $F$ sends
squares in $\cG_1*^\cart \cG_2$ to squares in $\cT^{i-1}_{12}$.

Let us first show that there exists a right Kan extension $F$ of $G$ along
$\varsigma^n\times \id$ such that for each vertex $(x,u)$ of $\Cart^n\times
\Delta^{[n_k]_{k\in K}}_L$, the morphism
$G(\pi^n(x),u)=F(\varsigma^n(\pi^n(x)),u)\to F(x,u)$ is in $\cE^i$. We
construct the restriction of $F$ to $\Cart^n_{\varsigma^n(p,0)/}\times
\Delta^{[n_k]_{k\in K}}_L$ by descending induction on $p$. In the case $p=n$,
$\Crt^n_{\varsigma^n(n,0)/}$ is contained in the image of $\varsigma^n$ and
there is nothing to prove. For $0\le p\le n-1$, and $x\in \Crt^n$ satisfying
$\pi^n(x)=(p,q)$, consider the commutative diagram
\begin{equation}\label{5eq:squares}
\xymatrix{\varsigma^n(p,q)\ar[r]\ar[d] & x\ar[r]\ar[d]& \varsigma^n(p,q')\ar[d]\\
\varsigma^n(p+1,q)\ar[r] & x\vee \varsigma^n(p+1,q)\ar[r] & \varsigma^n(p+1,q'),}
\end{equation}
where $q'=\min \{q_0\mid (p,q_0)\in x\}$. The right square is exact. The
vertical (resp.\ horizontal) arrows are in $\cF_1$ (resp.\ $\cF_2$). The
horizontal arrows in the left square are in $\cF_1\cap \cF_2$. By induction
hypothesis, the morphism $G(p+1,q,u)\to F(x\vee \varsigma^n(p+1,q),u)$ is in
$\cE^i$, so that $G(p,q,u)\to F(x\vee \varsigma^n(p+1,q),u)$ is in $\cT_1$,
since $\cT_1$ is stable under composition. Thus, by the assumption
$\cT_1*_\cC\cT_2=\cT_1*_\cC^{\cC_1}\cT_2$, the pullback $F(x\vee
\varsigma^n(p+1,q),u)\times_{G(p+1,q',u)} G(p,q',u)$ exists in $\cC$, which
provides $F(x,u)$ by the proof of Lemma \ref{5le:U7}. The morphism
$G(p,q,u)\to F(x,u)$ is the composition
\[G(p,q,u)\to G(p+1,q,u)\times_{G(p+1,q',u)} G(p,q',u)\to F(x,u),\]
where the first arrow is in $\cE^i$ by the assumption that $G$ carries
$\bar\cG_1*\bar\cG_2$ into $\cT^i_{12}=\cT_1*_\cC^{\cE^i}\cT_2$, and the
second arrow is in $\cE^i$ by the assumption that $\cE^i$ is stable under
pullback by $\cT_1$.

We claim that $F$ sends $\cG_1$ to $\cT_1$, $\cG_2$ to $\cT_2$, and
$\cG_1\cap\cG_2$ to $\cE^i$. Let $e\colon (x,u)\to (y,u)$ be an edge in
$\cG_1\cup \cG_2$, where $x\to y$ in $\cF_1\cup \cF_2$ and $u$ is a vertex of
$\Delta^{[n_k]_{k\in K}}$. We show by induction on $\# x$ that $F(e)\in
\cT_1$ for $e\in \cG_1$, $F(e)\in \cT_2$ for $e\in \cG_2$, and $F(e)\in
\cE^i$ for $\in \cG_1\cap \cG_2$. By Lemma \ref{5le:exact}, any morphism
$x\to y$ in $\Crt^n$ is a composition of a finite sequences of morphisms of
the following classes:
\begin{enumerate}
  \item An exact pullback  of $\omega^n(p,n)\colon \varsigma^n(p,n)\to
      \varsigma^n(p+1,n)$ by $c\in \cF_2$;

  \item An exact pullback of $\omega^n(n,q)\colon \varsigma^n(n,q)\to
      \varsigma^n(n,q+1)$ by $c\in \cF_1$;

  \item An exact pullback of $\omega^n(p,q)\colon \varsigma^n(p,q)\to
      \varsigma^n(p,q)-\{(p,q)\}$ by $c$,
\end{enumerate}
where we have $(p,q)\in [n-1]\times [n-1]$, and $c\colon x'\to y'$ satisfies
$\# x'<\# x$. If $e\in \cG_1$ (resp.\ $e\in\cG_2$), then class (2) (resp.\
(1)) does not appear. Since $\cT_1$, $\cT_2$ and $\cE^i$ are stable under
composition, we may assume that $x\to y$ is in one of the three classes. In
class (1), $\omega^n(p,n)$ is in $\varsigma^n(\bar\cF_1)$, and we conclude
by Lemma \ref{5le:U7} and the assumption that $\cT_1$ is stable under
pullback by $\cT_2$. In class (2), $\omega^n(n,q)$ is in
$\varsigma^n(\bar\cF_2)$, and we conclude by Lemma \ref{5le:U7} and the
assumption that $\cT_2$ is stable under pullback by $\cT_1$. In class (3),
$c$ is a composition of an edge in $\cF_1$ and an edge in $\cF_2$ both
satisfying the induction hypothesis, and we have a diagram
\[\xymatrix{\varsigma^n(p,q)\ar[rd]^{\omega^n(p,q)}\\
&\varsigma^n(p,q)-\{(p,q)\} \ar[r]\ar[d] & \varsigma^n(p,q+1) \ar[d] \\
&\varsigma^n(p+1,q) \ar[r] & \varsigma^n(p+1,q+1) }
\]
with exact square in $\Crt^n$. By Lemma \ref{5le:U7}, the morphism
$F(\omega^n(p,q)\times \id_u)$ can be identified with the induced morphism
\[G(p,q,u)\to G(p+1,q,u)\times_{G(p+1,q+1,u)} G(p,q+1,u),\]
which belongs to $\cE^i$ since $G$ carries $\bar\cG_1*\bar \cG_2$ into
$\cT^i_{12}=\cT_1*_\cC^{\cE^i}\cT_2$. We conclude by the assumption that
$\cE^i$ is stable under pullback by $\cT_1\cup \cT_2$. This finishes the
proof of the claim.

Similarly, applying Condition (3), we see that $F$ carries $\cG_1*\cG_k$
into $\cT_{1k}$ and $\cG_2*\cG_k$ into $\cT_{2k}$ for all $k\in K$.

Next we show that $F$ carries squares in $\cG_k*\cG_l$ into $\cT_{kl}$ for
all $k,l\in K$, $k\neq l$. Consider such a square and let $x\in \Crt^n$ be
its projection. For $x$ in the image of $\varsigma^n$, this follows from the
assumption that $G$ carries $\bar\cG_k*\bar\cG_l$ to $\cT_{kl}$. For the
general case, we proceed by descending induction on $\pi_1(x)$. If
$\pi(x)=(n,q)$, then $x=\varsigma^n(n,q)$ is in the image of $\varsigma^n$.
For $\pi(x)=(p,q)$ with $p<n$, we consider the right square of
\eqref{5eq:squares}. We conclude by Condition (4) and the induction
hypothesis applied to $x\vee \varsigma^n(p+1,q)$.

Finally we show that $F$ carries $\cG_1*\cG_2$ into $\cT^i_{12}$ and carries
$\cG_1*^\cart \cG_2$ into $\cT^{i-1}_{12}$. Every square in
$\cF_1*_{\Cart^n} \cF_2$ of the form \eqref{5eq:square} has a canonical
decomposition
\[\xymatrix{w\ar[rd]\\
&y\wedge z\ar[r]\ar[d] & y\ar[d]\\&z \ar[r] & y\vee z\ar[rd]\\
&&& x,}
\]
where the vertical (resp.\ horizontal) arrows are in $\cF_1$ (resp.\
$\cF_2$), and oblique arrows are in $\cF_1\cap \cF_2$. Note that
$\cF^\cart_{12}$ is the set of squares such that $w=y\wedge z$. Multiplying
by $\id_u$ and applying $F$, we obtain a similar diagram where the inner
square is a pullback by Lemma \ref{5le:U7} and the oblique arrows are in
$\cE^i$ by the previous claim. Since we have already proved that $F$ carries
$\cG_1*\cG_2$ into $\cT_1*_\cC\cT_2$, all we need to show is that the induced
morphism $F(y\wedge z,u)\to F(y,u)\times_{F(x,u)}F(z,u)$ belongs to
$\cE^{i-1}$. However, by Lemma \ref{5le:pullback}, this morphism can be
identified with the left vertical arrow of the pullback square
\[\xymatrix{
F(y,u)\times_{F(x',u)}F(z,u) \ar[r]\ar[d]& F(x',u) \ar[d] \\
F(y,u)\times_{F(x,u)}F(z,u) \ar[r]& F(x',u)\times_{F(x,u)}F(x',u), }
\]
where for brevity we have written $x'$ for $y\vee z$, and the right vertical
arrow is the diagonal of $F(x',u)\to F(x,u)$ and hence belongs to
$\cE^{i-1}$. The lower horizontal arrow is a composition of a pullback of a
morphism in $\cT_1$ by a morphism in $\cT_2$ and a pullback of a morphism in
$\cT_2$ by a morphisms in $\cT_1$. Since $\cE^{i-1}$ is stable under pullback
by $\cT_1\cup \cT_2$, the left vertical arrow belongs to $\cE^{i-1}$ as well.
\end{proof}

The functor $g^n$ in Construction \ref{5cs:YZ} is induced by the map
\begin{equation}\label{5eq:gn}
\delta_2^*\delta^{2+}_* (\Cart^n,\cF)\to \Cart^n,
\end{equation}
which carries a square in $\cF_1*_{\Cart^n} \cF_2$ to its diagonal. We now
construct a family of sections of this map.

\begin{construction}\label{5co:section}
Let $n\geq 0$ be an integer.
\begin{enumerate}
  \item For $x\le y$ in $\Crt^n$ and $(p,q)$ in $[n]\times [n]$, we
      define two elements of $\Crt^n_{x//y}$:
      \[\lambda^n_p(x,y)=(\varsigma^n(\pi^n_1(y)\vee p,0)\vee x)\wedge y,\qquad
      \mu^n_q(x,y)=(\varsigma^n(0,\pi^n_2(y)\vee q)\vee x)\wedge y.
      \]
      These formulas are increasing in $p$, $q$, $x$ and $y$. Moreover, we
      have the following properties:
      \begin{gather}
      \lambda^n_p(x,x)=\mu^n_q(x,x)=x,\label{5eq:lambdamu1}\\
      \pi^n(\lambda^n_p(x,y))=(\pi^n_1(y),\pi^n_2(x)),\quad \pi^n(\mu^n_q(x,y))=(\pi^n_1(x),\pi^n_2(y)).\label{5eq:lambdamu2}
      \end{gather}

  \item We construct a map
      \[\alpha^n\colon A^n\colonequals (\Delta^n\times \Delta^n)^\sharp \times (\Cart^n)^\flat\to (\delta^*_2\delta_*^{2+}(\Cart^n,\cF))^\flat\]
      as follows. For an $m$-simplex $\tau=(\tau_1,\tau_2,\tau_3)\colon
      \Delta^m\to \Delta^n\times \Delta^n\times \Cart^n$, we define
      $\alpha^n(\tau)$ to be the map $\Delta^m\times \Delta^m\to \Cart^n$
      carrying $(a,b)$ to $\lambda^n_{\tau_1(b)}(\tau_3(b),\tau_3(a))$ for
      $a\ge b$, and to $\mu^n_{\tau_2(a)}(\tau_3(a),\tau_3(b))$ for $a\le
      b$. By \eqref{5eq:lambdamu1}, the two definitions coincide for
      $a=b$. By \eqref{5eq:lambdamu2}, $\alpha^n(\tau)$ is an $m$-simplex
      of $\delta^*_2\delta_*^{2+}(\Cart^n,\cF)$. In particular, $\alpha^n$
      carries an edge $(p,q,x)\to (p',q',y)$ of $\Delta^n\times
      \Delta^n\times \Cart^n$ to the square
      \begin{equation}\label{5eq:lambdamusquare}
      \xymatrix{x\ar[r]\ar[d] & \mu^n_q(x,y)\ar[d]\\ \lambda^n_p(x,y)\ar[r] &y}
      \end{equation}
      in $\cF_1*_{\Cart^n}\cF_2$. By \eqref{5eq:lambdamu1}, $\alpha^n$
      carries marked edges of $A^n$ to degenerate edges. The composition
      \[A^n\xto{\alpha^n}(\delta^*_2\delta_*^{2+}(\Cart^n,\cF))^\flat \to (\Cart^n)^\flat,\]
      where the second map is \eqref{5eq:gn}, is the projection.
\end{enumerate}
\end{construction}

\begin{remark}
For an order-preserving map $d\colon[m]\to[n]$, we have identities
\[\Crt(d)(\lambda^m_p(x,y))=\lambda^n_{d(p)}(\Crt(d)(x),\Crt(d)(y)),
\qquad\Crt(d)(\mu^m_q(x,y))=\mu^n_{d(q)}(\Crt(d)(x),\Crt(d)(y)).\] Thus the
maps $\alpha^n$ for different $n$ are compatible with $\Cart(d)$ in the
obvious sense.
\end{remark}

Next we define a restriction of $\alpha^n$, taking values in
$\delta_2^*\delta^{2\square}_* (\Cart^n,\cF^\cart)$.

\begin{construction}\label{5cs:boxplus}
Let $n\geq 0$ be an integer.
\begin{enumerate}
  \item We define order-preserving maps
      \[\xi^n,\eta^n\colon [n]\times[n]\to \Crt^n\] by
      \[\xi^n(p,q)=\varsigma^n(p,0)\wedge \varsigma^n(0,q),\quad \eta^n(p,q)=\varsigma^n(p,n)\wedge\varsigma^n(n,q).\]
      We have $\xi^n(p,q)\le \varsigma^n(p,q)\le \eta^n(p,q)$. We define
      a sublattice of $\Crt^n$ by
      \[\Crt^n_{p,q}\colonequals\Crt^n_{\xi^n(p,q)//\eta^n(p,q)}\] and we put $\boxplus^n_{p,q}\colonequals\N(\Crt^n_{p,q})$. We put
      \[\boxplus^n\colonequals\bigcup_{0\le p, q\le n}\boxplus^n_{p,q}\subseteq \Cart^n.
      \] Note that $\eta^n$ induces an isomorphism of lattices
      $[n]\times [n]\simeq\Crt^n_{n,n}=\Crt^n_{\xi^n(n,n)/}$ via which
      $\pi^n\colon \Crt^n\to[n]\times [n]$ can be identified with the
      morphism of lattices $-\vee \xi^n(n,n)\colon
      \Crt^n\to\Crt^n_{n,n}$.

  \item We define a marked simplicial subset $B^n$ of $A^n$ by
      \[B^n=\bigcup_{x\le y} \N(I_{x,y})^\sharp \times (\Cart^n_{x//y})^\flat \subseteq (\Delta^n\times \Delta^n)^\sharp\times
      (\boxplus^n)^\flat\subseteq A^n.\]  Here $x$ and $y$ run over
      elements of $\Crt^n$ and $I_{x,y}\subseteq[n]\times[n]$ denote the
      full subcategory spanned by pairs $(p,q)$ satisfying
      \begin{equation}\label{5eq:Bn}
      \xi^n(p,q)\le x\le y\le \eta^n(p,q),
      \end{equation}
      or, equivalently, satisfying
      $\Cart^n_{x//y}\subseteq\boxplus^n_{p,q}$. We note that $\eta^n$ is
      a right adjoint of $\pi^n$: $y\le\eta^n(p,q)$ if and only if
      $\pi^n(y)\le (p,q)$.
\end{enumerate}
\end{construction}

We refer the reader to \eqref{5eq:Hasse} for graphic depictions of
$\Crt^n_{p,q}$ for some small values of $n$, $p$, $q$.

\begin{remark}
Let $d\colon [m]\to [n]$ be an order-preserving map. For $0\leq p,q\leq m$,
we have
\begin{align*}
\Crt(d)(\xi^m(p,q))&=\varsigma^n(d(p),d(0))\wedge
\varsigma^n(d(0),d(q))\ge \xi^n(d(p),d(q)),\\
\Crt(d)(\eta^m(p,q))&=\varsigma^n(d(p),d(m))\wedge
\varsigma^n(d(m),d(q))\le \eta^n(d(p),d(q)).
\end{align*}
Thus $\Crt(d)$ induces morphisms of lattices $\Crt^m_{p,q}\to
\Crt^n_{d(p),d(q)}$ and hence maps $\boxplus^m_{p,q}\to
\boxplus^n_{d(p),d(q)}$ and $B^m\to B^n$.
\end{remark}

\begin{lemma}
The map $\alpha^n$ induces a map
\[\beta^n\colon B^n\to(\delta^*_2\delta_*^{2\square}(\Cart^n,\cF^\cart))^\flat.\]
\end{lemma}

\begin{proof}
It suffices to show that for every $m$-simplex $\tau$ of the underlying
simplicial set of $B^n$, the diagram $\alpha^n(\tau)\colon \Delta^m\times
\Delta^m\to \Cart^n$ carries the square spanned by the vertices $(a,b)$,
$(a+1,b)$, $(a,b+1)$, $(a+1,b+1)$ to a pullback. For $a=b$, the assertion
amounts to saying that for every edge $(p,q,x)\le (p',q',y)$ of $B^n$, the
square \eqref{5eq:lambdamusquare} is a pullback. We have
\[\lambda^n_p(x,y)\wedge\mu^n_q(x,y)=(\xi^n(\pi^n(y)\vee (p,q))\vee x)\wedge
y,
\]
which equals $x$ by the assumption $\xi^n(p,q)\le x$. For $a>b$, the
assertion amounts to saying that for every $3$-simplex $(p,q,x)\le
(p',q',y)\le (p'',q'',z)\le (p''',q''',w)$ of $B^n$, the square
\[\xymatrix{\lambda^n_p(x,z)\ar[r]\ar[d] & \lambda^n_{p'}(y,z)\ar[d]\\
\lambda^n_p(x,w)\ar[r] &\lambda^n_{p'}(y,w)}
\]
is a pullback. This is clear since $\lambda^n_p(x,z)=(\varsigma^n(p,0)\vee
x)\wedge z$ by the assumption $\pi_1(z)\le p$ and similarly for the other
vertices of the squares. The case $a<b$ is similar, with $\lambda^n_p$
replaced by $\mu^n_q$.
\end{proof}

\begin{remark}
The proof shows in fact that $\alpha^n$ carries the the simplicial subset
$S\subseteq \Delta^n\times \Delta^n\times\Cart^n$ spanned by those edges
$(p,q,x)\to (p',q',y)$ satisfying \eqref{5eq:Bn} (with no restrictions on
$(p',q')$) into $\delta^*_2\delta_*^{2\square}(\Cart^n,\cF^\cart)$. Note
that $S$ is bigger than the underlying simplicial set of $B^n$ for $n\ge 1$.
\end{remark}

The proof of Theorem \ref{5th:cartesian_gluing} relies on the following
lemmas. The proof of Lemma \ref{5le:cart_inner} will be given in Lemma
\ref{6le:cart_inner}.

\begin{lemma}\label{5le:cart_inner}
The inclusion $\boxplus^n\subseteq\Cart^n$ is a categorical equivalence.
\end{lemma}

\begin{lemma}\label{5le:trivial_co}
The inclusion $B^n\subseteq A^n$ is a trivial cofibration in the category
$\Mset$ for the Cartesian model structure.
\end{lemma}

\begin{proof}
Choose an exhaustion of $\boxplus^n$ by a sequence of simplicial subsets
\[\emptyset=K^0\subseteq K^1\subseteq \dots \subseteq K^N=\boxplus^n\]
such that each $K^i$, $1\le i\le N$ is obtained from $K^{i-1}$ by adjoining a
single nondegenerate simplex $\sigma^i\colon \Delta^{l_i}\to K^i$. This
induces inclusions
\[B^n=L^0\subseteq L^1\subseteq \dots
\subseteq L^N=(\Delta^n\times \Delta^n)^{\sharp}\times(\boxplus^n)^{\flat},
\]
where $L^i=B^n\cup ((\Delta^n\times \Delta^n)^{\sharp}\times(K^i)^{\flat})$.
By Lemma \ref{5le:cart_inner}, $(\boxplus^n)^{\flat}\subseteq
(\Cart^n)^\flat$ is a trivial cofibration in $\Mset$, so that
$(\Delta^n\times \Delta^n)^{\sharp}\times(\boxplus^n)^{\flat}\subseteq A^n$
is a trivial cofibration in $\Mset$ by \cite{Lu1}*{Corollary 3.1.4.3}.
Therefore, it suffices to show that the inclusion $L^{i-1}\subseteq L^i$ is
a trivial cofibration in $\Mset$ for all $1\le i\le N$. However, this
inclusion is a pushout of the map
\[(\Delta^n\times \Delta^n)^\sharp \times (\partial \Delta^{l_i})^\flat\coprod_{\N(I_{x,y})^\sharp\times(\partial \Delta^{l_i})^\flat}
\N(I_{x,y})^{\sharp}\times
(\Delta^{l_i})^\flat\to (\Delta^n\times \Delta^n)^\sharp\times (\Delta^{l_i})^\flat,
\]
where $x=\sigma^i(0)$, $y=\sigma^i(l_i)$. By the assumption that $\sigma^i$
is a simplex of $\boxplus^n$, the partially ordered set $I_{x,y}$ is
nonempty, and admits an initial object $\pi^n(y)$. Thus the inclusion
$\N(I_{x,y})\subseteq \Delta^n\times \Delta^n$ is anodyne. It follows that
the inclusion $\N(I_{x,y})^\sharp\subseteq (\Delta^n\times \Delta^n)^\sharp$
is a trivial cofibration in $\Mset$ (by Remark \ref{3re:Quillen}), and so is
its smash product with $(\partial\Delta^{l_i})^\flat\subseteq
(\Delta^{l_i})^\flat$ by \cite{Lu1}*{Corollary 3.1.4.3}.
\end{proof}

\begin{proof}[Proof of Theorem \ref{5th:cartesian_gluing}]
We adopt the notation of Lemma \ref{5le:restriction}. By the first part of
Condition (2), we have $\cE=\bigcup_{i\ge -2} \cE^i$, $\cT'=\bigcup_{i\ge
-2}\cT^i$, and
\[W_\infty\colonequals \delta^*_{\{1,2\}\amalg K,L}\delta_*^{(\{1,2\}\amalg
K)\square} (\cC,\cT')=\bigcup_{i\ge -2} W_i,
\]
where $W_i=\delta^*_{\{1,2\}\amalg K,L}\delta_*^{(\{1,2\}\amalg K)\square}
(\cC,\cT^i)$. Since $\cE^{-2}$ is the set of equivalences of $\cC$, we have
$\cT^{-2}=\cT$. Thus, the map $\iota$ in question is the transfinite
composition of inclusions
\[W_{-2}\to W_{-1}\to \dots \to W_i\to \dots \to W_\infty.\]
Since the Joyal model structure on $\Sset$ is combinatorial, the trivial
cofibrations form a weakly saturated class \cite{Lu1}*{Definition A.1.2.2}.
Thus it suffices to show that each inclusion $W_{-2}\to W_i$ is a categorical
equivalence for every integer $i\geq -1$. By Lemma
\ref{1le:categorical_equivalence} and induction, it suffices to show that for
every $i \ge -1$ and every commutative diagram
\[\xymatrix{W_{-2}\ar[r]^{f'} &W_{i-1}\ar[r]^-v\ar[d]_f & \Fun(\Delta^l,\cD)\ar[d]^p\\
&W_i\ar[r]^-w \ar@{..>}[ur]^-{u} & \Fun(\partial \Delta^l,\cD)}
\]
where $f$ and $f'$ are inclusions and $p$ is induced by the inclusion
$\partial \Delta^l\subseteq \Delta^l$, there exists a map $u\colon W_i\to
\Fun(\Delta^l,\cD)$ satisfying $p\circ u=w$ such that $u\circ f\circ f'$ and
$v\circ f'$ are homotopic over $\Fun(\partial \Delta^l,\cD)$. The proof is
mostly parallel to the proof of Theorem \ref{4th:multisimplicial_descent}.

Let $\sigma$ be an $n$-simplex of $W_i$, corresponding to a map $\tau\colon
\Delta^{n,n,n_k\res k\in K}_L\to \delta^{(\{1,2\}\amalg
K)\square}_*(\cC,\cT^i)$, where $n_k=n$. We consider the maps
\begin{align*}
w_* g(\tau)&\colon \Kart(\tau)\to \Fun(\delta_2^*\delta^{2+}_*(\Cart^n,\cF)\times \Delta^{[n_k]_{k\in K}},\Fun(\partial \Delta^l,\cD)),\\
v_* h(\tau)&\colon \Kart(\tau)\to \Fun(\delta_2^*\delta^{2\square}_*(\Cart^n,\cF^\cart)\times \Delta^{[n_k]_{k\in K}},\Fun(\Delta^l,\cD)),
\end{align*}
compositions of the maps $g(\tau)$ and $h(\tau)$ defined after the statement
of Lemma \ref{5le:restriction} and the maps induced by $w$ and $v$,
respectively. Since $\Kart(\tau)$ is a contractible Kan complex, the maps
$w_* g(\tau)$ and $v_*h(\tau)$ factorize through
\begin{align*}
w_* g(\tau)&\colon \Kart(\tau)\to \Map^\sharp((\partial \Delta^l)^\flat\times
(\delta_2^*\delta^{2+}_*(\Cart^n,\cF))^\flat\times (\Delta^{[n_k]_{k\in K}})^\flat,\cD^\natural),\\
v_* h(\tau)&\colon \Kart(\tau)\to \Map^\sharp((\Delta^l)^\flat\times
(\delta_2^*\delta^{2\square}_*(\Cart^n,\cF^\cart))^\flat\times(\Delta^{[n_k]_{k\in K}})^\flat,\cD^\natural),
\end{align*}
respectively. Composing with $\beta^n$ and $\alpha^n$, respectively, we
obtain maps
\begin{align*}
\psi(\tau)&\colon \Kart(\tau)\to\Map^\sharp((\partial \Delta^l)^\flat\times A^n\times (\Delta^{[n_k]_{k\in
K}})^\flat,\cD^\natural),\\
\phi(\tau)&\colon \Kart(\tau)\to\Map^\sharp((\Delta^l)^\flat\times B^n\times
(\Delta^{[n_k]_{k\in K}})^\flat,\cD^\natural).
\end{align*}

Consider the commutative diagram
\begin{align}\label{5eq:descent}
\xymatrix{\cN(\sigma)\ar[d]\ar[r]&\Kart(\tau)\ar[d]^-{h} \\
\Map^\sharp((\Delta^l)^\flat\times A^n\times(\Delta^{[n_k]_{k\in K}})^\flat,\cD^\natural)
\ar[r]^-{\RES_1}\ar[dd]^-{\RES_2} &\Map^\sharp(H\times(\Delta^{[n_k]_{k\in K}})^\flat,\cD^\natural)\ar[d]\\
&\Map^\sharp((\partial \Delta^l)^\flat\times A^n\times (\Delta^{[n_k]_{k\in K}})^\flat,\cD^\natural)\ar[d]^-{\RES_2} \\
\Map^\sharp((\Delta^l)^\flat\times (\Delta^n)^\flat,\cD^\natural)\ar[r]
&\Map^\sharp((\partial\Delta^l)^\flat\times(\Delta^n)^\flat,\cD^\natural).}
\end{align}
In the above diagram,
\begin{itemize}
  \item $\RES_1$ is induce by
     \[j\colon H=(\Delta^l)^\flat\times B^n\coprod_{(\partial
     \Delta^l)^\flat \times B^n } (\partial \Delta^l)^\flat\times A^n \hookrightarrow (\Delta^l)^\flat \times A^n;\]

  \item $h$ is the amalgamation of $\phi(\tau)$ and $\psi(\tau)$;

  \item $\cN(\sigma)$ is defined so that the upper square is a pullback
      square;

  \item the two maps $\RES_2$ are both induced by the composite embedding
      \[\Delta^n\xrightarrow{\r{diag}} \Delta^n\times \Delta^n\times \Delta^{n}\times \Delta^n\times
      \Delta^{[n_k]_{k\in K}} \xrightarrow{\id_{\Delta^n\times \Delta^n}\times \varsigma^n\times \id_{\Delta^{[n_k]_{k\in K}}}}
      \Delta^n\times \Delta^n\times \Cart^n\times \Delta^{[n_k]_{k\in K}};\]

  \item the unmarked arrows in the lower square are obvious restrictions.
\end{itemize}

By Lemma \ref{5le:trivial_co} and \cite{Lu1}*{Corollary 3.1.4.3}, the map
$j\times\id_{(\Delta^{[n_k]_{k\in K}})^\flat}$ is a trivial cofibration in
$\Mset$ and consequently $\RES_1$ is a trivial Kan fibration. Thus
$\cN(\sigma)$ is a contractible Kan complex.

We let $\Phi(\sigma)\colon \cN(\sigma)\to \Map^\sharp((\Delta^n)^\flat,
\Fun(\Delta^l,\cD)^\natural)$ denote the composition of the vertical arrows
in the first column of \eqref{5eq:descent}. This construction is functorial
in $\sigma$, giving rise to a morphism $\Phi\colon \cN\to
\Map[W_i,\Fun(\Delta^l,\cD)]$ in $(\Sset)^{(\del_{/W_i})^{op}}$. The
composition of the vertical arrows in the second column of
\eqref{5eq:descent} is constant of value $w(\sigma)$. Thus $\Map[W_i,p]\circ
\Phi$ factors through the morphism $\Delta^0_{(\del_{/W_i})^{op}}$
corresponding to $w$ via Remark \ref{2re:functors}.

Let $\sigma'$ be an $n$-simplex of $W_{-2}$ corresponding to a map
$\tau'\colon\Delta^{n,n,n_k\res k\in K}_L\to \delta^{(\{1,2\}\amalg
K)\square}_*(\cC,\cT^{-2})$. The composition
\[\Cart^n\times\Delta^{[n_k]_{k\in K}}_L\xrightarrow{\pi^n\times\id}
\Delta^n\times\Delta^n\times\Delta^{[n_k]_{k\in
K}}_L\xrightarrow{\tau'}\cC
\]
is a vertex of $\Kart(\tau')$ by Lemma \ref{5le:kan} and the equality
$\cT^{-2}_{12}=\cT_1*_\cC^\cart \cT_2$. This vertex, together with the
composition
\[\Delta^n\times \Delta^n\times \Cart^n\times \Delta^{[n_k]_{k\in K}}\to \Cart^n\times
\Delta^{[n_k]_{k\in K}} \xto{\pi^n\times \id}
\Delta^n\times \Delta^n\times\Delta^{[n_k]_{k\in K}}\to \Fun(\Delta^l,\cD),
\]
where the first map is the projection and the last map corresponds to the
composition
\[\Delta^{n,n,n_k\res k\in K}\xto{\tau'} \op^{\{1,2\}\amalg
K}_L\delta^{(\{1,2\}\amalg
K)\square}_*(\cC,\cT^{-2})\xto{v\circ f'}\delta_*^{\{1,2\}\amalg
K}\Fun(\Delta^l,\cD),
\]
provides a vertex of $\cN(f(f'(\sigma')))$, whose image under
$\Phi(f(f'(\sigma')))$ is $v(f'(\sigma'))$. This construction is functorial
in $\sigma'$, giving rise to $\nu\in \Gamma((f\circ f')^* \cN)_0$ satisfying
$(f\circ f')^*\Phi\circ \nu=v\circ f'$.  Applying Proposition
\ref{1pr:extension} to $\Phi$, the map $f\circ f'$, and the global section
$\nu$, we obtain a map $u\colon W_i\to \Fun(\Delta^l,\cD)$ satisfying
$p\circ u=w$ such that $u\circ f\circ f'$ and $v\circ f'$ are homotopic over
$\Fun(\partial \Delta^l,\cD)$, as desired.
\end{proof}

\begin{remark}
As a special case of Theorem \ref{5th:cartesian_gluing}, the inclusion
\[\delta^*_2\delta_*^{2\square}(\Cart^n,\cF^\cart)\subseteq \delta^*_2\delta_*^{2+}(\Cart^n,\cF)\]
is a categorical equivalence. If we have a direct proof of this special
case, Construction \ref{5co:section} through Lemma \ref{5le:trivial_co} are
not necessary and the proof of Theorem \ref{5th:cartesian_gluing} can be
achieved with $\psi(\tau)$ and $\phi(\tau)$ replaced by $g(\tau)$ and
$h(\tau)$, respectively.
\end{remark}

\section{Some trivial cofibrations}
\label{6ss}

In this section, we prove that certain inclusions of simplicial sets defined
in combinatorial manners are inner anodyne or categorical equivalences. In
particular, they are trivial cofibrations in $\Sset$ for the Joyal model
structure \cite{Lu1}*{Theorem 2.2.5.1}. Results of this section are used in
Sections \ref{4ss} and \ref{5ss}.

We let $\star$  denote \emph{joins} of categories and simplicial sets
\cite{Lu1}*{Section 1.2.8}.

\begin{lemma}\label{6le:join}
Let $A_0\subseteq A$, $B_0\subseteq B$, $C_0\subseteq C$ be inclusions of
simplicial sets. If $A_0\subseteq A$ is right anodyne and $C_0\subseteq C$ is
left anodyne \cite{Lu1}*{Definition 2.0.0.3}, then the induced inclusion
\[A\star B_0 \star C\coprod_{A_0\star B_0\star C_0} A_0\star B\star C_0 \subseteq A\star B\star C\]
is inner anodyne.
\end{lemma}

\begin{proof}
Consider the commutative diagram of inclusions with pushout squares
\[\xymatrix{A_0\star B_0\star C_0\ar[r]\ar[d] & A_0\star B\star C_0\ar[rd]\ar[d]\\
A\star B_0\star C_0\ar[r]\ar[d] &A\star B_0\star C_0 \coprod_{A_0\star B_0 \star C_0} A_0\star B\star C_0\ar[r]^-f\ar[d]
& A\star B\star C_0\ar[d]\ar[rd]\\
A\star B_0\star C\ar[r] & A\star B_0 \star C\coprod_{A_0\star B_0\star C_0} A_0\star B\star C_0\ar[r]^-{f'}
& A\star B_0\star C\coprod_{A\star B_0\star C_0} A\star B\star C_0\ar[r]^-g
& A\star B\star C.}
\]
By \cite{Lu1}*{Lemma 2.1.2.3}, $f$ is inner anodyne since $A_0\subseteq A$
is right anodyne; $g$ is inner anodyne since $C_0\subseteq C$ is left
anodyne. It follows that $g\circ f'$ is inner anodyne.
\end{proof}

\begin{lemma}\label{6le:cut}
Let $S$ be a partially ordered set and let $Q=[2] \subseteq S$,
$R=S-\{1\}\subseteq S$ be full inclusions. Assume that $0$ is a final object
of $R_{/1}$ and $2$ is an initial object of $R_{1/}$. Then the inclusion
$\N(Q)\cup\N(R)\subseteq \N(S)$ is inner anodyne.
\end{lemma}

\begin{proof}
Consider the commutative diagram of inclusions
\[\xymatrix{\N(Q\cap R)\ar[r]\ar[d] & \N(Q)\ar[d]\ar[rd]\\
\N(R_{/1}\star R_{1/})\ar[r]\ar[d]&\N(R_{/1}\star R_{1/})\coprod_{\N(\{0\}\star
\{2\})}
\N([2])\ar[r]^-f\ar[d]&
\N(R_{/1}\star \{1\}\star R_{1/})\ar[d]\\
\N(R)\ar[r]&\N(R)\cup \N(Q)\ar[r]^g& \N(S)}
\]
in which the square on the left are clearly pushouts. Note that for any
simplex $\sigma$ of $N(S)$, if $\sigma$ is not a simplex of $\N(R)$, then
$1$ is a vertex of $\sigma$, so that $\sigma$ is a simplex of
$\N(R_{/1}\star \{1\}\star R_{1/})$. Thus the lower outer square is a
pushout. It follows that $g$ is a pushout of $f$. By assumption and
\cite{Lu1}*{Lemma 4.2.3.6}, $\N(\{0\})\subseteq \N(R_{/1})$ is right anodyne
and $\N(\{2\})\subseteq \N(R_{1/})$ is left anodyne. It follows that $f$ is
inner anodyne by Lemma \ref{6le:join}. Therefore, $g$ is inner anodyne.
\end{proof}

\begin{remark}\label{6re:order}
Let $P\subseteq Q$ and $P\subseteq R$ be full inclusions of partially
ordered sets. The pushout $S=Q\coprod_P R$ in the category of partially
ordered sets admits the following description. The underlying set of $S$ is
the set-theoretic pushout. The partial order on $S$ is uniquely
characterized by the following properties:
\begin{enumerate}
  \item $Q\subseteq S$ and $R\subseteq S$ are full inclusions; and

  \item for $q\in Q$, $r\in R$, we have $q\le r$ (resp.\ $q\ge r$) if and
      only if there exists $p\in P$ satisfying $q\le p\le r$ (resp.\
      $q\ge p\ge r$).
\end{enumerate}
\end{remark}

\begin{lemma}\label{6le:union}
Let $P\subseteq Q$ and $P\subseteq R$ be full inclusions of partially
ordered sets and $S=Q\coprod_P R$ the pushout in the category of partially
ordered sets. Suppose that the following conditions are satisfied:
\begin{enumerate}
  \item $Q$ admits pushouts and pushouts are preserved by the inclusion
      $Q\subseteq S$.

  \item $Q-P$ is finite.

  \item $P$ is an up-set of $Q$, that is, a subset such that $p\in P$ and
      $q\geq p$ with $q\in Q$ imply $q\in P$.
\end{enumerate}
Then the inclusion $\N(Q)\cup \N(R)\subseteq \N(S)$ is inner anodyne.
\end{lemma}

\begin{proof}
We proceed by induction on $n=\# (Q-P)$. For $n=0$, we have $R=S$ and the
assertion is trivial. For $n=1$, put $Q-P=\{q\}$. Then Condition (3) means
that $q$ is a minimal element of $Q$, hence of $S$. Note that
$\N(R)\cup\N(S_{q/})=\N(S)$. Indeed, for any simplex $\sigma$ of $\N(S)$, if
$\sigma$ is a simplex of $\N(R)$, then $q$ is a vertex of $\sigma$, so that
$\sigma$ is a simplex of $\N(S_{q/})$. Thus the inclusion $\N(Q)\cup
\N(R)\subseteq \N(S)$ is a pushout of the inclusion $\N(Q_{q/})\cup
\N(R_{q/})\subseteq \N(S_{q/})$. The latter is isomorphic to the inclusion
\begin{equation}\label{7eq:preinner}
\N(P_{q/})^{\triangleleft}\coprod_{\N(P_{q/})}\N(R_{q/})\subseteq\N(R_{q/})^{\triangleleft}.
\end{equation}
By Condition (1), for every $r\in R_{q/}$, the partially ordered set
$P_{q//r}$ is filtered. Indeed, for $p,p'\in P_{q//r}$, the pushout $p\vee_q
p'$ is a common upper bound in $P_{q//r}$. Thus $\N(P_{q//r})$ is weakly
contractible by \cite{Lu1}*{Theorem 5.3.1.13, Lemma 5.3.1.18}. It follows
that $\N(P_{q/})^{op}\subseteq \N(R_{q/})^{op}$ is cofinal by
\cite{Lu1}*{Theorem 4.1.3.1}, thus right anodyne by \cite{Lu1}*{Proposition
4.1.1.3 (4)}. Therefore, \eqref{7eq:preinner} is inner anodyne by
\cite{Lu1}*{Lemma 2.1.2.3}.

For $n\ge 2$, we choose a minimal element $q$ of $Q-P$. Then Condition (3)
implies that $q$ is a minimal element of $Q$, hence of $S$. Put $S'=
S-\{q\}\supseteq R$ and $Q'=Q-\{q\}=S'\cap Q$. Consider the diagram of
inclusions with pushout square
\[\xymatrix{\N(Q')\cup \N(R)\ar[r]^-f\ar[d] & \N(S')\ar[d]\\
\N(Q)\cup \N(R)\ar[r]& \N(Q)\cup \N(S')\ar[r]^-g & \N(S).}
\]

By the induction hypothesis applied to the inclusions $P\subseteq Q'$ and
$P\subseteq R$, we know that $f$ is inner anodyne. Indeed, we have $P=Q'\cap
R$ and $S'$ is the pushout $Q'\coprod_P R$ in the category of partially
ordered set, by the description in Remark \ref{6re:order}. Condition (1)
holds since $q$ is a minimal element of $Q$, the partially ordered set $Q'$
admits pushouts and pushouts are stable under the inclusion $Q'\subseteq Q$,
hence under the inclusions $Q'\subseteq S$ and $Q'\subseteq S'$; for
Condition (2), we have $\#(Q'-P)=n-1$; and for Condition (3), $P$ is an
up-set of $Q$, hence of $Q'$.

By the induction hypothesis applied to the inclusions $Q'\subseteq Q$ and
$Q'\subseteq S'$, we know that $g$ is inner anodyne as well. Indeed, we have
$Q'=Q\cap S'$ and $S$ is the pushout $Q\coprod_{Q'} S$ in the category of
partially ordered sets; Condition (1) remains unchanged; for Condition (2),
we have $\#(Q-Q')=1$; and for Condition (3), $Q'$ is an up-set of $Q$ since
$q$ is minimal.

Therefore, the inclusion $\N(Q)\cup \N(R)\subseteq \N(S)$ is inner anodyne.
\end{proof}

\begin{lemma}\label{6le:lattice}
Let $P$ be a finite partially ordered set admitting pushouts and let
$p_0\le\dots\le p_s$; $q_0\le\dots\le q_s$ be elements of $P$ such that
$p_i\le q_{i-1}$ for $1\le i\le s$. Then the inclusion
\begin{align*}
\bigcup_{i=0}^s\N(P_{p_i//q_i})
\subseteq\N\left(\bigcup_{i=0}^{s}P_{p_i//q_i}\right)
\end{align*}
is inner anodyne.
\end{lemma}

\begin{proof}
Put $P_i=P_{p_i//q_i}$. The inclusion can be decomposed as $Q_0\subseteq
\dots \subseteq Q_n$, where
\[Q_j=\N\left(\bigcup_{i=0}^jP_i\right)\cup
\bigcup_{i=j+1}^{n}\N(P_i).
\]
For $1\le j\le n$, the inclusion $Q_{j-1}\subseteq Q_j$ is a pushout of
\begin{equation}\label{6eq:lattice}
\N\left(\bigcup_{i=0}^{j-1}P_i\right)\cup \N(P_j)\subseteq \N\left(\bigcup_{i=0}^jP_i\right).
\end{equation}
Indeed, for $k>j$, we have $P_k\cap \left(\bigcup_{i=0}^j P_i\right)\subseteq
P_j$. It then suffices to check that \eqref{6eq:lattice} satisfies the
assumptions of Lemma \ref{6le:union}. We denote coproducts in $P_{p_0/}$ by
$\vee$. Take $x\in A=\bigcup_{i=0}^{j-1}P_i$ and $y\in P_j$. If $x\ge y$,
then $x,y\in P_{p_j//q_{j-1}}=P_{j-1}\cap P_j$. If $x\le y$, then $x\le x\vee
p_j\le y$, where $x\vee p_j\in P_{j-1}\cap P_j$ by the assumption that
$p_j\le q_{j-1}$. Thus $\bigcup_{i=0}^jP_i$ is the pushout $A\coprod_{A\cap
P_j} P_j$ in the category of partially ordered sets, by Remark
\ref{6re:order}. Condition (1) of Lemma \ref{6le:union} follows from the fact
that for $x\in P_i$, $y\in P_{i'}$, we have $x\vee y\in P_{\max\{i,i'\}}$.
Condition (2) is clear. For Condition (3), it suffices to note that $A\cap
P_j=A_{p_j/}$.
\end{proof}

By an \emph{interval sublattice} of a finite lattice $P$, we mean a subset
of the form $P_{p//q}$, where $p\le q$ are in $P$.

\begin{lemma}\label{6le:equivalence}
Let $P$ be a finite lattice and let $p_0\le\dots\le p_s\le q_0\le\dots\le
q_s$ be elements of $P$ satisfying $\bigcup_{i=0}^s P_{p_i//q_i}=P$. Let
$Q_1,\dots,Q_t$ be interval sublattices of $P$. Then the inclusion
\[\bigcup_{i=0}^s\N(P_{p_i//q_i})\cup \bigcup_{j=1}^t \N(Q_j)\subseteq \N(P)\]
is a categorical equivalence.
\end{lemma}

Note that the assumptions imply that $p_0$ is the minimum of $P$ and $q_s$
is the maximum of $P$.

\begin{proof}
We proceed by induction on $t$.  Put $P_i=P_{p_i//q_i}$ and
$R_j=\bigcup_{i=0}^s\N(P_i)\cup \bigcup_{k=1}^j \N(Q_k)$. We need to show
that $R_t\subseteq N(P)$ is a categorical equivalence. By Lemma
\ref{6le:lattice}, the inclusion $R_0=\bigcup_{i=0}^s \N(P_i)\subseteq
\N(P)$ is inner anodyne, thus a categorical equivalence \cite{Lu1}*{Lemma
2.2.5.2}. Thus for $t=0$ we are done. For $t\ge 1$, it suffices to show that
the inclusions $R_0\subseteq \dots \subseteq R_t$ are categorical
equivalences. For $1\le j\le t$, the inclusion $R_{j-1}\subseteq R_j$ is a
pushout of
\begin{equation}\label{8eq:cateq}
\bigcup_{i=0}^s\N(P_i\cap Q_j)\cup \bigcup_{k=1}^{j-1}\N(Q_k\cap Q_j)\subseteq \N(Q_j)
\end{equation}
by an inclusion. By \cite{Lu1}*{Lemma A.2.4.3}, it suffices to show that
\eqref{8eq:cateq} is a categorical equivalence, which follows from the
induction hypothesis. In fact, if we write $Q_j=P_{p//q}$, then $P_i\cap
Q_j=P_{p_i\vee p//q_i\wedge q}$, and for $0\le i,i'\le s$ such that $P_i\cap
Q_j\neq \emptyset$, $P_{i'}\cap Q_j\neq \emptyset$, we have $p_i\vee p\le
q_{i'}\wedge q$.
\end{proof}

Now we prove Lemmas \ref{4le:cpt_inner} and \ref{5le:cart_inner}.

\begin{lemma}\label{6le:cpt_inner}
The inclusion $\Box^n\subseteq\CCpt^n$ is inner anodyne.
\end{lemma}

\begin{proof}
We apply Lemma \ref{6le:lattice} to the lattice $P=[n]\times [n]$, with
$s=n$, $p_i=(0,i)$ and $q_i=(i,n)$. We have $p_0\le \dots \le p_n=q_0\le
\dots \le q_n$. Thus
\begin{align*}
\Box^n&=\bigcup_{i=0}^n\N\left(\RCpt^n_{(0,i)//(i,n)}\right)
\subseteq\N\left(\bigcup_{i=0}^n\RCpt^n_{(0,i)//(i,n)}\right)=\N(\RCpt^n)=\CCpt^n
\end{align*}
is inner anodyne.
\end{proof}

\begin{lemma}\label{6le:cart_inner}
The inclusion $\bigcup_{0\le p\le n} \boxplus^n_{p,n}\subseteq \Cart^n$ is
inner anodyne and the inclusion $\boxplus^n \subseteq \Cart^n$ is a
categorical equivalence.
\end{lemma}

\begin{proof}
We apply Lemmas \ref{6le:lattice} and \ref{6le:equivalence} to the lattice
$P=\Crt^n$, with $s=n$, $p_i=\xi^n(i,n)$, $q_i=\eta^n(i,n)$, and the $Q_j$'s
given by $\Crt^n_{p,q}$ with $0\le p\le n$ and $0\le q<n$. We have
$\xi^n(0,n)\le \dots \le\xi^n(n,n)\le \eta^n(0,n)\le\dots\le \eta^n(n,n)$.
It remains to show $\Crt^n=\bigcup_{p=0}^n \Crt^n_{p,n}$. For $Q\in \Crt^n$,
we let $p$ denote $\pi^n_1(Q)=\min_{(p',q')\in Q} p'$. Then we have
\[\xi^n(p,n)\le \varsigma^n(p,0)\le Q\le \varsigma^n(p,n)=\eta^n(p,n),\]
so that $Q\in \Crt^n_{p,n}$.
\end{proof}

\end{document}